\documentclass[a4paper,11pt]{amsart}
\renewcommand{\uppercasenonmath}[1]

\allowdisplaybreaks
\usepackage[T1]{fontenc}
\usepackage[utf8]{inputenc}

\usepackage{amsmath}
\usepackage{amsthm}
\usepackage{amssymb}
\usepackage{amstext}
\usepackage{mathtools}

\usepackage{mathrsfs}
\usepackage{bm}
\usepackage{shuffle}
\usepackage[all]{xy}
\usepackage{ytableau}
\ytableausetup{boxsize=normal,aligntableaux=center}

\usepackage{xcolor} 
\usepackage{graphicx}
\usepackage{tikz}

\usepackage[shortalphabetic]{amsrefs}
\usepackage{ascmac}
\usepackage{comment}
\usepackage{here}
\usepackage{time}
\usepackage{fullpage} 

\usepackage{hyperref}
\hypersetup{
  colorlinks=true,
  linkcolor=blue,
  citecolor=blue,
  filecolor=blue,
  urlcolor=blue
}

\usepackage[capitalize,nameinlink,noabbrev,nosort]{cleveref}

\DeclareFontEncoding{OT2}{}{}
\DeclareFontSubstitution{OT2}{cmr}{m}{l}
\DeclareFontFamily{OT2}{cmr}{\hyphenchar\font45 }
\DeclareFontShape{OT2}{cmr}{m}{l}{%
  <5><6><7><8><9>gen*wncyr%
  <10><10.95><12><14.4><17.28><20.74><24.88>wncyr10}{}
\DeclareMathAlphabet{\mathcyr}{OT2}{cmr}{m}{l}
\DeclareMathAlphabet{\mathcyb}{OT2}{cmr}{b}{l}
\SetMathAlphabet{\mathcyr}{bold}{OT2}{cmr}{b}{l}

\makeatletter
\@namedef{subjclassname@2020}{%
  \textup{2020} Mathematics Subject Classification}
\makeatother

\newtheorem{theoremcounter}{Theorem Counter}[section]

\theoremstyle{definition}

\newtheorem{remark}[theoremcounter]{Remark}

\theoremstyle{plain}
\newtheorem{lemma}[theoremcounter]{Lemma}
\newtheorem{proposition}[theoremcounter]{Proposition}
\newtheorem{corollary}[theoremcounter]{Corollary}

\newtheorem{theorem}[theoremcounter]{Theorem}

\theoremstyle{definition}

\theoremstyle{plain}

\numberwithin{equation}{section}

\newcommand{\Z}{\mathbb{Z}}

\newcommand{\R}{\mathbb{R}}

\begin{document}

\title{Optimal Rates in Local and Central Limit Theorems for the Elephant Random Walk}

\author{Hideki Murahara}
\address[Hideki Murahara]{The University of Kitakyushu, 4-2-1 Kitagata, Kokuraminami-ku, Kitakyushu, Fukuoka, 802-8577, Japan}
\email{hmurahara@mathformula.page}

\author{Go Tokumitsu}
\address[Go Tokumitsu]{Department of Mathematics, Graduate School of Science, The University of Osaka, 1-1 Machikaneyama-cho, Toyonaka, Osaka, 560-0043, Japan}
\email{go.tokumitsu@gmail.com}

\date{}



\begin{abstract}
We determine the optimal rates of convergence in the local and central limit theorems for the one-dimensional elephant random walk in the diffusive and critical regimes, $0\le p\le3/4$. For every fixed initial distribution, we treat both exact finite-time centering and standardization and the normalization $S_n/\sqrt n$
(or $S_n/\sqrt{n\log n}$ at criticality). 
We establish uniform Edgeworth expansions for local probabilities and distribution functions, and use these expansions to derive the optimal convergence rates, which exhibit several phase transitions depending on the symmetry of the initial distribution, the normalization, and the type of approximation.
Our approach uses elephant polynomials for
precise characteristic-function expansions
and a random recursive tree representation
for global Fourier bounds.
\end{abstract}

\maketitle

\section{Introduction}\label{sec:intro}
We determine optimal rates of Gaussian approximation for the one-dimensional elephant random walk (ERW). 
Uniform Edgeworth expansions yield a classification according to the symmetry of the initial distribution, the normalization, and whether the approximation concerns local probabilities or distribution functions.
\subsection{Model and notation}
\mbox{Sch\"utz} and Trimper~%
\cite{SchutzTrimper2004} introduced the ERW, a process whose next increment is obtained by choosing a past increment uniformly and either repeating or reversing it.
Let $p,q\in[0,1]$. Let $X_1\in\{-1,1\}$ satisfy $\mathbb{P}(X_1=1)=q$ and $\mathbb{P}(X_1=-1)=1-q$. For $n\ge2$, let $U_n$ be uniform on $\{1,\ldots,n-1\}$ and let $\eta_n\in\{-1,1\}$ satisfy $\mathbb{P}(\eta_n=1)=p$ and $\mathbb{P}(\eta_n=-1)=1-p$. Assume that $X_1,U_2,U_3,\ldots,\eta_2,\eta_3,\ldots$ are mutually independent. Define $S_0=0$,
\[
X_n=\eta_n X_{U_n}\quad(n\ge2),\qquad
S_n=\sum_{j=1}^n X_j\quad(n\ge1).
\]
Then $(S_n)_{n\ge0}$ is the one-dimensional ERW with memory parameter $p$ and initial bias $q$. Its law and expectation are denoted by $\mathbb{P}_q$ and $\mathbb{E}_q$.

The long-time behavior changes sharply with $p$: the transition from diffusion to superdiffusion occurs at $p=3/4$~%
\cite{SchutzTrimper2004}. Write $\mu_n=\mathbb{E}_q[S_n]$ and $v_n^2=\operatorname{Var}_q(S_n)$,
where $\operatorname{Var}_q$ denotes the variance under $\mathbb{P}_q$. The exact mean is
\begin{align}\label{eq:m1_asym}
\mu_n=(2q-1)\prod_{j=1}^{n-1}\left(1+\frac{2p-1}{j}\right).
\end{align}
For $p>0$ and $q\ne1/2$, this gives $\mu_n\sim(2q-1)n^{2p-1}/\Gamma(2p)$. For $q=1/2$ the mean is zero, and for $p=0$ it is zero for every $n\ge2$. Moreover,
\begin{align}\label{vnasymp}
v_n^2\sim
\begin{cases}
\displaystyle \frac{n}{3-4p},&0\le p<3/4,\\[2mm]
n\log n,&p=3/4,\\[2mm]
\displaystyle\left\{\frac{1}{(4p-3)\Gamma(4p-2)}
-\frac{(2q-1)^2}{\Gamma(2p)^2}\right\}n^{4p-2},&3/4<p<1.
\end{cases}
\end{align}
Here $a_n\sim b_n$ means $a_n/b_n\to1$. At $p=1$, every increment equals $X_1$, so $S_n=nX_1$, $\mu_n=(2q-1)n$, and $v_n^2=\{1-(2q-1)^2\}n^2$, including the zero-variance cases $q=0,1$. These formulas were established by %
\cite{SchutzTrimper2004}. Thus a single memory parameter produces both anomalous diffusion and a phase transition.

Coletti, Gava, and Sch\"utz~%
\cite{ColettiGavaSchutz2017} proved central limit theorems with the normalizations $\sqrt n$ in the diffusive regime $0\le p<3/4$ and $\sqrt{n\log n}$ at the critical point $p=3/4$. They also showed that, in the superdiffusive regime $3/4<p<1$,
$S_n/n^{2p-1}$ converges almost surely to a nondegenerate
non-Gaussian random variable $L_{p,q}$. In the superdiffusive regime $3/4<p<1$, Kubota and Takei~\cite{KubotaTakei2019}
proved that $(S_n-n^{2p-1}L_{p,q})/\sqrt n$
converges in distribution to a centered Gaussian law
with variance $1/(4p-3)$. 

In this paper, we focus on the diffusive and critical regimes $0\le p\le3/4$. 
Berry--Esseen bounds were obtained by
Fan, Hu, and Ma~\cite{FanHuMa2021};
Dedecker, Fan, Hu, and Merlev\`ede~%
\cite{DedeckerFanHuMerlevede2023};
and Qin~%
\cite{Qin2025}.
These results do not, by themselves, determine the optimal rates for every normalization and initial distribution considered here. 
Local limit estimates also appear in Fan, Hu, and Ma~%
\cite{FanHuMa2021} and Huang, Fan, and Xiang~%
\cite{HuangFanXiang2026}. Appendix~\ref{app:local} discusses some issues with these earlier local limit estimates.

Set $\mathcal F_n=\sigma(X_1,\ldots,X_n)$. The definition gives
\begin{align}\label{eq:transition}
\mathbb{P}_q(S_{n+1}=S_n\pm1\mid\mathcal F_n)
=\frac12\pm\frac{(2p-1)S_n}{2n},\qquad n\ge1,
\end{align}
with matching signs. In particular, $(S_n)$ is a time-inhomogeneous Markov chain. When $p=q=1/2$, it is the simple symmetric random walk. When $q=1/2$, its law is symmetric: $(S_n)_{n\ge0}\overset{d}{=}(-S_n)_{n\ge0}$.

Let $\varphi(x)=e^{-x^2/2}/\sqrt{2\pi}$ and $\Phi(x)=\int_{-\infty}^x\varphi(y)\,dy$ be the standard Gaussian density and distribution function. The Hermite polynomials are
\[
H_j(x)=(-1)^j e^{x^2/2}\frac{d^j}{dx^j}e^{-x^2/2};
\qquad H_2(x)=x^2-1,\quad H_3(x)=x^3-3x,\quad H_4(x)=x^4-6x^2+3.
\]
Write $\kappa_r^{(q)}(n)$ for the $r$th cumulant of $S_n$ under $\mathbb{P}_q$.

Throughout, put $\alpha=2p-1$, $\beta=2q-1$, and $\mathcal{H}_n=\sum_{j=1}^n j^{-1}$. To describe the remainders, for $0<\alpha<1/2$, $k=1,2,3$, and $n\ge2$, define
\begin{align}\label{eq:G}
G_k(n)=
\begin{cases}
n,&2k\alpha<1,\\
n\log n,&2k\alpha=1,\\
n^{2k\alpha},&2k\alpha>1
\end{cases}
\end{align}
and
\begin{align}\label{eq:Rscale}
\mathcal R_n(\alpha,\beta)=
\begin{cases}
n^{-2},&-1\le\alpha\le0,\\
G_3(n)n^{-3}
+|\beta|n^{\alpha-5/2}G_2(n),
&0<\alpha<1/2,\\
\mathcal{H}_n^{-3}+|\beta|\mathcal{H}_n^{-5/2},
&\alpha=1/2.
\end{cases}
\end{align}
The symbols $O_p(\,\cdot\,)$ and $O_\alpha(\,\cdot\,)$ denote deterministic bounds with constants depending on the fixed parameter $p$ or $\alpha$, respectively; they do not denote stochastic order. Unless stated otherwise, these bounds are uniform in $q\in[0,1]$ and other relevant variables.
The notation $o_p(\,\cdot\,)$ is likewise deterministic, with the same uniformity conventions. A double subscript, as in $O_{p,q}$, indicates that the constant may also depend on $q$.

\subsection{Main results: Edgeworth expansions}
The first theorem gives a local expansion with third- and fourth-cumulant corrections throughout the diffusive regime and at criticality.
\begin{theorem}\label{thm:edgeworth}
Fix $0\le p\le3/4$. As $n\to\infty$, uniformly in $q\in[0,1]$ and integers $s\equiv n\pmod2$,
\begin{align}\label{eq:edgeworth}
\frac{v_n}{2}\mathbb{P}_q(S_n=s)
&=
\varphi\left(\frac{s-\mu_n}{v_n}\right)
\left\{
1
+
\frac{\kappa_3^{(q)}(n)}{6v_n^3}
H_3\left(\frac{s-\mu_n}{v_n}\right)
+
\frac{\kappa_4^{(q)}(n)}{24v_n^4}
H_4\left(\frac{s-\mu_n}{v_n}\right)
\right\}
\\
&\quad+
O_p\bigl(\mathcal R_n(\alpha,\beta)\bigr).
\notag\end{align}
In particular,
\begin{align*}
\sup_{q\in[0,1]}
\sup_{\substack{s\in\Z\\ s\equiv n\!\!\!\pmod2}}
\left|
\frac{v_n}{2}\mathbb{P}_q(S_n=s)
-
\varphi\left(\frac{s-\mu_n}{v_n}\right)
\right|
&\longrightarrow0.
\end{align*}
\end{theorem}

The following theorem will be used to establish the optimal Berry--Esseen bounds in Theorem \ref{thm:kolmogorov}, which in turn yield Corollary \ref{cor:usual-kolmogorov}.

\begin{theorem}
\label{lem:kolmogorov-edgeworth}
Fix $0\le p\le3/4$ and set $Z_n=(S_n-\mu_n)/v_n$ for sufficiently large $n$. Uniformly in $q\in[0,1]$,
\begin{align}
\sup_{x\in\R}
\Bigg|
\mathbb{P}_q(Z_n\le x)
-\Phi(x)+
\frac{\kappa_3^{(q)}(n)}{6v_n^3}
\varphi(x)H_2(x)+
\frac{\kappa_4^{(q)}(n)}{24v_n^4}
\varphi(x)H_3(x)
\Bigg|
=
O_p\left(
\frac1{v_n}
+
\mathcal R_n(\alpha,\beta)
\right).
\label{eq:kolmogorov-edgeworth}
\end{align}
\end{theorem}
In both expansions, products of the cumulant corrections, including the squared third-cumulant term, are absorbed into the remainder.

\subsection{Main results: Optimal rates of Gaussian approximation}
Tables~\ref{tab:local-rates} and~\ref{tab:kolmogorov-rates}
summarize the optimal rates of convergence to zero
of the local Gaussian approximation error and
the distribution-function error, respectively,
under the indicated normalizations, for fixed $p,q$.
The precise definitions of these errors are given below.
\begin{table}[H]
  \centering
  \caption{Optimal uniform local CLT rates (Theorem~\ref{thm:localrate} and Corollary~\ref{cor:usual-local}).
  }
  \label{tab:local-rates}
  \begingroup
  \renewcommand{\arraystretch}{1.3}
  \setlength{\tabcolsep}{5pt}
  \begin{tabular}{lcc}
    \hline
    Normalization & Symmetric: $q=1/2$ & Asymmetric: $q\ne1/2$ \\
    \hline
    $\displaystyle \frac{S_n-\mu_n}{v_n}$
    &
    $\begin{cases}
      n^{-1},          & 0\le p<5/8,\\
      n^{-1}\log n,    & p=5/8,\\
      n^{8p-6},        & 5/8<p<3/4,\\
      (\log n)^{-2},   & p=3/4
    \end{cases}$
    &
    $\begin{cases}
      n^{-1},          & 0\le p\le7/12,\\
      n^{6p-9/2},      & 7/12<p<3/4,\\
      (\log n)^{-3/2}, & p=3/4
    \end{cases}$
    \\[1.2em]
    $\begin{cases}
  S_n/\sqrt{n},       & p<3/4,\\
  S_n/\sqrt{n\log n}, & p=3/4
\end{cases}$
    &
    $\begin{cases}
      n^{-1},          & 0\le p\le1/2,\\
      n^{4p-3},        & 1/2<p<3/4,\\
      (\log n)^{-1},   & p=3/4
    \end{cases}$
    &
    $\begin{cases}
      n^{-1},          & 0\le p\le1/4,\\
      n^{2p-3/2},      & 1/4<p<3/4,\\
      (\log n)^{-1/2}, & p=3/4
    \end{cases}$
    \\[1em]
    \hline
  \end{tabular}
  \endgroup
\end{table}
The mechanisms behind these phase transitions
are explained in Remarks~\ref{rem:local-transitions}
and~\ref{rem:usual-local-transitions}.
\begin{table}[H]
  \centering
  \caption{Optimal CLT rates
  (Theorem~\ref{thm:kolmogorov} and Corollary~\ref{cor:usual-kolmogorov}).
  }
  \label{tab:kolmogorov-rates}
  \begingroup
  \renewcommand{\arraystretch}{1.3}
  \setlength{\tabcolsep}{5pt}
  \begin{tabular}{lcc}
    \hline
    Normalization & Symmetric: $q=1/2$ & Asymmetric: $q\ne1/2$ \\
    \hline
    $\displaystyle \frac{S_n-\mu_n}{v_n}$
    &
    $\begin{cases}
      n^{-1/2},        & 0\le p\le11/16,\\
      n^{8p-6},        & 11/16<p<3/4,\\
      (\log n)^{-2},   & p=3/4
    \end{cases}$
    &
    $\begin{cases}
      n^{-1/2},        & 0\le p\le2/3,\\
      n^{6p-9/2},      & 2/3<p<3/4,\\
      (\log n)^{-3/2}, & p=3/4
    \end{cases}$
    \\[1.2em]
    $\begin{cases}
  S_n/\sqrt{n/(3-4p)},       & p<3/4,\\
  S_n/\sqrt{n\log n}, & p=3/4
\end{cases}$
    &
    $\begin{cases}
      n^{-1/2},        & 0\le p\le5/8,\\
      n^{4p-3},        & 5/8<p<3/4,\\
      (\log n)^{-1},   & p=3/4
    \end{cases}$
    &
    $\begin{cases}
      n^{-1/2},        & 0\le p\le1/2,\\
      n^{2p-3/2},      & 1/2<p<3/4,\\
      (\log n)^{-1/2}, & p=3/4
    \end{cases}$
    \\[1em]
    \hline
  \end{tabular}
  \endgroup
\end{table}
The mechanisms behind these phase transitions
are explained in Remarks~\ref{rem:kolmogorov-transitions}
and~\ref{rem:usual-kolmogorov-transitions}.

\subsubsection{Optimal rates in the uniform local CLT}
We write $a_n\asymp b_n$ if $c b_n\le a_n\le C b_n$ for all sufficiently large $n$, with constants $0<c\le C<\infty$ that may depend on fixed $p,q$.
Define the local Gaussian approximation error under exact standardization by
\begin{align*}
E_n(p,q)
\coloneqq 
\sup_{\substack{s\in\Z\\ s\equiv n\pmod 2}}
\left|
\frac{v_n}{2}\mathbb{P}_q(S_n=s)
-
\varphi\left(\frac{s-\mu_n}{v_n}\right)
\right|
 .
\end{align*}

The following theorem gives the optimal rates of convergence in the uniform local central limit theorem under exact standardization.
\begin{theorem}\label{thm:localrate}
Fix $p\in[0,3/4]$. For $q=1/2$, as $n\to\infty$,
\[
E_n(p,1/2)
\asymp
\begin{cases}
n^{-1},        & 0\le p<5/8,\\
n^{-1}\log n,  & p=5/8,\\
n^{8p-6},      & 5/8<p<3/4,\\
(\log n)^{-2}, & p=3/4.
\end{cases}
\]
For each fixed $q\in[0,1]\setminus\{1/2\}$,
\[
E_n(p,q)
\asymp
\begin{cases}
n^{-1},         & 0\le p\le7/12,\\
n^{6p-9/2},      & 7/12<p<3/4,\\
(\log n)^{-3/2}, & p=3/4.
\end{cases}
\]
\end{theorem}

We next consider the local Gaussian approximation without exact centering and variance normalization. For $0\le p<3/4$, define
\begin{align*}
\widetilde E_n(p,q)
\coloneqq 
\sup_{\substack{s\in\Z\\ s\equiv n\pmod 2}}
\left|
\frac{\sqrt n}{2}\mathbb{P}_q(S_n=s)
-
\sqrt{3-4p}\,
\varphi\left(\sqrt{3-4p}\frac{s}{\sqrt n}\right)
\right|,
\end{align*}
and, for $p=3/4$, define
\begin{align*}
\widetilde E_n(3/4,q)
\coloneqq 
\sup_{\substack{s\in\Z\\ s\equiv n\pmod 2}}
\left|
\frac{\sqrt{n\log n}}{2}\mathbb{P}_q(S_n=s)
-
\varphi\left(\frac{s}{\sqrt{n\log n}}\right)
\right|
 .
\end{align*}

The following corollary gives the optimal rates under these normalizations.
\begin{corollary}\label{cor:usual-local}
Fix $p\in[0,3/4]$. For $q=1/2$, as $n\to\infty$,
\[
\widetilde E_n(p,1/2)
\asymp
\begin{cases}
n^{-1},       & 0\le p\le1/2,\\
n^{4p-3},     & 1/2<p<3/4,\\
(\log n)^{-1},& p=3/4.
\end{cases}
\]
For each fixed $q\in[0,1]\setminus\{1/2\}$,
\[
\widetilde E_n(p,q)
\asymp
\begin{cases}
n^{-1},         & 0\le p\le1/4,\\
n^{2p-3/2},     & 1/4<p<3/4,\\
(\log n)^{-1/2},& p=3/4.
\end{cases}
\]
\end{corollary}

\subsubsection{Optimal rates in the CLT}
Define the uniform distribution-function error under exact standardization by
\begin{align*}
D_n(p,q)
\coloneqq 
\sup_{x\in\R}
\left|
\mathbb{P}_q\left(\frac{S_n-\mu_n}{v_n}\le x\right)
-\Phi(x)
\right|
 .
\end{align*}
The following theorem gives its optimal rate.

\begin{theorem}\label{thm:kolmogorov}
Fix $p\in[0,3/4]$. For $q=1/2$, as $n\to\infty$,
\[
D_n(p,1/2)
\asymp
\begin{cases}
n^{-1/2},      & 0\le p\le11/16,\\
n^{8p-6},      & 11/16<p<3/4,\\
(\log n)^{-2}, & p=3/4.
\end{cases}
\]
For each fixed $q\in[0,1]\setminus\{1/2\}$,
\[
D_n(p,q)
\asymp
\begin{cases}
n^{-1/2},        & 0\le p\le2/3,\\
n^{6p-9/2},      & 2/3<p<3/4,\\
(\log n)^{-3/2}, & p=3/4.
\end{cases}
\]
\end{theorem}

For $0\le p<3/4$, define
\begin{align*}
\widetilde D_n(p,q)
\coloneqq 
\sup_{x\in\R}
\left|
\mathbb P_q\left(
\frac{S_n}{\sqrt{n/(3-4p)}}\le x
\right)
-\Phi(x)
\right|,
\end{align*}
and, for $p=3/4$, define
\begin{align*}
\widetilde D_n(3/4,q)
\coloneqq 
\sup_{x\in\R}
\left|
\mathbb{P}_q\left(\frac{S_n}{\sqrt{n\log n}}\le x\right)
-
\Phi(x)
\right|
 .
\end{align*}

The following corollary gives its optimal rate.
\begin{corollary}\label{cor:usual-kolmogorov}
Fix $p\in[0,3/4]$. For $q=1/2$, as $n\to\infty$,
\[
\widetilde D_n(p,1/2)
\asymp
\begin{cases}
n^{-1/2},      & 0\le p\le5/8,\\
n^{4p-3},      & 5/8<p<3/4,\\
(\log n)^{-1}, & p=3/4.
\end{cases}
\]
For each fixed $q\in[0,1]\setminus\{1/2\}$,
\[
\widetilde D_n(p,q)
\asymp
\begin{cases}
n^{-1/2},        & 0\le p\le1/2,\\
n^{2p-3/2},      & 1/2<p<3/4,\\
(\log n)^{-1/2}, & p=3/4.
\end{cases}
\]
\end{corollary}

\begin{remark}
For $q=1$, Qin~%
\cite{Qin2025}*{Proposition 2.4} gives Berry--Esseen upper bounds of the same orders under the normalization of Corollary~\ref{cor:usual-kolmogorov}. The corollary proves their optimality.
\end{remark}

\subsection{Related work and proof strategy}
\subsubsection{Gaussian approximation rates and phase transitions}
Fan, Hu, and Ma~\cite{FanHuMa2021} gave earlier
Berry--Esseen bounds under a martingale normalization. Dedecker, Fan, Hu, and Merlev\`ede~%
\cite{DedeckerFanHuMerlevede2023} subsequently obtained sharper Berry--Esseen bounds
for general initial distributions under a normalization natural
for the martingale representation of the ERW.
Matching lower bounds were not established in their work.

Classical distribution-function Edgeworth expansions for sums of independent random variables are given, for example, in Petrov~%
\cite{Petrov1975}*{Chapter VI, Theorem 7}. They cannot be applied directly to the dependent increments of the ERW.

Hayashi, Oshiro, and Takei~%
\cite{HayashiOshiroTakei2023} determined rates of convergence of the moments of $S_n/\sqrt n$ (or $S_n/\sqrt{n\log n}$ at criticality) to Gaussian moments. Their higher even-moment rates change at $p=1/2$.  Bertoin~%
\cite{Bertoin2022} also identified a transition at $p=1/4$: the expected first return time to the origin is finite if and only if $p<1/4$.

The earlier Berry--Esseen upper bounds change order at
$p=1/2$ in Qin's result and at $p=5/8$ in the result of
Dedecker et al.
As shown in Tables~\ref{tab:local-rates} and
\ref{tab:kolmogorov-rates}, our optimal rates under exact
standardization exhibit several phase transitions.
To the best of our knowledge, the values
$p=7/12,\ 2/3,\ 11/16$ have not previously been identified
as phase transition points in the standard one-dimensional ERW.

Hu~\cite{Hu2026} obtained Berry--Esseen bounds for positively and negatively step-reinforced random walks with general step distributions. For symmetric Rademacher steps, his Theorems 1.1 and 1.2 apply to the ERW with symmetric initial distribution in the ranges $1/2<p\le3/4$ and $0<p<1/2$, respectively. In the former range, translating his parameter into our notation and replacing his variance approximation by the exact variance yields
\[
D_n(p,1/2)=
\begin{cases}
O(n^{-1/2}),&1/2<p<2/3,\\
O(n^{-1/2}\log n),&p=2/3,\\
O(n^{6p-9/2}),&2/3<p<3/4,\\
O((\log n)^{-3/2}),&p=3/4.
\end{cases}
\]
Theorem~\ref{thm:kolmogorov} sharpens these estimates throughout $2/3\le p\le3/4$.

\subsubsection{Relation to local limit theorems}
Classical local Edgeworth expansions for sums of independent lattice variables,
such as Petrov~%
\cite[Chapter VII, Theorem 13]{Petrov1975}, do not apply directly
to the dependent increments of the ERW.
A local limit theorem in the superdiffusive regime $3/4<p<1$ was recently obtained in~%
\cite{TrinasValle2026}.
For $0< p\le3/4$, local estimates appear in Fan--Hu--Ma~%
\cite{FanHuMa2021}
and, with random step sizes, Huang--Fan--Xiang~%
\cite{HuangFanXiang2026}.
Appendix~\ref{app:local} points out a common issue in the proofs of these
local estimates.

Gu\'erin, Laulin, and Raschel~\cite{GuerinLaulinRaschel2025}
also studied exact distribution formulas for the ERW
at $p=0$ and $p=1/4$ with $q=1/2$.
The resulting asymptotics for the return probabilities
$\mathbb{P}_{1/2}(S_{2n}=0)$ are recovered as special cases
of Theorem~\ref{thm:edgeworth}.

To the best of our knowledge, a local CLT uniform in the initial
distribution and lattice position throughout $0\le p\le3/4$
was not previously available.
Our local Edgeworth expansion also yields the optimal convergence
rates for the local Gaussian approximation.

\subsubsection{Proof strategy}
For the simple symmetric random walk, the point probabilities
are given explicitly by binomial coefficients, and Stirling's
formula yields optimal local Gaussian estimates;
see Lawler and Limic~%
\cite[Section~2.5.1]{LawlerLimic2010} for this approach.
For the ERW, however, a comparably tractable closed-form
expression for the point probabilities is not generally available. Instead, we derive the Edgeworth expansions from precise
estimates of the characteristic function, using the elephant
polynomials studied by Gu\'erin, Laulin, and Raschel~%
\cite{GuerinLaulinRaschel2025}.
Esseen's smoothing inequality and Fourier inversion
then yield the distribution-function and local Edgeworth
expansions, respectively. We then use these expansions to determine the optimal rates of convergence in the local CLT and the CLT.

Limit theorems for the ERW have been established using
the P\'olya-urn approach of Baur and Bertoin~%
\cite{BaurBertoin2016} and the martingale approach of Bercu~%
\cite{Bercu2018}.
However, these approaches do not directly yield the main results of this paper.
\subsubsection{Potential application to the elephant random walk with stops}
The elephant random walk with stops, introduced by Kumar, Harbola, and Lindenberg~\cite{KumarHarbolaLindenberg2010} and subsequently studied by Bercu~\cite{Bercu2022}, admits a representation as an ordinary elephant random walk evaluated at an independent random clock. Using a random clock representation, Liang~\cite{Liang2026} obtained quantitative bounds for normal approximation in the more general setting allowing random step sizes. Combining the random clock representation with the uniform Edgeworth expansions established in the present paper is expected to yield a more complete description of the optimal rates of convergence in the central and local limit theorems for the elephant random walk with stops. We investigate these problems in forthcoming work.
\subsection{Organization of the paper}
Section~\ref{sec:finite} derives finite-time recurrences and exact moment identities.
Section~\ref{sec:roots} studies elephant polynomials and their roots, and Section~\ref{sec:fourier} establishes bounds on the full Fourier interval.
Section~\ref{sec:analytic} gives analytic logarithmic expansions and uniform remainders;
Section~\ref{sec:cumulants} computes the leading cumulants.
Section~\ref{sec:inversion} establishes the local Edgeworth expansion via Fourier inversion and derives the optimal rates of local Gaussian approximation under exact and asymptotic normalization.
Section~\ref{sec:kolmogorov} proves the distribution-function expansion and derives the optimal rates of convergence in the central limit theorem under both normalizations.

\section{Finite-time recurrences and exact moments}\label{sec:finite}
The conditional expectation satisfies
\begin{align*}
 \mathbb{E}_q[X_{n+1}\mid\mathcal F_n]=\alpha S_n/n
\end{align*}
For $j\ge1$, define
\[
h_j(n)=\prod_{k=1}^{n-1}\left(1+\frac{j\alpha}{k}\right),
\qquad h_n\coloneqq h_1(n),
\]
with the empty product equal to $1$. Reversing the initial sign reverses the entire path, so all even moments of $S_n$ are independent of $q$. Write
\[
\mathcal{H}_n^{(2)}=\sum_{k=1}^n\frac1{k^2},\qquad
\sigma_n^2=\mathbb{E}_q[S_n^2]=\mathbb{E}_{1/2}[S_n^2].
\]
The exact moment formulas of Sch\"utz and Trimper~\cite{SchutzTrimper2004}
can be written in our notation as follows:
\begin{align}\label{eq:meanvar}
\mu_n=\beta h_n,\qquad
\sigma_n^2=
\begin{cases}
\dfrac{n-2\alpha h_2(n)}{1-2\alpha},&\alpha<1/2,\\[4pt]
n\mathcal{H}_n,&\alpha=1/2,
\end{cases}
\qquad
v_n^2=\sigma_n^2-\beta^2h_n^2.
\end{align}

Let $M_n^{(q)}(z)=\mathbb{E}_qe^{zS_n}$ and $\chi_n^{(q)}(\theta)=M_n^{(q)}(\mathrm{i}\theta)$. Since $e^{zX}=\cosh z+X\sinh z$ for $X\in\{-1,1\}$, conditioning gives
\begin{align}\label{Mnzenka}
 M_{n+1}^{(q)}(z)&=\cosh z\,M_n^{(q)}(z)+\frac\alpha n\sinh z\,(M_n^{(q)})'(z),\\
 \notag\chi_{n+1}^{(q)}(\theta)&=\cos\theta\,\chi_n^{(q)}(\theta)
       +\frac\alpha n\sin\theta\,(\chi_n^{(q)})'(\theta).
\end{align}
The initial transform is $M_1^{(q)}(z)=\cosh z+\beta\sinh z$. Define polynomials by $F_1(x)=x$, $J_1(x)=1$, and
\begin{align}
 F_{n+1}(x)&=xF_n(x)-\frac\alpha n(1-x^2)F_n'(x),\label{eq:polyF}\\
 J_{n+1}(x)&=(1+\alpha/n)xJ_n(x)-\frac\alpha n(1-x^2)J_n'(x).\label{eq:polyJ}
\end{align}
The polynomials $F_n$ are precisely the elephant polynomials introduced by Gu\'erin, Laulin, and Raschel~%
\cite{GuerinLaulinRaschel2025}, while $J_n$ forms a second polynomial sequence used in the decomposition below.
Differentiating the induction hypothesis for $M_n^{(q)}(z)$
and substituting the result into~\eqref{Mnzenka},
we obtain the following decomposition by induction,
using the polynomial recurrences~\eqref{eq:polyF} and~\eqref{eq:polyJ}.
\begin{align}\label{eq:decomp}
 M_n^{(q)}(z)=F_n(\cosh z)+\beta\sinh z\,J_n(\cosh z).
\end{align}
\begin{lemma}\label{lem:basic-polynomials}
If $\alpha>-1$, then, for every $n\ge1$,
\[
\deg F_n=n,\qquad \deg J_n=n-1,
\]
and both leading coefficients equal $(1+\alpha)^{n-1}$. Moreover,
\begin{align}\label{guuki}
F_n(-x)=(-1)^nF_n(x),\qquad
J_n(-x)=(-1)^{n-1}J_n(x),
\end{align}
and $F_n(1)=1$, $J_n(1)=h_n$.
\end{lemma}
\begin{proof}
The assertions hold at $n=1$ because $F_1(x)=x$ and $J_1(x)=1$.
The recurrences first give the degree bounds $\deg F_n\le n$ and $\deg J_n\le n-1$ by induction. Write
\[
F_n(x)=c_nx^n+\text{lower-degree terms},\qquad
J_n(x)=d_nx^{n-1}+\text{lower-degree terms}.
\]
Comparing the coefficients of $x^{n+1}$ in \eqref{eq:polyF} and $x^n$ in \eqref{eq:polyJ} gives
\[
c_{n+1}=c_n+\frac\alpha n nc_n=(1+\alpha)c_n,
\]
\[
d_{n+1}=\left(1+\frac\alpha n\right)d_n
 +\frac\alpha n(n-1)d_n=(1+\alpha)d_n.
\]
Since $c_1=d_1=1$, both coefficients equal $(1+\alpha)^{n-1}$ and are nonzero when $\alpha>-1$. This proves the exact degree assertions.

For parity, the induction hypothesis $F_n(-x)=(-1)^nF_n(x)$ implies $F_n'(-x)=(-1)^{n+1}F_n'(x)$. 
Thus
\begin{align*}
F_{n+1}(-x)
&=(-x)F_n(-x)-\frac\alpha n(1-x^2)F_n'(-x)\\
&=(-1)^{n+1}\left\{xF_n(x)-\frac\alpha n(1-x^2)F_n'(x)\right\}
 =(-1)^{n+1}F_{n+1}(x).
\end{align*}
Similarly, $J_n'(-x)=(-1)^nJ_n'(x)$ and \eqref{eq:polyJ} imply $J_{n+1}(-x)=(-1)^nJ_{n+1}(x)$.
Finally, evaluation at $x=1$ yields
\[
F_{n+1}(1)=F_n(1),\qquad
J_{n+1}(1)=\left(1+\frac\alpha n\right)J_n(1).
\]
The initial values therefore give $F_n(1)=1$ and
$J_n(1)=\prod_{k=1}^{n-1}(1+\alpha/k)=h_n$. The proof is complete.
\end{proof}

As noted above, every even moment is independent of $q$, whereas every odd moment is $\beta$
times its value at $q=1$. For integers $r\ge1$, expanding
$(S_n+X_{n+1})^r$ and using \eqref{eq:transition} yields the moment recurrence
\begin{align*}
 \mathbb{E}_q[S_{n+1}^r]
 = \sum_{\substack{0\le k\le r\\ k\ \mathrm{even}}}
 \binom{r}{k}\mathbb{E}_q[S_n^{r-k}]
 + \frac{\alpha}{n}
 \sum_{\substack{1\le k\le r\\ k\ \mathrm{odd}}}
 \binom{r}{k}\mathbb{E}_q[S_n^{r-k+1}],
\end{align*}
which is equivalent to the moment recursion derived by
Bercu, Chabanol, and Ruch %
\cite[Section~4]{BercuChabanolRuch2019}.
In particular, we have
\begin{align}
 \mu_{n+1}
 &=
 \left(1+\frac{\alpha}{n}\right)\mu_n,
 \label{eq:m1-rec}\\
 \mathbb{E}_q[S_{n+1}^2]
 &=
 \left(1+\frac{2\alpha}{n}\right)\mathbb{E}_q[S_n^2]+1,
 \label{eq:m2-rec}\\
 \mathbb{E}_q[S_{n+1}^3]
 &=
 \left(1+\frac{3\alpha}{n}\right)\mathbb{E}_q[S_n^3]
 +\left(3+\frac{\alpha}{n}\right)\mu_n,
 \label{eq:m3-rec}\\
 \mathbb{E}_q[S_{n+1}^4]
 &=
 \left(1+\frac{4\alpha}{n}\right)\mathbb{E}_q[S_n^4]
 +\left(6+\frac{4\alpha}{n}\right)\mathbb{E}_q[S_n^2]+1.
 \label{eq:m4-rec}
\end{align}

\section{Polynomial roots and their power sums}\label{sec:roots}

This section investigates the location and interlacing
of the roots of $F_n$ and positive-residue decompositions
of $J_n/F_n$. The endpoint cases $\alpha=0$ and $\alpha=-1$ are treated separately.
\subsection{Root interlacing and positive-residue decompositions}
When $0<\alpha\le1/2$, the interlacing of the zeros of $F_n$ was established by
Gu\'erin--Laulin--Raschel~\cite[Proposition~1]{GuerinLaulinRaschel2025}.
The following lemma gives an alternative proof based on the
monotonicity of $F_{n+1}/F_n$ and further establishes a partial
fraction decomposition of $J_n/F_n$ with positive weights.
\begin{lemma}\label{pos:lem:residues}
Suppose $0<\alpha\le1/2$. All roots $\rho_{n,1},\ldots,\rho_{n,n}$ of $F_n$ are simple and lie in $(-1,1)$. There are strictly positive weights $w_{n,j}$ with $\sum_{j=1}^nw_{n,j}=1$ such that
\begin{align}\label{pos:eq:residues}
 \frac{J_n(x)}{F_n(x)}
 =\sum_{j=1}^n\frac{w_{n,j}}{x-\rho_{n,j}}.
\end{align}
In particular, $\sum_{j=1}^nw_{n,j}/(1-\rho_{n,j})=h_n$.
\end{lemma}
\begin{proof}
At $n=1$, $F_1(x)=x$ has a unique simple root
$\rho_{1,1}=0\in(-1,1)$. Moreover, since $J_1(x)=1$,
\[
\frac{J_1(x)}{F_1(x)}=\frac1{x-\rho_{1,1}},
\qquad w_{1,1}=1.
\]
Assume that $F_n$ has $n$ simple roots
$-1<\rho_{n,1}<\cdots<\rho_{n,n}<1$ and that
\eqref{pos:eq:residues} holds with $w_{n,j}>0$ and
$\sum_{j=1}^n w_{n,j}=1$.
Write $m(x)=J_n(x)/F_n(x)$. For $x$ outside the roots, define
\begin{align*}
D(x)=\frac{F_{n+1}(x)}{F_n(x)}
&=x-\frac\alpha n(1-x^2)\frac{F_n'(x)}{F_n(x)}\\
&=x-\frac\alpha n(1-x^2)\sum_{j=1}^n\frac1{x-\rho_{n,j}}.
\end{align*}
Here we used \eqref{eq:polyF} and $\deg F_n=n$. 
Differentiating yields
\[
D'(x)
=1+\frac\alpha n\sum_j
\frac{1+x^2-2x\rho_{n,j}}{(x-\rho_{n,j})^2}>0,
\]
because the numerator is $(x-\rho_{n,j})^2+1-\rho_{n,j}^2>0$.
Moreover,
\[
\lim_{x\uparrow\rho_{n,j}}D(x)=+\infty,\qquad
\lim_{x\downarrow\rho_{n,j}}D(x)=-\infty,
\qquad D(-1)=-1,\quad D(1)=1.
\]
Continuity and strict monotonicity give exactly one zero in each interval
\[
(-1,\rho_{n,1}),\quad
(\rho_{n,j},\rho_{n,j+1})\ (1\le j\le n-1),\quad
(\rho_{n,n},1).
\]
These are $n+1$ simple roots of $F_{n+1}$, none shared with $F_n$.

Let $r=\rho_{n+1,k}$ be one of these roots. Then $F_n(r)\ne0$ and $D(r)=0$. From \eqref{eq:polyJ},
\[
\frac{J_{n+1}(x)}{F_n(x)}
=m(x)D(x)+\frac\alpha n\{xm(x)-(1-x^2)m'(x)\}.
\]
Since
\[
m(x)=\sum_{j=1}^n\frac{w_{n,j}}{x-\rho_{n,j}},\qquad
m'(x)=-\sum_{j=1}^n\frac{w_{n,j}}{(x-\rho_{n,j})^2},
\]
we obtain
\[
\frac{J_{n+1}(r)}{F_n(r)}
=\frac\alpha n\sum_{j=1}^n
 w_{n,j}\frac{1-r\rho_{n,j}}{(r-\rho_{n,j})^2}.
\]
Taylor expansion at $r$ gives
\[
D(x)=D'(r)(x-r)+O((x-r)^2),\qquad
\frac{J_{n+1}(x)}{F_n(x)}
=\frac{J_{n+1}(r)}{F_n(r)}+O(x-r).
\]
Since $F_{n+1}(x)=F_n(x)D(x)$, the above Taylor expansions give
\[
\frac{J_{n+1}(x)}{F_{n+1}(x)}
=
\frac{J_{n+1}(r)}{F_n(r)D'(r)}
\frac1{x-r}+O(1).
\]
Let $w_{n+1,k}$ denote the coefficient of $(x-r)^{-1}$
in this expansion. Then
\[
w_{n+1,k}
=
\frac{J_{n+1}(r)}{F_n(r)D'(r)}
=
\frac\alpha{nD'(r)}\sum_{j=1}^n
 w_{n,j}\frac{1-r\rho_{n,j}}{(r-\rho_{n,j})^2}>0.
\]
Positivity follows from $\alpha>0$, $D'(r)>0$, and $r,\rho_{n,j}\in(-1,1)$.

By Lemma~\ref{lem:basic-polynomials}, $\deg J_{n+1}=n<\deg F_{n+1}=n+1$. Since all denominator roots are simple, partial fractions give
\[
\frac{J_{n+1}(x)}{F_{n+1}(x)}
=\sum_{k=1}^{n+1}\frac{w_{n+1,k}}{x-\rho_{n+1,k}}.
\]
The leading coefficients of numerator and denominator both equal $(1+\alpha)^n$ by Lemma~\ref{lem:basic-polynomials}, so the ratio is asymptotic to $1/x$ as $x\to\infty$. Comparing with the partial fractions gives $\sum_kw_{n+1,k}=1$.

Finally, evaluating \eqref{pos:eq:residues} at $x=1$ and using
the identities $F_n(1)=1$ and $J_n(1)=h_n$
established in Lemma~\ref{lem:basic-polynomials}
proves the last assertion. The proof is complete.
\end{proof}

When $-1<\alpha<0$, the roots of $F_n$ are simple and
purely imaginary, and those of $F_n$ and $F_{n+1}$
interlace along the imaginary axis, as established by
Gu\'erin, Laulin, and Raschel~\cite[Proposition~2]{GuerinLaulinRaschel2025}.
The following lemma further establishes a partial fraction
decomposition of $J_n/F_n$ with positive weights.
\begin{lemma}\label{neg:lem:rotated}
Suppose $-1<\alpha<0$. All roots of $F_n$ are simple and lie on the imaginary axis, say $\mathrm{i} t_{n,1},\ldots,\mathrm{i} t_{n,n}$ with $t_{n,j}\in\mathbb R$. There are strictly positive weights $w_{n,j}$ with $\sum_{j=1}^nw_{n,j}=1$ such that
\begin{align}\label{neg:eq:residues}
 \frac{J_n(x)}{F_n(x)}
 =\sum_{j=1}^n\frac{w_{n,j}}{x-\mathrm{i} t_{n,j}}.
\end{align}
\end{lemma}
\begin{proof}
Set $P_n(x)=\mathrm{i}^{-n}F_n(\mathrm{i} x)$ and
$Q_n(x)=\mathrm{i}^{-(n-1)}J_n(\mathrm{i} x)$.
The parity identities \eqref{guuki} restrict the powers in $F_n$ and $J_n$ to $n,n-2,\ldots$ and $n-1,n-3,\ldots$, respectively; hence $P_n$ and $Q_n$ have real coefficients.
The identities
\[
F_n(\mathrm{i} x)=\mathrm{i}^nP_n(x),\quad F_n'(\mathrm{i} x)=\mathrm{i}^{n-1}P_n'(x),\qquad
J_n(\mathrm{i} x)=\mathrm{i}^{n-1}Q_n(x),\quad J_n'(\mathrm{i} x)=\mathrm{i}^{n-2}Q_n'(x)
\]
transform \eqref{eq:polyF}--\eqref{eq:polyJ} into
\[
P_{n+1}(x)=xP_n(x)+\frac\alpha n(1+x^2)P_n'(x),
\]
\[
Q_{n+1}(x)=\left(1+\frac\alpha n\right)xQ_n(x)
 +\frac\alpha n(1+x^2)Q_n'(x).
\]
The leading coefficients of $P_n$ and $Q_n$ are both
$(1+\alpha)^{n-1}>0$ by Lemma~\ref{lem:basic-polynomials}.

We prove by induction on $n$ that $P_n$ has $n$ simple
real roots and that $Q_n/P_n$ admits a partial fraction
decomposition with strictly positive residues.
For $n=1$, we have $P_1(x)=x$ and $Q_1(x)=1$, so
$Q_1(x)/P_1(x)=1/x$, which establishes the base case.

For the induction step, assume that $P_n$ has simple real roots
$t_{n,1}<\cdots<t_{n,n}$ and
\[
\widetilde m(x)=\frac{Q_n(x)}{P_n(x)}
 =\sum_{j=1}^n\frac{w_{n,j}}{x-t_{n,j}},\qquad w_{n,j}>0.
\]
The parity identities \eqref{guuki} imply that $\widetilde m$ is odd. If its residue at $t$ is $w$, then
\begin{align}\label{symresidue}
\widetilde m(x)=\frac{w}{x-t}+O(1)\quad(x\to t),\qquad
\widetilde m(x)=-\widetilde m(-x)
 =\frac{w}{x+t}+O(1)\quad(x\to-t).
\end{align}
Thus opposite poles have equal weights.
For $x$ outside the roots, define
\[
\widetilde D(x)=\frac{P_{n+1}(x)}{P_n(x)}
=x+\frac\alpha n(1+x^2)\sum_{j=1}^n\frac1{x-t_{n,j}}.
\]
Then
\begin{align*}
\widetilde D'(x)=1+\alpha-\frac\alpha n
 \sum_{j=1}^n\frac{1+t_{n,j}^2}{(x-t_{n,j})^2}>0.
\end{align*}
Using the identity
\[
\frac{1+x^2}{x-t}=x+t+\frac{1+t^2}{x-t},
\]
we can rewrite $\widetilde D(x)$ as
\[
\widetilde D(x)
=(1+\alpha)x
+\frac{\alpha}{n}\sum_{j=1}^n t_{n,j}
+\frac{\alpha}{n}\sum_{j=1}^n
\frac{1+t_{n,j}^2}{x-t_{n,j}}.
\]
It follows that $\widetilde D(x)\sim(1+\alpha)x$
as $x\to\pm\infty$. Since $1+\alpha>0$, we have
$\lim_{x\to\pm\infty}\widetilde D(x)=\pm\infty$.
Moreover, for each $j$, as $x\to t_{n,j}$,
\[
\widetilde D(x)
=\frac{\alpha}{n}
\frac{1+t_{n,j}^2}{x-t_{n,j}}+O(1).
\] Hence
\[
\lim_{x\uparrow t_{n,j}}\widetilde D(x)=+\infty,\qquad
\lim_{x\downarrow t_{n,j}}\widetilde D(x)=-\infty.
\]
Strict monotonicity gives exactly one zero in each of
\[(-\infty,t_{n,1}),\quad (t_{n,j},t_{n,j+1})\ (1\le j\le n-1),\quad (t_{n,n},\infty).\]
Denote these $n+1$ simple real roots of $P_{n+1}$
by $t_{n+1,1}<\cdots<t_{n+1,n+1}$.
They interlace those of $P_n$.

Since $Q_n(x)=\widetilde m(x) P_n(x)$, differentiation gives
$Q_n'(x)=\widetilde m'(x)P_n(x)+\widetilde m(x)P_n'(x)$.
Substituting this into the recurrence for $Q_{n+1}$
and using the definition of $\widetilde D$ yields
\[
\frac{Q_{n+1}(x)}{P_n(x)}
=\widetilde m(x)\widetilde D(x)
 +\frac\alpha n\{x\widetilde m(x)+(1+x^2)\widetilde m'(x)\}.
\]
At a root $s$ of $P_{n+1}$, interlacing gives $P_n(s)\ne0$ and $\widetilde D(s)=0$. Since
\[
\widetilde m'(x)=-\sum_{j=1}^n\frac{w_{n,j}}{(x-t_{n,j})^2},
\]
we find
\[
\frac{Q_{n+1}(s)}{P_n(s)}
=-\frac\alpha n\sum_{j=1}^nw_{n,j}
 \frac{1+st_{n,j}}{(s-t_{n,j})^2}.
\]
Taylor expansion as $y\to s$ gives
\[
\widetilde D(y)=\widetilde D'(s)(y-s)+O((y-s)^2),\qquad
\frac{Q_{n+1}(y)}{P_n(y)}=
\frac{Q_{n+1}(s)}{P_n(s)}+O(y-s).
\]
Noting that $P_{n+1}(y)=P_n(y)\widetilde D(y)$,
the above Taylor expansions yield, for each root
$t_{n+1,k}$ of $P_{n+1}$, as $y\to t_{n+1,k}$,
\[
\frac{Q_{n+1}(y)}{P_{n+1}(y)}
=\frac{w_{n+1,k}}{y-t_{n+1,k}}+O(1),
\]
where
\begin{align*}
w_{n+1,k}
&=\frac{Q_{n+1}(t_{n+1,k})}
{P_n(t_{n+1,k})\widetilde D'(t_{n+1,k})}\\
&=\frac{-\alpha}{n\widetilde D'(t_{n+1,k})}
\sum_{j=1}^n w_{n,j}
\frac{1+t_{n+1,k}t_{n,j}}
{(t_{n+1,k}-t_{n,j})^2}.
\end{align*}
The individual summands in the expression for $w_{n+1,k}$
need not be positive.
To show that their sum is positive, let $t=t_{n,j}>0$
be a root of $P_n$.
By the parity of $P_n$, which follows from \eqref{guuki},
$-t$ is also a root of $P_n$.
By \eqref{symresidue}, the corresponding weights are equal,
and they are strictly positive by the induction hypothesis.
Denote their common value by $w$.
Thus, writing $s=t_{n+1,k}$, the two terms corresponding to
$t$ and $-t$ contribute
\[
w\left\{
\frac{1+st}{(s-t)^2}
+\frac{1-st}{(s+t)^2}
\right\}
\]
to the sum.
The expression in braces is strictly positive, since
\begin{align*}
\frac{1+st}{(s-t)^2}+\frac{1-st}{(s+t)^2}
&=\frac{(1+st)(s+t)^2+(1-st)(s-t)^2}
{(s^2-t^2)^2}\\
&=\frac{2(s^2+t^2+2s^2t^2)}
{(s^2-t^2)^2}>0.
\end{align*}
The term corresponding to a zero root, if present, is
clearly positive by the induction hypothesis.
Thus, the whole sum is positive.
Since $-\alpha>0$ and $\widetilde D'(t_{n+1,k})>0$,
all $w_{n+1,k}$ are positive.

Now $\deg Q_{n+1}=n<n+1=\deg P_{n+1}$,
and the denominator has only simple roots.
Partial fractions yield
\[
\frac{Q_{n+1}(y)}{P_{n+1}(y)}
=\sum_{k=1}^{n+1}\frac{w_{n+1,k}}{y-t_{n+1,k}}.
\]
Equality of the leading coefficients implies that
the left side is asymptotic to $1/y$ as $y\to\infty$;
comparison gives
\[
\sum_{k=1}^{n+1}w_{n+1,k}=1.
\]
Oddness of $Q_{n+1}/P_{n+1}$ again ensures equal
residues at opposite roots, completing the induction.

By the definitions of $P_n$ and $Q_n$,
we obtain \eqref{neg:eq:residues}. The proof is complete.
\end{proof}

The following remark reformulates \eqref{neg:eq:residues}
in a form useful for the proof of Lemma~\ref{neg:lem:log}.

\begin{remark}
Let $w_{n,0}$ denote the weight at the zero root,
or $0$ if no such root exists.
Recall that $P_n$ and $Q_n$ were introduced in the proof of
Lemma~\ref{neg:lem:rotated}.
By \eqref{symresidue} and their definitions,
the residues at opposite poles are equal.
Pairing opposite poles in \eqref{neg:eq:residues} gives
\begin{align}\label{neg:eq:paired}
 \frac{J_n(x)}{F_n(x)}
 =\frac{w_{n,0}}x+
 \sum_{\substack{1\le j\le n\\t_{n,j}>0}}
 \frac{2w_{n,j}x}{x^2+t_{n,j}^2}.
\end{align}
Evaluating at $x=1$ and using $F_n(1)=1$ and
$J_n(1)=h_n$, we obtain
\begin{align}\label{neg:eq:hn-paired}
 h_n=w_{n,0}+
 \sum_{\substack{1\le j\le n\\t_{n,j}>0}}
 \frac{2w_{n,j}}{1+t_{n,j}^2}.
\end{align}
\end{remark}
\subsection{Product representations and triangular recurrences}
Let $\delta_n=n\bmod2$.
For $0<\alpha\le1/2$, let
\[
\rho_{n,1}',\ldots,\rho_{n,[n/2]}'
\]
be the positive roots of $F_n$ and set
\[
b_{n,j}=\frac{1}{1-(\rho_{n,j}')^2}>1,
\qquad 1\le j\le[n/2].
\]
For $-1<\alpha<0$, let
\[
\mathrm{i}t_{n,1}',\ldots,\mathrm{i}t_{n,[n/2]}'
\]
be the roots of $F_n$ with positive imaginary part and set
\[
b_{n,j}=\frac{1}{1+(t_{n,j}')^2}\in(0,1),
\qquad 1\le j\le[n/2].
\]
Parity identities \eqref{guuki} and $F_n(1)=1$ give
\begin{align}
F_n(x)&=x^{\delta_n}
 \prod_{j=1}^{[n/2]}\left(1-b_{n,j}(1-x^2)\right),\label{eq:product1}\\
F_n(\cosh z)&=(\cosh z)^{\delta_n}
 \prod_{j=1}^{[n/2]}\left(1+b_{n,j}\sinh^2z\right).\label{eq:product2}
\end{align}
At $\alpha=0$, \eqref{eq:polyF} reduces to $F_{n+1}(x)=xF_n(x)$, so $F_n(x)=x^n$ and all roots are zero, counted with multiplicity.
Defining $b_{n,j}=1$ for $1\le j\le[n/2]$ therefore makes \eqref{eq:product1}--\eqref{eq:product2} valid at $\alpha=0$ as well.
For these parameter ranges, define
\begin{align}\label{eq:nonminus1theta}
\Theta_k(n)=\frac{\delta_n}{2}+\sum_{j=1}^{[n/2]}b_{n,j}^k
 \quad(k\ge1),\qquad
B_n=\max\left(1,\max_{1\le j\le[n/2]}b_{n,j}\right).
\end{align}
At $n=1$, set $\Theta_k(1)=1/2$ and $B_1=1$.
For $|u|<B_n^{-1}$,
\begin{align}\label{eq:rootlog}
\log F_n(\sqrt{1-u})
=-\sum_{k=1}^{\infty}\frac{\Theta_k(n)}k u^k.
\end{align}
Indeed, substitute $x=\sqrt{1-u}$ in \eqref{eq:product1} and expand the logarithms, choosing the analytic branches with $\sqrt1=1$ and $\log1=0$.

The following proposition establishes triangular recurrences
for $\Theta_k(n)$ and relates these quantities to the
cumulants of the ERW with $q=1/2$.
\begin{proposition}\label{prop:rootrec}
For $-1<\alpha\le1/2$, $n\ge1$, and $k\ge1$,
\begin{align}\label{eq:rootrec}
\Theta_k(n+1)=\Theta_k(n)+\frac12
-k[u^k]\log\left(1-\frac{2\alpha}{n}
 \sum_{j=1}^{\infty}\Theta_j(n)u^j\right),
\end{align}
where $[u^k]f(u)$ denotes the coefficient of $u^k$ in $f(u)$. In particular,
\begin{align}
\Theta_1(n+1)&=(1+2\alpha/n)\Theta_1(n)+\tfrac12,\label{eq:theta1}\\
\Theta_2(n+1)&=(1+4\alpha/n)\Theta_2(n)+\tfrac12
 +4\alpha^2\Theta_1(n)^2/n^2,\label{eq:theta2}\\
\Theta_3(n+1)&=(1+6\alpha/n)\Theta_3(n)+\tfrac12
 +12\alpha^2\Theta_1(n)\Theta_2(n)/n^2
 +8\alpha^3\Theta_1(n)^3/n^3.\label{eq:theta3}
\end{align}
Moreover,
\begin{align}\label{eq:cumroots}
\kappa_2^{(1/2)}=2\Theta_1,\quad
\kappa_4^{(1/2)}=8\Theta_1-12\Theta_2,\quad
\kappa_6^{(1/2)}=32\Theta_1-240\Theta_2+240\Theta_3.
\end{align}
\end{proposition}
\begin{proof}
Put $\Psi_n(u)=F_n(\sqrt{1-u})$. The chain rule gives
$F_n'(\sqrt{1-u})=-2\sqrt{1-u}\,\Psi_n'(u)$, so \eqref{eq:polyF} becomes
\[
\Psi_{n+1}(u)=\sqrt{1-u}
 \left\{\Psi_n(u)+\frac{2\alpha}{n}u\Psi_n'(u)\right\}.
\]
Termwise differentiation of \eqref{eq:rootlog} yields
\[
u\frac{\Psi_n'(u)}{\Psi_n(u)}=-\sum_{j=1}^{\infty}\Theta_j(n)u^j.
\]
Dividing the recurrence by $\Psi_n(u)$ and taking logarithms near zero gives
\[
\log\Psi_{n+1}(u)-\log\Psi_n(u)
=\frac12\log(1-u)+
 \log\left(1-\frac{2\alpha}{n}\sum_{j=1}^{\infty}\Theta_j(n)u^j\right).
\]
The left side equals
$-\sum_{k\ge1}(\Theta_k(n+1)-\Theta_k(n))u^k/k$, whereas
$\frac12\log(1-u)=-\sum_{k\ge1}u^k/(2k)$.
Comparing coefficients and multiplying by $-k$ proves \eqref{eq:rootrec}.

To obtain the first three recurrences, use $-\log(1-z)=z+z^2/2+z^3/3+\cdots$.
The coefficient of $u$ immediately gives \eqref{eq:theta1}.
Writing $A(u)=\sum_{j\ge1}\Theta_j(n)u^j$, we have
\[
A(u)^2=\Theta_1(n)^2u^2+2\Theta_1(n)\Theta_2(n)u^3+O(u^4),\qquad
A(u)^3=\Theta_1(n)^3u^3+O(u^4).
\]
Consequently,
\begin{align*}
[u^2]\{-\log(1-2\alpha A(u)/n)\}
&=\frac{2\alpha}{n}\Theta_2(n)
 +\frac12\left(\frac{2\alpha}{n}\right)^2\Theta_1(n)^2,\\
[u^3]\{-\log(1-2\alpha A(u)/n)\}
&=\frac{2\alpha}{n}\Theta_3(n)
+\left(\frac{2\alpha}{n}\right)^2\Theta_1(n)\Theta_2(n)
+\frac13\left(\frac{2\alpha}{n}\right)^3\Theta_1(n)^3.
\end{align*}
Substituting these coefficients into \eqref{eq:rootrec}
gives \eqref{eq:theta2} and \eqref{eq:theta3}.

For \eqref{eq:cumroots}, note that
$\chi_n^{(1/2)}(\theta)=F_n(\cos\theta)$.
Near zero, $\cos\theta=\sqrt{1-\sin^2\theta}$, and \eqref{eq:rootlog} gives
\[
\log\chi_n^{(1/2)}(\theta)
=-\sum_{k=1}^{\infty}\frac{\Theta_k(n)}k\sin^{2k}\theta.
\]
As $\theta\to0$,
\[
\sin^2\theta=\theta^2-\frac{\theta^4}{3}+\frac{2\theta^6}{45}+O(\theta^8),\quad
\sin^4\theta=\theta^4-\frac23\theta^6+O(\theta^8),\quad
\sin^6\theta=\theta^6+O(\theta^8).
\]
Thus
\begin{align*}
\log\chi_n^{(1/2)}(\theta)
&=-\Theta_1(n)\theta^2
 +\left(\frac{\Theta_1(n)}3-\frac{\Theta_2(n)}2\right)\theta^4\\
&\quad
+\left(-\frac{2\Theta_1(n)}{45}
 +\frac{\Theta_2(n)}3-\frac{\Theta_3(n)}3\right)\theta^6
 +O(\theta^8).
\end{align*}
On the other hand, the cumulant expansion is
$\log\chi_n^{(1/2)}(\theta)
=\sum_{j\ge1}\kappa_j^{(1/2)}(n)(\mathrm{i}\theta)^j/j!$.
Symmetry at $q=1/2$ makes the odd cumulants vanish, so this is
\[
-\frac{\kappa_2^{(1/2)}(n)}{2!}\theta^2
 +\frac{\kappa_4^{(1/2)}(n)}{4!}\theta^4
 -\frac{\kappa_6^{(1/2)}(n)}{6!}\theta^6+O(\theta^8).
\]
Comparing the coefficients of $\theta^2$, $\theta^4$, and $\theta^6$ proves \eqref{eq:cumroots}. The proof is complete.
\end{proof}

The following lemma will be used to obtain a uniform Taylor expansion
of $\log M_n^{(q)}$ in Lemma~\ref{pos:lem:log}.
\begin{lemma}\label{pos:lem:rootbounds}
Suppose $0<\alpha<1/2$.
For $n\ge2$ and $k=1,2,3$, $\Theta_k(n)=O_\alpha(G_k(n))$,
where $G_k$ is defined in \eqref{eq:G}.
In particular,
\[
\max_{1\le j\le n}\frac1{1-\rho_{n,j}}
 =O_\alpha(G_2(n)^{1/2}).
\]
\end{lemma}
\begin{proof}
Equation~\eqref{eq:theta1} and Lemma~\ref{lem:difference} give
\[
\frac{\Theta_1(n)}n\longrightarrow\frac1{2(1-2\alpha)}.
\]
Thus, we have $\Theta_1(n)=O_\alpha(G_1(n))$.

The forcing term in \eqref{eq:theta2} therefore satisfies
\[
f_n\coloneqq \frac12+\frac{4\alpha^2}{n^2}\Theta_1(n)^2
 \longrightarrow\frac12+\frac{\alpha^2}{(1-2\alpha)^2}.
\]
If $\alpha<1/4$, then $4\alpha<1$, and Lemma~\ref{lem:difference}, applied to
$\Theta_2(n+1)=(1+4\alpha/n)\Theta_2(n)+f_n$, gives
\[
\frac{\Theta_2(n)}n\longrightarrow
\frac{\displaystyle\frac12+\frac{\alpha^2}{(1-2\alpha)^2}}{1-4\alpha}.
\]
For $\alpha\ge1/4$, define
\[
g_n=\prod_{j=1}^{n-1}\left(1+\frac{4\alpha}{j}\right)
 \asymp n^{4\alpha},\qquad
 g_{n+1}=\left(1+\frac{4\alpha}{n}\right)g_n.
\]
Dividing \eqref{eq:theta2} by $g_{n+1}$ gives
\[
\frac{\Theta_2(n+1)}{g_{n+1}}-\frac{\Theta_2(n)}{g_n}
 =\frac{f_n}{g_{n+1}}=O_\alpha(g_{n+1}^{-1}).
\]
Since $\Theta_2(1)=1/2$ and $g_1=1$, summation yields
\[
\frac{\Theta_2(n)}{g_n}
 =\frac12+O_\alpha\left(\sum_{j=1}^{n-1}\frac1{g_{j+1}}\right).
\]
At $\alpha=1/4$, $g_n=n$, and the sum is $O(\log n)$; hence $\Theta_2(n)=O(n\log n)$.
For $\alpha>1/4$, $g_{j+1}^{-1}=O_\alpha(j^{-4\alpha})$ is summable, giving $\Theta_2(n)=O_\alpha(n^{4\alpha})$.
Thus, we have
\[
\Theta_2(n)=
\begin{cases}
O_\alpha(n),&4\alpha<1,\\
O_\alpha(n\log n),&4\alpha=1,\\
O_\alpha(n^{4\alpha}),&4\alpha>1,
\end{cases}
\qquad \Theta_2(n)=O_\alpha(G_2(n)).
\]

Since $\Theta_1(n)=O_\alpha(n)$ and
$\Theta_2(n)=O_\alpha(G_2(n))$, the forcing term
in \eqref{eq:theta3} satisfies
\[
\frac12+\frac{12\alpha^2}{n^2}\Theta_1(n)\Theta_2(n)
 +\frac{8\alpha^3}{n^3}\Theta_1(n)^3
 =O_\alpha\left(1+\frac{G_2(n)}n\right).
\]
The corresponding homogeneous product is
$\prod_{j=1}^{n-1}(1+6\alpha/j)\asymp n^{6\alpha}$.
For $\alpha<1/4$, the forcing term is $O_\alpha(1)$, and division by this product followed by summation gives
\[
\Theta_3(n)=O_\alpha\left(n^{6\alpha}
 \left\{1+\sum_{j=1}^{n-1}j^{-6\alpha}\right\}\right)
 =\begin{cases}
O_\alpha(n),&6\alpha<1,\\
O_\alpha(n\log n),&6\alpha=1,\\
O_\alpha(n^{6\alpha}),&6\alpha>1.
\end{cases}
\]
At $\alpha=1/4$, the forcing term is $O(\log n)$, and
$\sum_jj^{-3/2}\log j<\infty$ gives
$\Theta_3(n)=O(n^{3/2})=O(G_3(n))$.
If $\alpha>1/4$, the forcing term is $O_\alpha(1+n^{4\alpha-1})$.
After division by the homogeneous product, both $n^{-6\alpha}$ and $n^{-2\alpha-1}$ are summable. Thus, we have
$\Theta_3(n)=O_\alpha(n^{6\alpha})=O_\alpha(G_3(n))$.
This proves the first assertion for $k=1,2,3$.

Finally, $\Theta_2(n)=\delta_n/2+\sum_\ell b_{n,\ell}^2$ implies
$b_{n,\ell}\le\Theta_2(n)^{1/2}=O_\alpha(G_2(n)^{1/2})$.
If $\rho_{n,j}=\rho_{n,\ell}'>0$, then
\[
\frac1{1-\rho_{n,j}}
 =\frac{1+\rho_{n,\ell}'}{1-(\rho_{n,\ell}')^2}
 \le2b_{n,\ell}=O_\alpha(G_2(n)^{1/2}).
\]
If $\rho_{n,j}\le0$, then $(1-\rho_{n,j})^{-1}\le1$.
Taking the maximum proves the root bound. The proof is complete.
\end{proof}

The preceding lemma is used only for $0<\alpha<1/2$.
When $-1<\alpha<0$, $b_{n,j}<1$, so $\Theta_k(n)\le n/2$ for every $k$.
The following lemma provides the corresponding bounds
at $\alpha=1/2$, which will be used to control
the analytic expansion of $\log M_n^{(q)}$.
\begin{lemma}\label{crit:lem:rootbounds}
If $\alpha=1/2$, then
$\Theta_1(n)=n\mathcal{H}_n/2$, $\Theta_2(n)=O(n^2)$,
$\Theta_3(n)=O(n^3)$, and $B_n=O(n)$.
\end{lemma}
\begin{proof}
By \eqref{eq:theta1}--\eqref{eq:theta3}, the recurrences specialize to
\begin{align*}
\Theta_1(n+1)&=\left(1+\frac1n\right)\Theta_1(n)+\frac12,\\
\Theta_2(n+1)&=\left(1+\frac2n\right)\Theta_2(n)
 +\frac12+\frac{\Theta_1(n)^2}{n^2},\\
\Theta_3(n+1)&=\left(1+\frac3n\right)\Theta_3(n)
 +\frac12+\frac{3\Theta_1(n)\Theta_2(n)}{n^2}
 +\frac{\Theta_1(n)^3}{n^3}.
\end{align*}
Solving the first recurrence with the initial condition
$\Theta_1(1)=1/2$ gives
\[
\Theta_1(n)=\frac{n\mathcal{H}_n}{2}.
\]

Substituting $\Theta_1(n)=n\mathcal{H}_n/2$ into the second
recurrence and dividing by $(n+1)(n+2)$ gives
\[
\frac{\Theta_2(n+1)}{(n+1)(n+2)}
-\frac{\Theta_2(n)}{n(n+1)}
=O\left(\frac{\mathcal{H}_n^2}{n^2}\right).
\]
Since $\sum_{n\ge1}\mathcal{H}_n^2/n^2<\infty$,
we obtain $\Theta_2(n)=O(n^2)$.

Similarly, using $\Theta_1(n)=n\mathcal{H}_n/2$ and
$\Theta_2(n)=O(n^2)$ in the third recurrence
and dividing by $(n+1)(n+2)(n+3)$ gives
\[
\frac{\Theta_3(n+1)}{(n+1)(n+2)(n+3)}
-\frac{\Theta_3(n)}{n(n+1)(n+2)}
=O\left(
\frac{\mathcal{H}_n}{n^2}
+\frac{\mathcal{H}_n^3}{n^3}
\right).
\]
Since the right-hand side is summable,
we obtain $\Theta_3(n)=O(n^3)$.
Finally, each $b_{n,\ell}\le\Theta_2(n)^{1/2}$, and hence
\[
B_n\le\max(1,\Theta_2(n)^{1/2})=O(n). \qedhere
\]
The proof is complete.
\end{proof}

\subsection{The endpoint cases $\alpha=0$ and $\alpha=-1$}\label{sec:endpoints}
At $\alpha=0$, \eqref{eq:transition} shows that the increments from the second step onward are independent fair signs. Therefore
\begin{align}\label{eq:independent}
 M_n^{(q)}(z)=(\cosh z)^n(1+\beta\tanh z),\qquad
 \mu_n=\beta,\quad v_n^2=n-\beta^2.
\end{align}
At $\alpha=-1$, $X_2=-X_1$, so $S_2=0$.
Since \eqref{eq:transition} depends only on the current position and time, the distribution at every $n\ge2$ is independent of $q$.

The following lemma establishes a product representation
for $F_n$ at $\alpha=-1$, which will be used in
Lemma~\ref{lem:antifourier} and in the derivation
of \eqref{eq:nonposlogsummary}.
\begin{lemma}\label{lem:minus-one-product}
Suppose $\alpha=-1$ and $n\ge2$.
Then $F_n$ has degree $n-2$, and all its roots
are simple and lie on the imaginary axis.
Moreover,
\[
F_n(x)=x^{\delta_n}
 \prod_{j=1}^{[(n-2)/2]}
 \{1-b_{n,j}(1-x^2)\},
\qquad 0<b_{n,j}<1.
\]
\end{lemma}
\begin{proof}
Set $P_n(y)=\mathrm{i}^{-n}F_n(\mathrm{i} y)$.
Since $F_1(x)=x$, \eqref{eq:polyF} with $\alpha=-1$ gives
\[
F_2(x)=x^2+(1-x^2)=1,
\qquad
P_2(y)=\mathrm{i}^{-2}F_2(\mathrm{i}y)=-1.
\]
Moreover, \eqref{eq:polyF} yields
\[
P_{n+1}(y)=yP_n(y)-\frac1n(1+y^2)P_n'(y).
\]
If $P_n$ has degree $n-2$ with leading coefficient $c_n\ne0$,
then the recurrence shows that $P_{n+1}$ has leading
coefficient
\[
c_{n+1}=\left(1-\frac{n-2}{n}\right)c_n
=\frac2n c_n\ne0.
\]
Since $P_2=-1$, induction gives $\deg P_n=\deg F_n=n-2$.
We follow the argument for real-rootedness and interlacing
in the proof of Lemma~\ref{neg:lem:rotated}.
Suppose that all the roots of $P_n$ are real and simple,
and denote them by $t_1,\ldots,t_{n-2}$.
The recurrence gives
\[
\left(\frac{P_{n+1}}{P_n}\right)'(y)
=\frac2n+\frac1n\sum_{j=1}^{n-2}
\frac{1+t_j^2}{(y-t_j)^2}>0
\]
for $y\in\mathbb R\setminus\{t_1,\ldots,t_{n-2}\}$.
Moreover, $P_{n+1}/P_n$ has a simple pole at each $t_j$
with residue $-(1+t_j^2)/n<0$, and
$P_{n+1}(y)/P_n(y)\sim 2y/n$ as $y\to\pm\infty$.
The same monotonicity argument therefore shows that
$P_{n+1}$ has simple real roots interlacing those of $P_n$.
Since $P_3=-y$, the assertion follows by induction.
By the definition of $P_n$, all roots of $F_n$ lie on
the imaginary axis.
Denote the positive roots of $P_n$ by
$t_{n,1}',\ldots,t_{n,[(n-2)/2]}'$ and set
$b_{n,j}=(1+(t_{n,j}')^2)^{-1}\in(0,1)$.
The parity of $F_n$ and $F_n(1)=1$ then give
\[
F_n(x)=x^{\delta_n}
 \prod_{j=1}^{[(n-2)/2]}
 \{1-b_{n,j}(1-x^2)\},\qquad 0<b_{n,j}<1\quad(n\ge2).
\]
The proof is complete.
\end{proof}
For $\alpha=-1$ and $n\ge2$, let $b_{n,j}$ be as
defined in Lemma~\ref{lem:minus-one-product}, and set
\begin{align}\label{eq:minus1theta}
\Theta_k(n)=\frac{\delta_n}{2}
+\sum_{j=1}^{[(n-2)/2]}b_{n,j}^k,
\qquad k\ge1.
\end{align}
The same argument as for $-1<\alpha\le1/2$ shows that
\eqref{eq:cumroots} also holds at $\alpha=-1$ for $n\ge2$.

\section{Fourier bounds from products and random trees}\label{sec:fourier}
Since $S_n\equiv n\pmod2$, we have
$\chi_n^{(q)}(\theta+\pi)=(-1)^n\chi_n^{(q)}(\theta)$.
For $s\equiv n\pmod2$, the integrand
$\chi_n^{(q)}(\theta)e^{-\mathrm{i}s\theta}$
is $\pi$-periodic. 
The Fourier inversion formula gives
\begin{align}
 \mathbb{P}_q(S_n=s)
 &=\frac1{2\pi}\left(
 \int_{-\pi}^{0}\chi_n^{(q)}(\theta)e^{-\mathrm{i}s\theta}\,d\theta
 +\int_{0}^{\pi}\chi_n^{(q)}(\theta)e^{-\mathrm{i}s\theta}\,d\theta
 \right)
 \notag\\
 &=\frac1{2\pi}\left(
 2\int_{-\pi/2}^{\pi/2}\chi_n^{(q)}(\theta)e^{-\mathrm{i}s\theta}\,d\theta
 \right)
 \notag\\
 &=\frac1\pi\int_{-\pi/2}^{\pi/2}
 \chi_n^{(q)}(\theta)e^{-\mathrm{i}s\theta}\,d\theta.
 \label{eq:inversion}
\end{align}

\subsection{General Fourier bounds}
The following lemma establishes an exponential bound for the characteristic function, uniformly in $q$.
\begin{lemma}\label{lem:fourier}
For $-1<\alpha<1$, $n\ge2$, every real $\theta$, and every $q\in[0,1]$,
\begin{align}\label{eq:global}
 |\chi_n^{(q)}(\theta)|\le3e^{-\lambda_\alpha n\sin^2\theta},\qquad
 \lambda_\alpha=\min\{(1-\alpha^2)/8,1/32\}>0.
\end{align}
\end{lemma}
\begin{proof}
We use the random recursive tree representation of the ERW
(see K\"ursten~%
\cite{Kursten2016}).
Recall from the definition of the ERW in the Introduction
that $X_j=\eta_jX_{U_j}$ for $j\ge2$.

The directed edges $U_j\to j$ form a random recursive tree $T_n$ on $\{1,\ldots,n\}$, rooted at $1$.
Labels increase strictly along each edge, so no cycle is possible.
Let $\mathcal L_n$ be its set of leaves, $L_n=|\mathcal L_n|$, and
\[
\mathcal G_n=\sigma\bigl(T_n,X_1,(\eta_k:2\le k\le n,\ k\notin\mathcal L_n)\bigr).
\]
For $n\ge2$, the root is not a leaf. If $j$ is a leaf, its parent $U_j$ and all ancestors of $U_j$ are nonleaves, so $X_{U_j}$ is $\mathcal G_n$-measurable.
Every $X_k$ at a nonleaf is likewise $\mathcal G_n$-measurable.
The leaf set depends only on the parent variables; their independence from the signs means that, conditionally on $\mathcal G_n$, the signs $(\eta_j:j\in\mathcal L_n)$ remain independent with probabilities $p$ and $1-p$.
Writing
\[
S_n=\sum_{k\notin\mathcal L_n}X_k
 +\sum_{j\in\mathcal L_n}\eta_jX_{U_j},
\]
we obtain
\[
\mathbb{E}_q[e^{\mathrm{i}\theta S_n}\mid\mathcal G_n]
=\exp\left(\mathrm{i}\theta\sum_{k\notin\mathcal L_n}X_k\right)
 \prod_{j\in\mathcal L_n}
 \left(pe^{\mathrm{i}\theta X_{U_j}}+(1-p)e^{-\mathrm{i}\theta X_{U_j}}\right).
\]
The modulus of each factor is independent of the sign of $X_{U_j}$, and
\begin{align*}
|pe^{\mathrm{i}\theta}+(1-p)e^{-\mathrm{i}\theta}|^2
&=p^2+(1-p)^2+2p(1-p)\cos(2\theta)\\
&=1-4p(1-p)\sin^2\theta\\
&=1-(1-\alpha^2)\sin^2\theta.
\end{align*}
Taking absolute values and then expectations gives
\[
|\chi_n^{(q)}(\theta)|
\le\mathbb{E}_q\left[\{1-(1-\alpha^2)\sin^2\theta\}^{L_n/2}\right].
\]

For $2\le j\le n$, vertex $j$ is a leaf exactly when no later vertex chooses it as parent. Independence of the $U_k$ therefore gives
\[
\mathbb{P}_q(j\in\mathcal L_n)
 =\prod_{k=j+1}^n\left(1-\frac1{k-1}\right)
 =\frac{j-1}{n-1},\qquad
\mathbb{E}_q [L_n]=\frac1{n-1}\sum_{j=2}^n(j-1)=\frac n2.
\]
Changing one parent variable from $u$ to $v$ can change the leaf status only of $u$ and $v$, so the change in $L_n$ is at most $2$.
McDiarmid's bounded differences inequality~%
\cite{McDiarmid1989}*{Lemma~1.2}, applied to the $n-1$ independent parent variables, gives
\[
\mathbb{P}_q(L_n\le n/4)
=\mathbb{P}_q(L_n-\mathbb{E} L_n\le-n/4)
\le2\exp\left(-\frac{2(n/4)^2}{4(n-1)}\right)
\le 2e^{-n/32}.
\]
Put $u=(1-\alpha^2)\sin^2\theta\in[0,1]$.
On $\{L_n>n/4\}$, $(1-u)^{L_n/2}\le(1-u)^{n/8}\le e^{-un/8}$; on the complementary event it is at most $1$.
Hence
\begin{align*}
|\chi_n^{(q)}(\theta)|
&\le\mathbb{E}_q[(1-u)^{L_n/2};L_n>n/4]
 +\mathbb{E}_q[(1-u)^{L_n/2};L_n\le n/4]\\
&\le e^{-un/8}+\mathbb{P}_q(L_n\le n/4)\\
&\le\exp\left(-\frac{1-\alpha^2}{8}n\sin^2\theta\right)
+2e^{-n/32}.
\end{align*}
Since $e^{-n/32}\le e^{-(n/32)\sin^2\theta}$, this proves \eqref{eq:global} with the stated $\lambda_\alpha$.
The argument does not depend on $q$, so the bound is uniform in $q\in[0,1]$. The proof is complete.
\end{proof}

Lemma \ref{lem:fourier} does not cover the endpoint $\alpha=-1$.
For $-1\le\alpha\le0$, the product representation of $F_n$
yields a Fourier bound when $q=1/2$.
\begin{lemma}\label{lem:antifourier}
If $-1\le\alpha\le0$ and $q=1/2$, then, for $n\ge2$ and $|\theta|\le\pi/2$,
\begin{align*}
 |\chi_n^{(1/2)}(\theta)|\le e^{-\sigma_n^2\sin^2\theta/2}.
\end{align*}
At $\alpha=-1$, the bound holds for every $q$ when $n\ge2$.
\end{lemma}
\begin{proof}
By \eqref{eq:decomp}, setting $q=1/2$ (so $\beta=0$) gives
$\chi_n^{(1/2)}(\theta)=F_n(\cos\theta)$.
For $-1<\alpha\le0$, put $N_n=[ n/2]$; for $\alpha=-1$ and $n\ge2$, put $N_n=[(n-2)/2]$.
The product representations in \eqref{eq:product1} and Lemma~\ref{lem:minus-one-product} give
\[
F_n(x)=x^{\delta_n}\prod_{j=1}^{N_n}\{1-b_{n,j}(1-x^2)\},
\qquad 0<b_{n,j}\le1.
\]
With $u=\sin^2\theta\in[0,1]$, each factor satisfies
$0\le1-b_{n,j}u\le e^{-b_{n,j}u}$.
Also, $\cos\theta\ge0$ on the stated interval, and
$(\cos\theta)^{\delta_n}=(1-u)^{\delta_n/2}\le e^{-\delta_nu/2}$.
Consequently, by \eqref{eq:nonminus1theta} and \eqref{eq:minus1theta},
\[
|\chi_n^{(1/2)}(\theta)|
\le\exp\left\{-u\left(\frac{\delta_n}{2}
 +\sum_{j=1}^{N_n}b_{n,j}\right)\right\}
=e^{-\Theta_1(n)u}.
\]
By \eqref{eq:cumroots} and its extension to $\alpha=-1$
stated immediately after the proof of Lemma~\ref{lem:minus-one-product},
$2\Theta_1(n)=\kappa_2^{(1/2)}(n)=\operatorname{Var}_{1/2}(S_n)=\sigma_n^2$,
proving the bound in these cases.
Finally, as noted in Section~\ref{sec:endpoints}, the law of $S_n$ at $\alpha=-1$ is independent of $q$ for $n\ge2$.
The bound therefore holds for every $q$ in this case. The proof is complete.
\end{proof}

\subsection{Refined Fourier bounds at criticality}\label{sec:criticalwindow}
When $\alpha=1/2$, the bound in Lemma~\ref{lem:fourier} is
insufficient to control the Fourier integral outside the central region
in the proof of Theorem~\ref{thm:edgeworth}.
The following lemma uses the product representation of $F_n$
to obtain a stronger Fourier bound.

\begin{lemma}\label{crit:lem:window}
Suppose $\alpha=1/2$. There are constants $\eta,c,C>0$ such that, for all sufficiently large $n$, all $q\in[0,1]$, and all $|\theta|\le\pi/2$ with $\sin^2\theta\le\eta\mathcal{H}_n/n$,
\begin{align}\label{crit:eq:window}
|\chi_n^{(q)}(\theta)|
\le C\{1+\sqrt{n\sin^2\theta}(1+n\sin^2\theta)\}
 e^{-cn\mathcal{H}_n\sin^2\theta}.
\end{align}
In particular, writing $\xi=v_n\theta$ in this region gives
$|\chi_n^{(q)}(\xi/v_n)|\le C'e^{-c'\xi^2}$ uniformly in $q$.
\end{lemma}
\begin{proof}
Put $x=\cos\theta\in[0,1]$ and $\tau=\sin^2\theta=1-x^2$.
We first note that
\[
|1-y|\le e^{-y+y^2}\qquad(y\ge0).
\]

Applying the preceding inequality to the factors in \eqref{eq:product1} gives
$|1-b_{n,j}\tau|\le\exp(-b_{n,j}\tau+b_{n,j}^2\tau^2)$, while
$x^{\delta_n}=(1-\tau)^{\delta_n/2}\le e^{-\delta_n\tau/2}$.
Therefore
\begin{align*}
|F_n(x)|
&\le\exp\left\{-\frac{\delta_n}{2}\tau
 -\tau\sum_{j=1}^{[n/2]}b_{n,j}+\tau^2\sum_{j=1}^{[n/2]}b_{n,j}^2\right\}\\
&\le\exp\{-\Theta_1(n)\tau+\Theta_2(n)\tau^2\}.
\end{align*}

To estimate $J_n$, for each root $\rho_{n,j}$ of $F_n$ as in Lemma~\ref{pos:lem:residues} define the polynomial
\[
T_{n,j}(x)=(1-\rho_{n,j})\frac{F_n(x)}{x-\rho_{n,j}}.
\]
It is a polynomial by the factor theorem.
If $\rho_{n,1}',\ldots,\rho_{n,[n/2]}'$ are the positive roots,
the parity relation \eqref{guuki} and the product representation
\eqref{eq:product1} give
\[
F_n(x)=x^{\delta_n}\prod_{k=1}^{[n/2]}
 \frac{x^2-(\rho_{n,k}')^2}{1-(\rho_{n,k}')^2},\qquad
\frac{x^2-(\rho_{n,k}')^2}{1-(\rho_{n,k}')^2}
 =\frac{x-\rho_{n,k}'}{1-\rho_{n,k}'}\cdot\frac{x+\rho_{n,k}'}{1+\rho_{n,k}'}.
\]
If $\rho_{n,j}=\rho_{n,k}'>0$, cancellation leaves
\[
T_{n,j}(x)=\frac{x+\rho_{n,k}'}{1+\rho_{n,k}'}\,
 x^{\delta_n}\prod_{\ell\ne k}
 \frac{x^2-(\rho_{n,\ell}')^2}{1-(\rho_{n,\ell}')^2},
\]
and $0\le(x+\rho_{n,k}')/(1+\rho_{n,k}')\le1$.
If $\rho_{n,j}=-\rho_{n,k}'<0$, cancellation instead leaves
\[
T_{n,j}(x)=\frac{x-\rho'_{n,k}}{1-\rho'_{n,k}}\,
 x^{\delta_n}\prod_{\ell\ne k}
 \frac{x^2-(\rho_{n,\ell}')^2}{1-(\rho_{n,\ell}')^2}.
\]
By Lemma~\ref{pos:lem:residues}, $0<\rho_{n,k}'<1$.
Since $1-x=\tau/(1+x)\le\tau$ and
\[
\frac1{1-\rho_{n,k}'}=(1+\rho_{n,k}')b_{n,k}
\le2b_{n,k}\le2B_n,
\]
the remaining factor satisfies
\[
\left|\frac{x-\rho_{n,k}'}{1-\rho_{n,k}'}\right|
 =\left|1-\frac{1-x}{1-\rho_{n,k}'}\right|
 \le1+2B_n\tau.
\]
For either nonzero root, the other factors satisfy
\begin{align*}
x^{\delta_n}\prod_{\ell\ne k}
 \left|\frac{x^2-(\rho_{n,\ell}')^2}{1-(\rho_{n,\ell}')^2}\right|
&\le\exp\left\{-\frac{\delta_n}{2}\tau
 -\sum_{\ell\ne k}b_{n,\ell}\tau
 +\sum_{\ell\ne k}b_{n,\ell}^2\tau^2\right\}\\
&\le\exp\{-\Theta_1(n)\tau+b_{n,k}\tau+\Theta_2(n)\tau^2\}\\
&\le\exp\{-\Theta_1(n)\tau+B_n\tau+\Theta_2(n)\tau^2\}.
\end{align*}
If $\rho_{n,j}=0$, then $\delta_n=1$ and cancellation removes the factor $x$, leaving
\[
T_{n,j}(x)=\prod_{\ell=1}^{[n/2]}
 \frac{x^2-(\rho_{n,\ell}')^2}{1-(\rho_{n,\ell}')^2}.
\]
Its modulus is at most
$\exp\{-\tau\sum_\ell b_{n,\ell}+\tau^2\sum_\ell b_{n,\ell}^2\}$,
which is also bounded by
$\exp\{-\Theta_1(n)\tau+B_n\tau+\Theta_2(n)\tau^2\}$ because $B_n\ge1$.
Thus every root satisfies
\begin{align}\label{crit:eq:quotient-bound}
|T_{n,j}(x)|\le(1+2B_n\tau)
 \exp\{-\Theta_1(n)\tau+B_n\tau+\Theta_2(n)\tau^2\}.
\end{align}

Lemma \ref{pos:lem:residues} gives
\[
J_n(x)=\sum_{j=1}^n\frac{w_{n,j}}{1-\rho_{n,j}}T_{n,j}(x),\qquad
\sum_{j=1}^n\frac{w_{n,j}}{1-\rho_{n,j}}=h_n.
\]
Since every weight is positive, \eqref{crit:eq:quotient-bound} implies
\[
|J_n(x)|\le h_n(1+2B_n\tau)
 \exp\{-\Theta_1(n)\tau+B_n\tau+\Theta_2(n)\tau^2\}.
\]
Now \eqref{eq:decomp} reads
$\chi_n^{(q)}(\theta)=F_n(\cos\theta)+\mathrm{i}\beta\sin\theta J_n(\cos\theta)$.
Combining the preceding bounds and $|\sin\theta|=\sqrt\tau$ yields
\[
|\chi_n^{(q)}(\theta)|
\le\{1+|\beta|h_n\sqrt\tau(1+2B_n\tau)\}
 \exp\{-\Theta_1(n)\tau+B_n\tau+\Theta_2(n)\tau^2\}.
\]

At $\alpha=1/2$, Lemma~\ref{crit:lem:rootbounds} gives
$\Theta_1(n)=n\mathcal{H}_n/2$, $\Theta_2(n)\le Cn^2$, and $B_n\le C'n$.
Choose $0<\eta\le1/(8C)$ and take $n$ sufficiently large that
$\mathcal{H}_n\ge8C'$.
Then, for $\tau\le\eta\mathcal{H}_n/n$,
\[
\begin{aligned}
-\Theta_1(n)\tau+B_n\tau+\Theta_2(n)\tau^2
&\le-\frac12n\mathcal{H}_n\tau+C'n\tau+Cn^2\tau^2\\
&\le-\frac14n\mathcal{H}_n\tau.
\end{aligned}
\]
Since $h_n=O(\sqrt n)$ and $|\beta|\le1$, we obtain uniformly in $q$
\[
|\chi_n^{(q)}(\theta)|
\le C\{1+\sqrt{n\tau}(1+n\tau)\}
 \exp\{-n\mathcal{H}_n\tau/4\},
\]
which proves \eqref{crit:eq:window}.

Finally, let $\theta=\xi/v_n$. Uniformly in $q$, $v_n^2\asymp n\mathcal{H}_n$ by \eqref{vnasymp}.
On $|\theta|\le\pi/2$,
$4\theta^2/\pi^2\le\sin^2\theta\le\theta^2$, so
\[
n\mathcal{H}_n\tau\ge\frac4{\pi^2}\frac{n\mathcal{H}_n}{v_n^2}\xi^2
 \ge c\xi^2,\qquad
n\tau\le\frac n{v_n^2}\xi^2
 \le C\frac{\xi^2}{\mathcal{H}_n}\le C\xi^2.
\]
It follows that $1+\sqrt{n\tau}(1+n\tau)\le C(1+|\xi|^3)$ and
\[
|\chi_n^{(q)}(\xi/v_n)|\le C(1+|\xi|^3)e^{-c\xi^2}
 \le C_0e^{-c_0\xi^2}.
\]
The proof is complete.
\end{proof}

The following lemma will be used in the proof of Theorem~\ref{thm:edgeworth}.
\begin{lemma}\label{crit:lem:outer}
Suppose $\alpha=1/2$, and let $\eta$ be as in Lemma~\ref{crit:lem:window}.
For sufficiently large $n$, put
\[
\theta_n=\arcsin\sqrt{\frac{\eta\mathcal{H}_n}{n}}.
\]
There is a constant $d>0$ such that, uniformly in $q\in[0,1]$,
\begin{align}\label{crit:eq:outer}
v_n\int_{\theta_n\le|\theta|\le\pi/2}
 |\chi_n^{(q)}(\theta)|\,d\theta=O(n^{-d}).
\end{align}
\end{lemma}
\begin{proof}
Lemma~\ref{lem:fourier} and
$\sin^2\theta\ge4\theta^2/\pi^2$ on $|\theta|\le\pi/2$ give
\[
|\chi_n^{(q)}(\theta)|\le3e^{-c_0n\sin^2\theta}
 \le3e^{-c_1n\theta^2},
\]
with constants independent of $q$.
Thus, changing variables to $u=\sqrt n\,\theta$,
\begin{align*}
v_n\int_{\theta_n\le|\theta|\le\pi/2}
 |\chi_n^{(q)}(\theta)|\,d\theta
&\le6v_n\int_{\theta_n}^{\pi/2}e^{-c_1n\theta^2}\,d\theta\\
&\le\frac{6v_n}{\sqrt n}
 \int_{\sqrt n\,\theta_n}^{\infty}e^{-c_1u^2}\,du\\
&\le C\sqrt{\mathcal{H}_n}\int_{\sqrt n\,\theta_n}^{\infty}e^{-c_1u^2}\,du,
\end{align*}
where the last inequality uses $v_n^2\asymp n\mathcal{H}_n$ uniformly in $q$ by \eqref{vnasymp}.
For some $C,c_2>0$, the Gaussian tail satisfies
$\int_a^\infty e^{-c_1u^2}\,du\le Ce^{-c_2a^2}$ for every $a\ge0$.
Also, $\arcsin y\ge y$ for $0\le y\le1$, so
$n\theta_n^2\ge\eta\mathcal{H}_n$.
Consequently,
\[
v_n\int_{\theta_n\le|\theta|\le\pi/2}
 |\chi_n^{(q)}(\theta)|\,d\theta
\le C\sqrt{\mathcal{H}_n}e^{-c_2n\theta_n^2}
\le C\sqrt{\mathcal{H}_n}e^{-c_2\eta\mathcal{H}_n}.
\]
Since $\mathcal{H}_n=\log n+O(1)$,
\[
\sqrt{\mathcal{H}_n}e^{-c_2\eta\mathcal{H}_n}
=O\bigl((\log n)^{1/2}n^{-c_2\eta}\bigr)
=O(n^{-d})
\]
for any fixed $d\in(0,c_2\eta)$.
This proves \eqref{crit:eq:outer}. The proof is complete.
\end{proof}

\section{Analytic logarithms and finite-time remainders}\label{sec:analytic}
In this section, we establish fourth-order expansions
of the logarithm of the moment generating function
on suitable complex disks, together with uniform
remainder estimates.
These estimates will be used in the proof of
Theorem~\ref{thm:edgeworth}.
We treat the different memory regimes and
the endpoints separately.

\begin{lemma}\label{neg:lem:log}
Fix $-1<\alpha<0$. There is an $R>0$, independent of $n$ and $q$, such that $M_n^{(q)}$ has an analytic logarithm in a neighborhood of $|z|\le R$, normalized to vanish at the origin. For $|z|\le R/2$,
\begin{align}
 \log M_n^{(q)}(z)
 &=\mu_nz+\frac{v_n^2z^2}{2}
 +\frac{\kappa_3^{(q)}(n)z^3}{6}
 +\frac{\kappa_4^{(q)}(n)z^4}{24}
 +O_p\bigl(n|z|^6+|\beta|h_n|z|^5\bigr).
 \label{neg:eq:log}
\end{align}
Moreover, $\kappa_3^{(q)}(n)=O_p(|\beta|h_n)$ and $\kappa_4^{(q)}(n)=O_p(n)$.
\end{lemma}

\begin{proof}
Choose a sufficiently small complex neighborhood of $x=1$ in which $|x|\le C$, $|1/x|\le C$, and $|x^2-1|\le1/2$. For every real $t$,
\[
 |x^2+t^2|\ge1+t^2-|x^2-1|\ge\tfrac12(1+t^2),
 \qquad
 \left|\frac{x}{x^2+t^2}\right|\le\frac{C}{1+t^2}.
\]
The paired representation~\eqref{neg:eq:paired} and positivity of $w_{n,j}$ therefore give
\[
 \left|\frac{J_n(x)}{F_n(x)}\right|
 \le w_{n,0}\left|\frac1x\right|
   +\sum_{t_{n,j}>0}2w_{n,j}\left|\frac{x}{x^2+t_{n,j}^2}\right|
 \le C\left(w_{n,0}+\sum_{t_{n,j}>0}\frac{2w_{n,j}}{1+t_{n,j}^2}\right)
 =Ch_n,
\]
where the last equality uses~\eqref{neg:eq:hn-paired}. For $z$ in a fixed sufficiently small disk, $x=\cosh z$ lies in this neighborhood. Thus, with
\[
 A_n(z)=\sinh z\,\frac{J_n(\cosh z)}{F_n(\cosh z)},
\]
we have $|A_n(z)|\le Ch_n|z|$. Since $h_n\le1$ for $-1<\alpha<0$, we may choose a fixed $R>0$ so small that $|A_n(z)|\le1/2$ on $|z|\le R$, uniformly in $n$. Consequently, $|\beta A_n(z)|\le1/2$, and $1+\beta A_n(z)$ has an analytic logarithm vanishing at zero.

Recall that \eqref{eq:product2} gives
\[
 F_n(\cosh z)=(\cosh z)^{\delta_n}
 \prod_{j=1}^{[n/2]}(1+b_{n,j}\sinh^2z),
 \qquad b_{n,j}=\frac{1}{1+(t_{n,j}')^2}\in(0,1).
\]
Shrinking $R$ further so that $|\sinh^2z|\le1/2$ ensures that every factor, as well as $\cosh z$, is nonzero. Hence $\log F_n(\cosh z)$ is analytic, with value $\log F_n(1)=0$ at zero. By~\eqref{eq:decomp},
\[
 M_n^{(q)}(z)=F_n(\cosh z)\{1+\beta A_n(z)\},\qquad
 \log M_n^{(q)}(z)=\log F_n(\cosh z)+\log(1+\beta A_n(z)).
\]
All inequalities leave a fixed margin, so the logarithms are analytic in a neighborhood of the closed disk.

For $|w|\le1/2$, $|\log(1+w)|\le C|w|$. Thus, on $|z|\le R$,
\[
 |\log(1+\beta A_n(z))|\le C|\beta|h_n|z|\le C_R|\beta|h_n.
\]
Cauchy's estimates give, for each fixed $r\ge1$,
\begin{align}\label{ineq:Cauchy}
 \left|\left.\frac{d^r}{dz^r}\log(1+\beta A_n(z))\right|_{z=0}\right|
 \le\frac{r!}{R^r}\sup_{|z|=R}|\log(1+\beta A_n(z))|
 \le C_{R,r}|\beta|h_n.
\end{align}
By Cauchy's estimates and the geometric series formula, for $|z|\le R/2$ we obtain
\begin{align}\label{eq:log}
 \log(1+\beta A_n(z))
 =\sum_{r=1}^4\frac{z^r}{r!}
 \left.\frac{d^r}{dz^r}\log(1+\beta A_n(z))\right|_{z=0}
 +O_p(|\beta|h_n|z|^5).
\end{align}

To expand $\log F_n(\cosh z)$, we use
\[
 \log\cosh z=\frac{z^2}{2}-\frac{z^4}{12}+O(|z|^6),
 \qquad \sinh^2z=z^2+\frac{z^4}{3}+O(|z|^6).
\]
Uniformly for $b\in[0,1]$,
\[
 \log(1+b\sinh^2z)
 =b\sinh^2z-\frac{b^2}{2}\sinh^4z+O(|z|^6)
 =bz^2+\left(\frac b3-\frac{b^2}{2}\right)z^4+O(|z|^6).
\]
Adding the factors gives
\begin{align*}
 \log F_n(\cosh z)
 &=\left(\frac{\delta_n}{2}+\sum_{j=1}^{[n/2]}b_{n,j}\right)z^2
 +\left\{-\frac{\delta_n}{12}
 +\sum_{j=1}^{[n/2]}\left(\frac{b_{n,j}}3-\frac{b_{n,j}^2}{2}\right)\right\}z^4
 +O_p(n|z|^6).
\end{align*}
By the definition of $\Theta_k(n)$, we obtain
\begin{align}\label{eq:Fncosh}
 \log F_n(\cosh z)
 =\Theta_1(n)z^2+
 \left(\frac{\Theta_1(n)}3-\frac{\Theta_2(n)}2\right)z^4
 +O_p(n|z|^6).
\end{align}
Since $M_n^{(1/2)}(z)=F_n(\cosh z)$,
comparing the coefficients of $z^4$ in \eqref{eq:Fncosh}
and the cumulant expansion yields
\[
 \kappa_4^{(1/2)}(n)=8\Theta_1(n)-12\Theta_2(n)=O_p(n),
\]
where the last estimate follows from $0\le\Theta_k(n)\le n/2$.
For $q=1/2$ the law is symmetric, so all odd cumulants vanish, including $\kappa_3^{(1/2)}(n)$.

Adding the expansions~\eqref{eq:log} and~\eqref{eq:Fncosh}
gives the remainder estimate in~\eqref{neg:eq:log},
while the Taylor coefficients are identified by the definition of the cumulants. Since \eqref{eq:Fncosh} contains no cubic term,
\eqref{ineq:Cauchy} yields
$\kappa_3^{(q)}(n)=O_p(|\beta|h_n)$.
Similarly, applying \eqref{ineq:Cauchy} with $r=4$, we obtain
\begin{align}\label{eq:kappa4}
 \kappa_4^{(q)}(n)
 =\kappa_4^{(1/2)}(n)+O_p(|\beta|h_n)
 =O_p(n),
\end{align}
where the last estimate follows from $|\beta|\le1$ and $h_n\le1$. The proof is complete.
\end{proof}

Here $G_k$ is as defined in~\eqref{eq:G}.
\begin{lemma}\label{pos:lem:log}
Let $0<\alpha<1/2$ and $n\ge2$. There is a constant $c_\alpha>0$ such that, with $R_n=c_\alpha G_2(n)^{-1/4}$, the logarithm of $M_n^{(q)}$ normalized to vanish at zero is analytic in a neighborhood of $|z|\le R_n$, uniformly for $q\in[0,1]$. For $|z|\le R_n/2$,
\begin{align}
 \log M_n^{(q)}(z)
 &=\mu_nz+\frac{v_n^2z^2}{2}
 +\frac{\kappa_3^{(q)}(n)z^3}{6}
 +\frac{\kappa_4^{(q)}(n)z^4}{24}\notag\\
 &\quad+O_\alpha\bigl(G_3(n)|z|^6
 +|\beta|n^\alpha G_2(n)|z|^5\bigr).
 \label{pos:eq:logrem}
\end{align}
\end{lemma}

\begin{proof}
Set $A_n(z)=\sinh z\,J_n(\cosh z)/F_n(\cosh z)$. Lemma~\ref{pos:lem:rootbounds} gives
\[
 \max_{1\le j\le n}\frac1{1-\rho_{n,j}}\le C_\alpha G_2(n)^{1/2},
 \qquad 1-\rho_{n,j}\ge C_\alpha^{-1}G_2(n)^{-1/2}.
\]
On a fixed small disk, $|\cosh z-1|\le C|z|^2$. Thus, for sufficiently small $c_\alpha$, uniformly in $n,j$ and $|z|\le R_n$,
\begin{align*}
 |\cosh z-1|\le Cc_\alpha^2G_2(n)^{-1/2}
 \le\frac{1-\rho_{n,j}}2.
\end{align*}
Consequently,
\begin{align*}
 |\cosh z-\rho_{n,j}|
 \ge1-\rho_{n,j}-|\cosh z-1|
 \ge\frac{1-\rho_{n,j}}2>0.
\end{align*}
Since the $\rho_{n,j}$ are all the zeros of $F_n$,
it follows that $F_n(\cosh z)\ne0$. By Lemma \ref{pos:lem:residues}, we have
\[
 \left|\frac{J_n(\cosh z)}{F_n(\cosh z)}\right|
 \le\sum_{j=1}^{n}\frac{w_{n,j}}{|\cosh z-\rho_{n,j}|}
 \le2\sum_{j=1}^{n}\frac{w_{n,j}}{1-\rho_{n,j}}=2h_n.
\]
Since $|\sinh z|\le C|z|$,
\begin{align}\label{pos:eq:A-bound}
 |A_n(z)|\le Ch_n|z|,\qquad |z|\le R_n.
\end{align}
Now $h_n=O_\alpha(n^\alpha)$, and
\[
 n^\alpha G_2(n)^{-1/4}=
 \begin{cases}
 n^{\alpha-1/4},&0<\alpha<1/4,\\
 (\log n)^{-1/4},&\alpha=1/4,\\
 1,&1/4<\alpha<1/2
 \end{cases}
 =O_\alpha(1).
\]
Hence $|A_n(z)|\le C_\alpha c_\alpha$ on this disk. Reducing $c_\alpha$ ensures $|A_n(z)|\le1/2$ and therefore $|\beta A_n(z)|\le1/2$. Both factors in~\eqref{eq:decomp} then have analytic logarithms vanishing at zero. Writing $L_n(z)=\log(1+\beta A_n(z))$, we obtain
\begin{align}\label{pos:eq:log-decomp}
 \log M_n^{(q)}(z)=\log F_n(\cosh z)+L_n(z).
\end{align}

The inequality $|\log(1+w)|\le C|w|$ for $|w|\le1/2$, together with~\eqref{pos:eq:A-bound}, yields
\[
 \sup_{|z|=R_n}|L_n(z)|\le C|\beta|h_nR_n.
\]
If $L_n(z)=\sum_{r\ge1}(L_n^{(r)}(0)/r!)z^r$, Cauchy's estimates give
$|L_n^{(r)}(0)/r!|\le C|\beta|h_nR_n^{1-r}$. Thus, for $|z|\le R_n/2$,
\[
 \left|L_n(z)-\sum_{r=1}^4\frac{L_n^{(r)}(0)}{r!}z^r\right|
 \le C|\beta|h_nR_n\sum_{r=5}^\infty
       \left(\frac{|z|}{R_n}\right)^r
 \le C|\beta|h_nR_n^{-4}|z|^5.
\]
As $R_n^{-4}=c_\alpha^{-4}G_2(n)$, this proves
\begin{align}\label{pos:eq:L-rem}
 L_n(z)=\sum_{r=1}^4\frac{L_n^{(r)}(0)}{r!}z^r
 +O_\alpha\bigl(|\beta|n^\alpha G_2(n)|z|^5\bigr).
\end{align}

To expand $\log F_n(\cosh z)$, we use the product representation~\eqref{eq:product2}:
\[
 F_n(\cosh z)=(\cosh z)^{\delta_n}
 \prod_{j=1}^{[n/2]}(1+b_{n,j}\sinh^2z).
\]
Since $b_{n,j}^2\le\Theta_2(n)=O_\alpha(G_2(n))$ by Lemma \ref{pos:lem:rootbounds}, we have $b_{n,j}\le C_\alpha G_2(n)^{1/2}$. For $|z|\le R_n$,
\[
 |\sinh^2z|\le C|z|^2\le Cc_\alpha^2G_2(n)^{-1/2},
 \qquad b_{n,j}|\sinh^2z|\le C_\alpha c_\alpha^2.
\]
Reducing $c_\alpha$ further if necessary, we have
$b_{n,j}|\sinh^2z|\le1/2$ for all $n,j$ and $|z|\le R_n$.
Consequently,
\[
 \log(1+b_{n,j}\sinh^2z)
 =b_{n,j}\sinh^2z-\frac{b_{n,j}^2}{2}\sinh^4z
 +O(b_{n,j}^3|\sinh z|^6).
\]
Also $\delta_n\log\cosh z=(\delta_n/2)\log(1+\sinh^2z)$. Adding these expansions gives
\[
 \log F_n(\cosh z)
 =\Theta_1(n)\sinh^2z-\frac{\Theta_2(n)}2\sinh^4z
 +O(\Theta_3(n)|\sinh z|^6).
\]
By $b_{n,j}\ge1$, we have $\Theta_1(n)\le\Theta_2(n)\le\Theta_3(n)$. Using
\[
 \sinh^2z=z^2+\frac{z^4}{3}+O(|z|^6),\qquad
 \sinh^4z=z^4+O(|z|^6),
\]
and $\Theta_3(n)=O_\alpha(G_3(n))$ from Lemma~\ref{pos:lem:rootbounds}, we conclude that
\begin{align}\label{pos:eq:F-rem}
 \log F_n(\cosh z)
 =\Theta_1(n)z^2+
 \left(\frac{\Theta_1(n)}3-\frac{\Theta_2(n)}2\right)z^4
 +O_\alpha(G_3(n)|z|^6).
\end{align}
Combining~\eqref{pos:eq:log-decomp}, \eqref{pos:eq:L-rem},
and~\eqref{pos:eq:F-rem} gives the remainder estimate
in~\eqref{pos:eq:logrem}, while the Taylor coefficients
are identified by the definition of the cumulants. The proof is complete.
\end{proof}

\begin{lemma}\label{crit:lem:log}
Let $\alpha=1/2$. There is a constant $c>0$ such that, for $R_n=c/\sqrt n$, the logarithm of $M_n^{(q)}$ normalized to vanish at zero is analytic in a neighborhood of $|z|\le R_n$, uniformly for $q\in[0,1]$. For $|z|\le R_n/2$,
\begin{align}
 \log M_n^{(q)}(z)
 &=\mu_nz+\frac{v_n^2z^2}{2}
 +\frac{\kappa_3^{(q)}(n)z^3}{6}
 +\frac{\kappa_4^{(q)}(n)z^4}{24}\notag\\
 &\quad+O\bigl(n^3|z|^6+|\beta|n^{5/2}|z|^5\bigr).
 \label{crit:eq:log}
\end{align}
Moreover, $\kappa_3^{(q)}(n)=O(|\beta|n^{3/2})$ and $\kappa_4^{(q)}(n)=O(n^2)$.
\end{lemma}
\begin{proof}
Set $A_n(z)=\sinh z\,J_n(\cosh z)/F_n(\cosh z)$.
By Lemma~\ref{crit:lem:rootbounds}, $1-\rho_{n,j}\ge (Cn)^{-1}$.
For $|z|\le R_n$, choosing $c>0$ sufficiently small gives
\[
 |\cosh z-1|\le \frac{Cc^2}{n}
 \le\frac{1-\rho_{n,j}}2,
 \qquad b_{n,j}|\sinh^2z|\le\frac12.
\]
Thus $F_n(\cosh z)\ne0$, since the $\rho_{n,j}$ are all
the zeros of $F_n$. By Lemma~\ref{pos:lem:residues},
\[
 |A_n(z)|\le Ch_n|z|\le Cc,
\]
where $h_n=O(\sqrt n)$.
Reducing $c$ further ensures $|\beta A_n(z)|\le1/2$.
Hence the factors in~\eqref{eq:decomp} have analytic
logarithms vanishing at zero, and
\begin{align}\label{crit:eq:log-decomp}
 \log M_n^{(q)}(z)
 =\log F_n(\cosh z)+L_n(z),
 \qquad L_n(z)=\log(1+\beta A_n(z)).
\end{align}

Since $|L_n(z)|\le C|\beta|h_n|z|$,
Cauchy's estimates give
\begin{align}\label{crit:eq:L-coeff}
 \left|\frac{L_n^{(r)}(0)}{r!}\right|
 \le C|\beta|h_nR_n^{1-r}.
\end{align}
As in the proof of Lemma~\ref{pos:lem:log},
summing the Taylor tail yields
\begin{align}\label{crit:eq:L-rem}
 L_n(z)=\sum_{r=1}^4\frac{L_n^{(r)}(0)}{r!}z^r
 +O(|\beta|n^{5/2}|z|^5),
\end{align}
for $|z|\le R_n/2$.
Moreover, \eqref{crit:eq:L-coeff} gives
\begin{align}
 L_n^{(3)}(0)&=O(|\beta|n^{3/2}),
 \label{crit:eq:L-third}\\
 L_n^{(4)}(0)&=O(|\beta|n^2).
 \label{crit:eq:L-fourth}
\end{align}

By Lemma~\ref{crit:lem:rootbounds},
$b_{n,j}=O(n)$, while
$|\sinh^2z|\le Cc^2/n$ for $|z|\le R_n$.
Reducing $c$ further if necessary, we have
$b_{n,j}|\sinh^2z|\le1/2$ for all $n,j$
and $|z|\le R_n$.
The same argument as in the proof of
Lemma~\ref{pos:lem:log}, using~\eqref{eq:product2},
gives
\[
 \log F_n(\cosh z)
 =\Theta_1(n)\sinh^2z
 -\frac{\Theta_2(n)}2\sinh^4z
 +O(\Theta_3(n)|\sinh z|^6).
\]
Since $\Theta_1(n)\le\Theta_2(n)\le\Theta_3(n)=O(n^3)$ by Lemma \ref{crit:lem:rootbounds},
we obtain
\begin{align}\label{crit:eq:F-rem}
 \log F_n(\cosh z)
 =\Theta_1(n)z^2+
 \left(\frac{\Theta_1(n)}3-\frac{\Theta_2(n)}2\right)z^4
 +O(n^3|z|^6).
\end{align}
Since $\Theta_1(n)=n\mathcal{H}_n/2$
and $\Theta_2(n)=O(n^2)$ by Lemma \ref{crit:lem:rootbounds}, we have
\begin{align}\label{crit:eq:k4-sym-bound}
 \kappa_4^{(1/2)}(n)
 =8\Theta_1(n)-12\Theta_2(n)=O(n^2).
\end{align}

Combining~\eqref{crit:eq:log-decomp},
\eqref{crit:eq:L-rem}, and~\eqref{crit:eq:F-rem}
gives the remainder estimate in~\eqref{crit:eq:log},
while the Taylor coefficients are identified
by the definition of the cumulants.
Since \eqref{crit:eq:F-rem} contains no cubic term,
\eqref{crit:eq:L-third} gives
$\kappa_3^{(q)}(n)=O(|\beta|n^{3/2})$.
Similarly, \eqref{crit:eq:L-fourth} and
\eqref{crit:eq:k4-sym-bound} yield
\[
 \kappa_4^{(q)}(n)
 =\kappa_4^{(1/2)}(n)+O(|\beta|n^2)
 =O(n^2).
\]
The proof is complete.
\end{proof}
\begin{lemma}\label{lem:endpoint-log}
Let $\alpha\in\{-1,0\}$ and $n\ge2$.
There is an $R>0$, independent of $n$ and $q$,
such that $M_n^{(q)}$ has an analytic logarithm
in a neighborhood of $|z|\le R$, normalized to vanish
at the origin. For $|z|\le R/2$,
\begin{align}\label{eq:nonposlogsummary}
 \log M_n^{(q)}(z)
 &=\mu_nz+\frac{v_n^2z^2}{2}
 +\frac{\kappa_3^{(q)}(n)z^3}{6}
 +\frac{\kappa_4^{(q)}(n)z^4}{24}
 +O\bigl(n|z|^6+|\beta||z|^5\bigr).
\end{align}
Moreover, $\kappa_3^{(q)}(n)=O(|\beta|)$
and $\kappa_4^{(q)}(n)=O(n)$.

For $\alpha=0$, the cumulants satisfy
\begin{align}\label{eq:independentcums}
 \kappa_3^{(q)}(n)=-2\beta(1-\beta^2),\qquad
 \kappa_4^{(q)}(n)=-2n+8\beta^2-6\beta^4.
\end{align}
For $\alpha=-1$ and $n\ge2$, we have
\begin{align}\label{eq:endpoint}
 \mathbb{E}_q[S_n]=\kappa_3^{(q)}(n)=0.
\end{align}
\end{lemma}
\begin{proof}
For $\alpha=0$, \eqref{eq:independent} gives
\[
 \log M_n^{(q)}(z)
 =n\log\cosh z+\log(1+\beta\tanh z).
\]
Both terms are analytic on a fixed sufficiently small disk,
uniformly for $|\beta|\le1$.
Since $n\log\cosh z$ is even, its fourth-order Taylor
remainder is $O(n|z|^6)$.
Moreover, $|\log(1+\beta\tanh z)|\le C|\beta||z|$
on this disk, so Cauchy's estimates give a
fourth-order remainder $O(|\beta||z|^5)$. This proves~\eqref{eq:nonposlogsummary} for $\alpha=0$.
Since $\alpha=0$, the increments are independent, and
$X_2,\ldots,X_n$ are fair signs.
As cumulants are additive for independent random variables,
we compute the cumulants of $X_1$ and a fair sign separately.
Since $\mathbb{E}_q[X_1]=\beta$ and $X_1^2=1$, we have
\[
 \kappa_3(X_1)
 =\mathbb{E}_q[(X_1-\beta)^3]
 =-2\beta(1-\beta^2),
\]
and
\[
 \kappa_4(X_1)
 =\mathbb{E}_q[(X_1-\beta)^4]
 -3(1-\beta^2)^2
 =-2+8\beta^2-6\beta^4.
\]
On the other hand, the logarithm of the moment generating
function of a fair sign is
\[
 \log\cosh z=\frac{z^2}{2}-\frac{z^4}{12}+O(|z|^6),
\]
so its third and fourth cumulants are $0$ and $-2$,
respectively.
Consequently,
\[
 \kappa_3^{(q)}(n)=\kappa_3(X_1),\qquad
 \kappa_4^{(q)}(n)=\kappa_4(X_1)-2(n-1),
\]
which proves~\eqref{eq:independentcums}.

For $\alpha=-1$ and $n\ge2$, we have $X_2=-X_1$
and hence $S_2=0$.
Since the transition kernel in~\eqref{eq:transition}
is invariant under sign reversal, the distribution
of $S_n$ is symmetric and independent of $q$
for every $n\ge2$. This proves~\eqref{eq:endpoint}.
Moreover, the product representation in
Lemma \ref{lem:minus-one-product} gives
\[
 F_n(\cosh z)
 =(\cosh z)^{\delta_n}
 \prod_{j=1}^{[(n-2)/2]}(1+b_{n,j}\sinh^2z),
 \qquad 0<b_{n,j}<1.
\]
Since $M_n^{(q)}(z)=F_n(\cosh z)$, the same argument
as in the proof of Lemma~\ref{neg:lem:log}
shows that its normalized logarithm is analytic
on a fixed sufficiently small disk.
On this disk, uniformly in $n,j$,
\[
 \log(1+b_{n,j}\sinh^2z)
 =b_{n,j}z^2+
 \left(\frac{b_{n,j}}3-\frac{b_{n,j}^2}{2}\right)z^4
 +O(|z|^6).
\]
Together with
$\log\cosh z=z^2/2-z^4/12+O(|z|^6)$,
this yields
\[
 \begin{aligned}
 \log M_n^{(q)}(z)
 &=\left(\frac{\delta_n}{2}
 +\sum_{j=1}^{[(n-2)/2]} b_{n,j}\right)z^2\\
 &\quad+\left\{-\frac{\delta_n}{12}
 +\sum_{j=1}^{[(n-2)/2]}\left(\frac{b_{n,j}}3
 -\frac{b_{n,j}^2}{2}\right)\right\}z^4
 +O(n|z|^6).
 \end{aligned}
\]
Thus, \eqref{eq:nonposlogsummary} also holds for $\alpha=-1$. In particular, since $0<b_{n,j}<1$, we have
\[
 \kappa_4^{(q)}(n)
 =-2\delta_n+\sum_{j=1}^{[(n-2)/2]}(8b_{n,j}-12b_{n,j}^2)
 =O(n).
\]
The proof is complete.
\end{proof}

\section{Third and fourth cumulants}\label{sec:cumulants}
In this section, we determine the asymptotic behavior
of the third and fourth cumulants of $S_n$.
These results will be used to establish the optimal
rates of Gaussian approximation in the local and
central limit theorems.
\subsection{The third cumulant}
The limiting coefficients are
\begin{align*}
 \gamma_3(\alpha,\beta)
 &=\frac{\beta}{1-2\alpha}
 \left\{\frac{6\alpha}{\Gamma(1+\alpha)\Gamma(1+2\alpha)}
 -\frac{3(1+\alpha)}{\Gamma(1+3\alpha)}\right\}
 +\frac{2\beta^3}{\Gamma(1+\alpha)^3},
 \qquad 0<\alpha<1/2,\\
 \gamma_{3,c}(\beta)
 &=\frac{2\beta}{\sqrt\pi}
 \left(6\log2-8+\frac{8\beta^2}{\pi}\right).
\end{align*}
\begin{proposition}\label{prop:third}
For $0<\alpha<1/2$, uniformly in $q$,
\begin{align}\label{eq:k3limit}
 \kappa_3^{(q)}(n)=O_\alpha(|\beta|n^{3\alpha}),\qquad
 n^{-3\alpha}\kappa_3^{(q)}(n)\longrightarrow\gamma_3(\alpha,\beta).
\end{align}
For every $\beta\ne0$, $\gamma_3(\alpha,\beta)\ne0$.
At $\alpha=1/2$,
\begin{align*}
 n^{-3/2}\kappa_3^{(q)}(n)\longrightarrow\gamma_{3,c}(\beta),
\end{align*}
again uniformly in $q$.
For every $\beta\ne0$, $\gamma_{3,c}(\beta)\ne0$.
\end{proposition}

\begin{proof}
First let $0<\alpha<1/2$. The moments satisfy
$\mathbb{E}_q[S_n^2]=v_n^2+\mu_n^2$ and $\mathbb{E}_q[S_n^3]=\kappa_3^{(q)}(n)+3\mu_n\mathbb{E}_q[S_n^2]-2\mu_n^3$. 
Their recurrences \eqref{eq:m1-rec}--\eqref{eq:m3-rec} are
\begin{align*}
 \mu_{n+1}&=(1+\alpha/n)\mu_n,\\
 \mathbb{E}_q[S_{n+1}^2]&=(1+2\alpha/n)\mathbb{E}_q[S_n^2]+1,\\
 \mathbb{E}_q[S_{n+1}^3]&=(1+3\alpha/n)\mathbb{E}_q[S_n^3]+(3+\alpha/n)\mu_n.
\end{align*}
Substituting into $\kappa_3^{(q)}(n+1)=\mathbb{E}_q[S_{n+1}^3]-3\mu_{n+1}\mathbb{E}_q[S_{n+1}^2]+2\mu_{n+1}^3$ gives
\begin{align*}
 \kappa_3^{(q)}(n+1)
 &=(1+3\alpha/n)\{\kappa_3^{(q)}(n)+3\mu_n\mathbb{E}_q[S_n^2]-2\mu_n^3\}
 +(3+\alpha/n)\mu_n\\
 &\quad-3(1+\alpha/n)\mu_n\{(1+2\alpha/n)\mathbb{E}_q[S_n^2]+1\}
 +2(1+\alpha/n)^3\mu_n^3\\
 &=(1+3\alpha/n)\kappa_3^{(q)}(n)-\frac{2\alpha\mu_n}{n}
 -\frac{6\alpha^2\mu_n\mathbb{E}_q[S_n^2]}{n^2}
 +\left(\frac{6\alpha^2}{n^2}+\frac{2\alpha^3}{n^3}\right)\mu_n^3.
\end{align*}
Using $\mathbb{E}_q[S_n^2]=v_n^2+\mu_n^2$, we obtain
\begin{align}\label{eq:k3rec}
 \kappa_3^{(q)}(n+1)
 =\left(1+\frac{3\alpha}{n}\right)\kappa_3^{(q)}(n)
 -\frac{2\alpha\mu_n}{n}
 -\frac{6\alpha^2\mu_nv_n^2}{n^2}
 +\frac{2\alpha^3\mu_n^3}{n^3}.
\end{align}
Now $\mu_n=O_\alpha(|\beta|n^\alpha)$ and $v_n^2=O_\alpha(n)$ by \eqref{eq:m1_asym} and \eqref{vnasymp}. The first two forcing terms are therefore $O_\alpha(|\beta|n^{\alpha-1})$. The last is $O_\alpha(|\beta|^3n^{3\alpha-3})=O_\alpha(|\beta|n^{\alpha-1})$, since $|\beta|\le1$ and $3\alpha-3\le\alpha-1$. Dividing~\eqref{eq:k3rec} by $h_3(n+1)=(1+3\alpha/n)h_3(n)$ gives
\[
 \frac{\kappa_3^{(q)}(n+1)}{h_3(n+1)}
 -\frac{\kappa_3^{(q)}(n)}{h_3(n)}
 =O_\alpha\left(|\beta|\frac{n^{\alpha-1}}{h_3(n+1)}\right)
 =O_\alpha(|\beta|n^{-2\alpha-1}).
\]
Here $h_3(n+1)\asymp n^{3\alpha}$. Since $\sum_{j\ge1}j^{-2\alpha-1}<\infty$ and $|\kappa_3^{(q)}(1)|=2|\beta|(1-\beta^2)\le2|\beta|$, this proves the bound in~\eqref{eq:k3limit}.

To evaluate the limit, we use the following exact
third-moment identity, which is equivalent to
the formula obtained by Bercu~\cite{Bercu2018}:
\begin{align*}
 \mathbb{E}_q[S_n^3]=\frac{3n+1+\alpha}{1-2\alpha}\mu_n
 -\frac{3(1+\alpha)}{1-2\alpha}\beta h_3(n).
\end{align*}
Although derived there in the superdiffusive regime,
the identity also holds for $0<\alpha<1/2$. Combining it with $\mathbb{E}_q[S_n^2]=\{n-2\alpha h_2(n)\}/(1-2\alpha)$ from~\eqref{eq:meanvar}, we find
\begin{align}
 \kappa_3^{(q)}(n)
 &=\mathbb{E}_q[S_n^3]-3\mu_n\mathbb{E}_q[S_n^2]+2\mu_n^3\notag\\
 &=\frac{1+\alpha}{1-2\alpha}\mu_n
 +\frac{6\alpha}{1-2\alpha}\mu_nh_2(n)
 -\frac{3(1+\alpha)}{1-2\alpha}\beta h_3(n)+2\mu_n^3.
 \label{eq:k3exactsub}
\end{align}
Since $\mu_n=\beta h_1(n)$ and $h_j(n)/n^{j\alpha}\to1/\Gamma(1+j\alpha)$ for $j=1,2,3$,
\begin{align*}
 n^{-3\alpha}\mu_n&\longrightarrow0,&
 n^{-3\alpha}\mu_nh_2(n)&\longrightarrow
 \frac{\beta}{\Gamma(1+\alpha)\Gamma(1+2\alpha)},\\
 n^{-3\alpha}\beta h_3(n)&\longrightarrow\frac{\beta}{\Gamma(1+3\alpha)},&
 n^{-3\alpha}\mu_n^3&\longrightarrow\frac{\beta^3}{\Gamma(1+\alpha)^3}.
\end{align*}
Thus~\eqref{eq:k3exactsub} yields
\[
 n^{-3\alpha}\kappa_3^{(q)}(n)\longrightarrow
 \frac{\beta}{1-2\alpha}
 \left\{\frac{6\alpha}{\Gamma(1+\alpha)\Gamma(1+2\alpha)}
 -\frac{3(1+\alpha)}{\Gamma(1+3\alpha)}\right\}
 +\frac{2\beta^3}{\Gamma(1+\alpha)^3}
 =\gamma_3(\alpha,\beta).
\]
Each convergence is uniform for $|\beta|\le1$, proving the asserted uniformity.

To determine the sign, suppose first that $\beta>0$. Then $\mu_n>0$ and $\kappa_3^{(q)}(1)=-2\beta(1-\beta^2)\le0$. The forcing in~\eqref{eq:k3rec} is
\begin{align*}
 -\frac{2\alpha\mu_n}{n}
 -\frac{6\alpha^2\mu_nv_n^2}{n^2}
 +\frac{2\alpha^3\mu_n^3}{n^3}
 =-\frac{2\alpha\mu_n}{n}
 \left(1+\frac{3\alpha v_n^2}{n}-\frac{\alpha^2\mu_n^2}{n^2}\right)<0,
\end{align*}
because $v_n^2\ge0$, $|\mu_n|\le n$, and the expression in parentheses is at least $1-\alpha^2>0$. Iteration gives
\[
 \frac{\kappa_3^{(q)}(n)}{h_3(n)}
 =\kappa_3^{(q)}(1)
 +\sum_{j=1}^{n-1}\frac1{h_3(j+1)}
 \left\{-\frac{2\alpha\mu_j}{j}
 -\frac{6\alpha^2\mu_jv_j^2}{j^2}
 +\frac{2\alpha^3\mu_j^3}{j^3}\right\}.
\]
All denominators are positive and all summands are strictly negative. The preceding estimates show absolute convergence. If $0<\beta<1$, the initial term is strictly negative; if $\beta=1$, it is zero but the sum is strictly negative. In either case the limit is strictly negative. Since $h_3(n)/n^{3\alpha}\to1/\Gamma(1+3\alpha)>0$, it follows that $\gamma_3(\alpha,\beta)<0$. Reversing the initial sign reverses the law of $S_n$ and its third cumulant, so $\gamma_3(\alpha,\beta)>0$ for $\beta<0$.

Finally, let $\alpha=1/2$. Appendix~\ref{app:criticalmoments} gives
\begin{align}\label{eq:k3exactcritical}
 \kappa_3^{(q)}(n)
 =\mu_n\left\{3(2n+1)(\mathcal{H}_{2n}-\mathcal{H}_n)
 +\frac32\mathcal{H}_n-8n\right\}+2\mu_n^3.
\end{align}
Using $\mu_n/\sqrt n\to2\beta/\sqrt\pi$,
$\mathcal{H}_{2n}-\mathcal{H}_n\to\log2$,
and $\mathcal{H}_n/n\to0$, we obtain, uniformly for $|\beta|\le1$,
\[
 n^{-3/2}\kappa_3^{(q)}(n)
 \longrightarrow\frac{2\beta}{\sqrt\pi}
 \left(6\log2-8+\frac{8\beta^2}{\pi}\right)
 =\gamma_{3,c}(\beta).
\]
Since $6\log2-8+8/\pi<0$, we have
$\gamma_{3,c}(\beta)\ne0$ for every $\beta\ne0$.
\end{proof}

\subsection{The fourth cumulant}
The fourth cumulants for general and symmetric initial laws satisfy
\begin{align}\label{eq:k4identity}
 \kappa_4^{(q)}(n)=\kappa_4^{(1/2)}(n)
 -4\mu_n\kappa_3^{(q)}(n)+2\mu_n^4.
\end{align}
Indeed, the even raw moments are independent of $q$, and
\begin{align*}
 \kappa_4^{(q)}(n)
 &=\mathbb{E}_q[S_n^4]-4\mu_n\mathbb{E}_q[S_n^3]-3\bigl(\mathbb{E}_q[S_n^2]\bigr)^2
 +12\mu_n^2\mathbb{E}_q[S_n^2]-6\mu_n^4,\\
 \kappa_4^{(1/2)}(n)&=\mathbb{E}_{1/2}[S_n^4]-3\bigl(\mathbb{E}_{1/2}[S_n^2]\bigr)^2.
\end{align*}
Substitution of $\mathbb{E}_q[S_n^3]=\kappa_3^{(q)}(n)+3\mu_n\mathbb{E}_q[S_n^2]-2\mu_n^3$ proves~\eqref{eq:k4identity}.

For $1/4<\alpha<1/2$, define
\[
 D_\alpha\coloneqq \frac1{\Gamma(1+4\alpha)}
 \left\{\frac12+\sum_{j=1}^{\infty}
 \frac{\frac12+4\alpha^2\Theta_1(j)^2/j^2}{h_4(j+1)}\right\}.
\]
The proof below shows that this series converges and that $0<D_\alpha<\infty$.

\begin{proposition}\label{prop:fourth}
For every fixed $-1\le\alpha<1/4$,
\begin{align}\label{eq:k4slope}
 \frac{\kappa_4^{(1/2)}(n)}n
 \longrightarrow-\frac{2(1+2\alpha^2)}{(1-2\alpha)^2(1-4\alpha)}.
\end{align}
At $\alpha=1/4$, $\kappa_4^{(1/2)}(n)\sim-9n\log n$. For $1/4<\alpha<1/2$, $\kappa_4^{(1/2)}(n)\sim-12D_\alpha n^{4\alpha}$. At the critical point $\alpha=1/2$,
\begin{align}
 \kappa_4^{(1/2)}(n)
 &=-n\{3(n+1)\mathcal{H}_n^{(2)}-3\mathcal{H}_n^2-10\mathcal{H}_n+9n\}\notag\\
 &=-\left(9+\frac{\pi^2}{2}\right)n^2+O(n\mathcal{H}_n^2).
 \label{eq:k4critical}
\end{align}
For general $q\in[0,1]$, uniformly in $q$,
\[
 \kappa_4^{(q)}(n)=
 \begin{cases}
 O_\alpha(G_2(n)),&0<\alpha<1/2,\\
 O(n^2),&\alpha=1/2.
 \end{cases}
\]
For $-1\le\alpha<1/4$, the limit in~\eqref{eq:k4slope} also holds with $\kappa_4^{(q)}(n)/n$ on the left, uniformly in $q\in[0,1]$.
\end{proposition}

\begin{proof}
First consider the case $q=1/2$ with $-1<\alpha<1/4$, initially excluding $\alpha=0$. By~\eqref{eq:cumroots},
\begin{align}\label{eq:k4theta}
 \kappa_4^{(1/2)}(n)=8\Theta_1(n)-12\Theta_2(n).
\end{align}
Also $\Theta_1(n)=\sigma_n^2/2=v_n^2/2$ for $q=1/2$,
so~\eqref{vnasymp} gives
\begin{align*}
 \frac{\Theta_1(n)}n\longrightarrow\frac1{2(1-2\alpha)}.
\end{align*}
Recurrence~\eqref{eq:theta2} reads
\begin{align}\label{eq:theta2again}
 \Theta_2(n+1)=\left(1+\frac{4\alpha}{n}\right)\Theta_2(n)
 +\frac12+\frac{4\alpha^2\Theta_1(n)^2}{n^2}.
\end{align}
Its forcing converges to $1/2+\alpha^2/(1-2\alpha)^2$. Since $4\alpha<1$, Lemma~\ref{lem:difference} yields
\[
 \frac{\Theta_2(n)}n\longrightarrow
 \frac{\frac12+\frac{\alpha^2}{(1-2\alpha)^2}}{1-4\alpha}
 =\frac{6\alpha^2-4\alpha+1}{2(1-2\alpha)^2(1-4\alpha)}.
\]
Substituting these limits into~\eqref{eq:k4theta}, we obtain
\[
 \frac{\kappa_4^{(1/2)}(n)}n
 \longrightarrow
 \frac4{1-2\alpha}
 -\frac{6(6\alpha^2-4\alpha+1)}{(1-2\alpha)^2(1-4\alpha)}
 =-\frac{2(1+2\alpha^2)}{(1-2\alpha)^2(1-4\alpha)},
\]
which proves~\eqref{eq:k4slope}. At $\alpha=0$, the same conclusion follows directly from~\eqref{eq:independentcums}, since $\kappa_4^{(1/2)}(n)=-2n$.

At $\alpha=-1$, we have $S_2=0$ and $h_2(n)=0$ for $n\ge3$.
Hence~\eqref{eq:meanvar} gives
$\mathbb E_{1/2}[S_n^2]=n/3$ for $n\ge3$. Recurrence~\eqref{eq:m4-rec} then gives
\[
 \kappa_4^{(1/2)}(n+1)=\left(1-\frac4n\right)\kappa_4^{(1/2)}(n)-\frac23,
 \qquad n\ge3.
\]
Taking $n=4$ gives $\kappa_4^{(1/2)}(5)=-2/3$. Induction therefore yields $\kappa_4^{(1/2)}(n)=-2n/15$ for $n\ge5$. Its slope $-2/15$ agrees with~\eqref{eq:k4slope} at $\alpha=-1$.

At $\alpha=1/4$, equation~\eqref{eq:meanvar} gives
$\Theta_1(n)=n-h_2(n)/2=n+O(n^{1/2})$.
Hence~\eqref{eq:theta2again} yields
\[
 \frac{\Theta_2(n+1)}{n+1}-\frac{\Theta_2(n)}n
 =\frac{\frac12+\frac{\Theta_1(n)^2}{4n^2}}{n+1}
 =\frac3{4(n+1)}+O(n^{-3/2}).
\]
Summing gives $\Theta_2(n)\sim(3/4)n\log n$.
Since $\Theta_1(n)=O(n)$, equation~\eqref{eq:k4theta} yields
$\kappa_4^{(1/2)}(n)\sim-9n\log n$.

For $1/4<\alpha<1/2$, divide~\eqref{eq:theta2again} by $h_4(n+1)=(1+4\alpha/n)h_4(n)$ and iterate:
\begin{align*}
 \frac{\Theta_2(n)}{h_4(n)}
 =\frac12+\sum_{j=1}^{n-1}
 \frac{\frac12+4\alpha^2\Theta_1(j)^2/j^2}{h_4(j+1)}.
\end{align*}
Here $\Theta_1(j)=\sigma_j^2/2=O_\alpha(j)$
by~\eqref{vnasymp}, so the numerator is $O_\alpha(1)$. Also $h_4(j+1)\asymp j^{4\alpha}$ and $4\alpha>1$. Thus the series converges; all summands are positive, so its limit is positive and finite. As $h_4(n)/n^{4\alpha}\to1/\Gamma(1+4\alpha)$, the definition of $D_\alpha$ gives $\Theta_2(n)/n^{4\alpha}\to D_\alpha$. Since $\Theta_1(n)=O_\alpha(n)=o(n^{4\alpha})$, equation~\eqref{eq:k4theta} yields $\kappa_4^{(1/2)}(n)\sim-12D_\alpha n^{4\alpha}$.

We next consider the case of general $q\in[0,1]$. For $0<\alpha<1/2$, Proposition~\ref{prop:third} and \eqref{eq:m1_asym} give $\kappa_3^{(q)}(n)=O_\alpha(|\beta|n^{3\alpha})$ and $\mu_n=O_\alpha(|\beta|n^\alpha)$. Hence~\eqref{eq:k4identity} implies
\begin{align}
 \kappa_4^{(q)}(n)-\kappa_4^{(1/2)}(n)
 &=-4\mu_n\kappa_3^{(q)}(n)+2\mu_n^4\notag\\
 &=O_\alpha(|\beta|^2n^{4\alpha})
 +O_\alpha(|\beta|^4n^{4\alpha})
 =O_\alpha(n^{4\alpha}).\label{eq:k4difference}
\end{align}
The estimates for $q=1/2$ obtained above
and~\eqref{eq:k4difference} imply, uniformly in $q\in[0,1]$,
\[
 \kappa_4^{(q)}(n)
 =\kappa_4^{(1/2)}(n)+O_\alpha(n^{4\alpha})
 =O_\alpha(G_2(n)).
\]
For $0<\alpha<1/4$, we also have
\[
 \frac{\kappa_4^{(q)}(n)-\kappa_4^{(1/2)}(n)}n
 =O_\alpha(n^{4\alpha-1})\longrightarrow0.
\]
Hence~\eqref{eq:k4slope} also holds with
$\kappa_4^{(q)}(n)$ in place of $\kappa_4^{(1/2)}(n)$,
uniformly in $q$.

For $-1<\alpha<0$, equation~\eqref{eq:kappa4} gives
\[
 \kappa_4^{(q)}(n)-\kappa_4^{(1/2)}(n)
 =O_p(|\beta|h_n)=o(n),
\]
since $h_n\sim n^\alpha/\Gamma(1+\alpha)$ and $\alpha<0$. At $\alpha=0$, equation~\eqref{eq:independentcums} gives a difference $8\beta^2-6\beta^4=O(1)=o(n)$. At $\alpha=-1$, the law of $S_n$ is independent of $q$ for $n\ge2$, so
$\kappa_4^{(q)}(n)=\kappa_4^{(1/2)}(n)$ for $n\ge2$. Thus, for $-1\le\alpha<1/4$, the limit in~\eqref{eq:k4slope}
also holds with $\kappa_4^{(q)}(n)$ in place of
$\kappa_4^{(1/2)}(n)$, uniformly in $q\in[0,1]$.

Finally, let $\alpha=1/2$. The first equality
in~\eqref{eq:k4critical} follows from
\eqref{eq:crit-kappa4} in Appendix~\ref{app:criticalmoments}.
Since $\mathcal{H}_n^{(2)}=\pi^2/6+O(n^{-1})$ and
$\mathcal{H}_n=O(\log n)$, we obtain
\[
 \kappa_4^{(1/2)}(n)
 =-3n(n+1)\mathcal{H}_n^{(2)}-9n^2
 +O(n\mathcal{H}_n^2)
 =-\left(9+\frac{\pi^2}{2}\right)n^2
 +O(n\mathcal{H}_n^2),
\]
which proves the second equality in~\eqref{eq:k4critical}. Proposition~\ref{prop:third} gives $\kappa_3^{(q)}(n)=O(n^{3/2})$, while $\mu_n=O(\sqrt n)$ by \eqref{eq:m1_asym}. Therefore~\eqref{eq:k4identity} gives
\[
 \kappa_4^{(q)}(n)-\kappa_4^{(1/2)}(n)
 =-4\mu_n\kappa_3^{(q)}(n)+2\mu_n^4=O(n^2).
\]
Together with~\eqref{eq:k4critical}, this proves $\kappa_4^{(q)}(n)=O(n^2)$.
\end{proof}

\section{Local Edgeworth expansions and optimal rates of local CLT}\label{sec:inversion}
In this section, we establish the local Edgeworth expansion
via Fourier inversion and use it to derive the optimal rates
of local Gaussian approximation under exact and asymptotic
normalization. We also explain the mechanisms behind
the resulting phase transitions.

We combine the analytic and Fourier estimates in the preceding sections to prove Theorem~\ref{thm:edgeworth}.

\begin{proof}[Proof of Theorem~\ref{thm:edgeworth}]
Set
\[
 f_n(\xi)=e^{-\mathrm{i}\mu_n\xi/v_n}\chi_n^{(q)}(\xi/v_n),
\]
and set
\[
 a_n^*=\frac{\kappa_3^{(q)}(n)}{6v_n^3},\qquad
 b_n^*=\frac{\kappa_4^{(q)}(n)}{24v_n^4}.
\]
Fix a sufficiently large constant $M>0$, and set
\begin{align*}
 T_n=\begin{cases}\sqrt{M\log n},&\alpha<1/2,\\
 \sqrt{M\log\mathcal{H}_n},&\alpha=1/2.
 \end{cases}
\end{align*}

We first check that the logarithmic expansions of Section~\ref{sec:analytic} apply on $|\xi|\le T_n$. If $-1<\alpha<0$, then $v_n^2\asymp n$ by \eqref{vnasymp}, so
\[
 \frac{T_n}{v_n}=O_p\left(\sqrt{\frac{\log n}{n}}\right)\longrightarrow0.
\]
Thus $|\mathrm{i}\xi/v_n|\le T_n/v_n\le R/2$ for all sufficiently large $n$, and Lemma~\ref{neg:lem:log} applies. For $\alpha\in\{-1,0\}$, the expansion~\eqref{eq:nonposlogsummary}
also applies at $z=\mathrm{i}\xi/v_n$ for $|\xi|\le T_n$
and all sufficiently large $n$.

For $0<\alpha<1/2$, the radius is $R_n=c_\alpha G_2(n)^{-1/4}$. Since $v_n^2\asymp n$ by \eqref{vnasymp},
\[
 \frac{T_n}{v_nR_n}=O_\alpha\left(\sqrt{\log n}\,n^{-1/2}G_2(n)^{1/4}\right).
\]
Here
\[
 G_2(n)=\begin{cases}n,&0<\alpha<1/4,\\
 n\log n,&\alpha=1/4,\\n^{4\alpha},&1/4<\alpha<1/2.
 \end{cases}
\]
Each case gives $T_n/(v_nR_n)\to0$, hence $|\xi|/v_n\le R_n/2$ for sufficiently large $n$. At $\alpha=1/2$, the radius $R_n=c/\sqrt n$ and
the estimate $v_n^2\asymp n\mathcal{H}_n$
from~\eqref{vnasymp} give
\[
 \frac{T_n}{v_nR_n}
 =O\left(\sqrt{\frac{\log\mathcal{H}_n}{\mathcal{H}_n}}\right)
 \longrightarrow0.
\]
Consequently, in every parameter regime the logarithmic expansion applies on $|\xi|\le T_n$, uniformly in $q$:
\begin{align}\label{eq:scaledlog}
 \log f_n(\xi)=-\frac{\xi^2}{2}+a_n^*(\mathrm{i}\xi)^3+b_n^*(\mathrm{i}\xi)^4+W_n(\xi),
\end{align}
where $W_n(\xi)$ denotes the remainder term. Indeed,
\[
 \log f_n(\xi)=-\frac{\mathrm{i}\mu_n\xi}{v_n}+\log M_n^{(q)}(\mathrm{i}\xi/v_n);
\]
the linear terms cancel, and the quadratic term is $(v_n^2/2)(\mathrm{i}\xi/v_n)^2=-\xi^2/2$.

The remainder estimates of Section~\ref{sec:analytic}, together with~\eqref{vnasymp}, which gives
$v_n^2\asymp n$ for $0\le p<3/4$
and $v_n^2\asymp n\mathcal{H}_n$ for $p=3/4$, imply
\begin{align}\label{eq:Wbounds}
 |W_n(\xi)|\le C_p\begin{cases}
 n^{-2}|\xi|^6+|\beta|n^{-5/2}|\xi|^5,&-1\le\alpha\le0,\\
 G_3(n)n^{-3}|\xi|^6+|\beta|n^{\alpha-5/2}G_2(n)|\xi|^5,&0<\alpha<1/2,\\
 \mathcal{H}_n^{-3}|\xi|^6+|\beta|\mathcal{H}_n^{-5/2}|\xi|^5,&\alpha=1/2.
 \end{cases}
\end{align}
Propositions~\ref{prop:third} and~\ref{prop:fourth},
together with Lemma~\ref{neg:lem:log},
\eqref{eq:independentcums}, and \eqref{eq:endpoint}, give
\begin{align}
 |a_n^*|&\le C_p\begin{cases}
 |\beta|n^{-3/2},&-1\le\alpha\le0,\\
 |\beta|n^{3\alpha-3/2},&0<\alpha<1/2,\\
 |\beta|\mathcal{H}_n^{-3/2},&\alpha=1/2,
 \end{cases}\label{eq:anbounds}\\
 |b_n^*|&\le C_p\begin{cases}
 n^{-1},&-1\le\alpha\le0,\\
 G_2(n)n^{-2},&0<\alpha<1/2,\\
 \mathcal{H}_n^{-2},&\alpha=1/2.
 \end{cases}\label{eq:bnbounds}
\end{align}
These estimates yield
\begin{align}\label{eq:correctionsmall}
 \sup_{|\xi|\le T_n}\bigl(|a_n^*||\xi|^3+|b_n^*||\xi|^4+|W_n(\xi)|\bigr)\longrightarrow0
\end{align}
uniformly in $q$.

We also have
\begin{align}\label{eq:squares}
 |a_n^*|^2+|b_n^*|^2\le C_p\mathcal R_n(\alpha,\beta).
\end{align}
Only $0<\alpha<1/2$ requires further verification. First,
\[
 |a_n^*|^2\le C_\alpha|\beta|^2n^{6\alpha-3}
 \le C_\alpha n^{6\alpha-3}=O_\alpha(G_3(n)n^{-3}).
\]
Next, $|b_n^*|^2\le C_\alpha G_2(n)^2/n^4$. For $0<\alpha<1/4$,
\[
 \frac{G_2(n)^2}{n^4}=n^{-2}=O_\alpha(G_3(n)n^{-3}).
\]
At $\alpha=1/4$ this ratio is
\[
 n^{-2}(\log n)^2=o(n^{-3/2})=o(G_3(n)n^{-3}),
\]
whereas for $1/4<\alpha<1/2$,
\[
 \frac{G_2(n)^2/n^4}{G_3(n)n^{-3}}=n^{2\alpha-1}\longrightarrow0.
\]
This proves \eqref{eq:squares}; in the remaining regimes it follows directly from \eqref{eq:anbounds}--\eqref{eq:bnbounds}.

By~\eqref{eq:correctionsmall}, for all sufficiently large $n$,
\[
 |a_n^*||\xi|^3+|b_n^*||\xi|^4+|W_n(\xi)|\le1,
 \qquad |\xi|\le T_n.
\] For complex $A,B,W$ with $|A|+|B|+|W|\le1$, Taylor's formula gives
\begin{align*}
 |e^{A+B+W}-1-A-B|\le C\bigl(|W|+|A|^2+|B|^2\bigr).
\end{align*}
Thus \eqref{eq:scaledlog} implies
\[
 \left|f_n(\xi)-e^{-\xi^2/2}\{1+a_n^*(\mathrm{i}\xi)^3+b_n^*(\mathrm{i}\xi)^4\}\right|
 \le Ce^{-\xi^2/2}\bigl\{|W_n(\xi)|+|a_n^*|^2|\xi|^6+|b_n^*|^2|\xi|^8\bigr\}.
\]
Using \eqref{eq:Wbounds} and \eqref{eq:squares}, we obtain, for $|\xi|\le T_n$,
\begin{align}\label{ineq-Gauss}
 \left|f_n(\xi)-e^{-\xi^2/2}\{1+a_n^*(\mathrm{i}\xi)^3+b_n^*(\mathrm{i}\xi)^4\}\right|
 \le C_p\mathcal R_n(\alpha,\beta)(|\xi|^5+|\xi|^6+|\xi|^8)e^{-\xi^2/2}.
\end{align}
The right-hand side is integrable on $\R$, so
\begin{align}\label{eq:centralintegral}
 \int_{-T_n}^{T_n}\left|f_n(\xi)-e^{-\xi^2/2}\{1+a_n^*(\mathrm{i}\xi)^3+b_n^*(\mathrm{i}\xi)^4\}\right|\,d\xi
 =O_p\bigl(\mathcal R_n(\alpha,\beta)\bigr).
\end{align}

We next estimate the Fourier integral over
$T_n<|\xi|\le\pi v_n/2$. For $-1<\alpha<1/2$, Lemma~\ref{lem:fourier} gives $|\chi_n^{(q)}(\theta)|\le3e^{-c_pn\sin^2\theta}$ on $|\theta|\le\pi/2$. Since $\sin^2 y\ge4y^2/\pi^2$ on this interval and $v_n^2\asymp n$ by \eqref{vnasymp},
\begin{align*}
 |f_n(\xi)|=|\chi_n^{(q)}(\xi/v_n)|
 \le3e^{-c_pn\sin^2(\xi/v_n)}\le3e^{-c_p'\xi^2},
 \qquad |\xi|\le\pi v_n/2.
\end{align*}
For $\alpha=-1$, we have $\mu_n=0$ and hence
$v_n^2=\sigma_n^2$. Thus Lemma~\ref{lem:antifourier}
gives
\[
 |f_n(\xi)|\le e^{-2\xi^2/\pi^2},
 \qquad |\xi|\le\pi v_n/2,
\]
uniformly in $q\in[0,1]$. Consequently,
\[
 \int_{T_n<|\xi|\le\pi v_n/2}|f_n(\xi)|\,d\xi
 \le C_p\int_{|\xi|>T_n}e^{-c_p'\xi^2}\,d\xi.
\]
Choosing $M$ sufficiently large in $T_n^2=M\log n$ yields
\begin{align}\label{eq:subcriticaltail}
 \int_{T_n<|\xi|\le\pi v_n/2}|f_n(\xi)|\,d\xi
 =O_p(n^{-3})=O_p\bigl(\mathcal R_n(\alpha,\beta)\bigr).
\end{align}

At $\alpha=1/2$, let $\theta_n=\arcsin\sqrt{\eta\mathcal{H}_n/n}$, as in Lemma~\ref{crit:lem:outer}. Lemma~\ref{crit:lem:window} gives $|f_n(\xi)|\le Ce^{-c\xi^2}$ for $|\xi|\le v_n\theta_n$. A sufficiently large $M$ therefore gives
\begin{align}\label{eq:criticalmiddletail}
 \int_{T_n<|\xi|\le v_n\theta_n}|f_n(\xi)|\,d\xi=O(\mathcal{H}_n^{-4}).
\end{align}
This interval is well defined for large $n$, since $(v_n\theta_n)^2\asymp\mathcal{H}_n^2$ whereas $T_n^2=M\log\mathcal{H}_n$, so $T_n<v_n\theta_n$. Lemma~\ref{crit:lem:outer} supplies $d>0$ such that
\[
 \int_{v_n\theta_n\le|\xi|\le\pi v_n/2}|f_n(\xi)|\,d\xi=O(n^{-d}).
\]
As $n^{-d}=o(\mathcal{H}_n^{-m})$ for every fixed $m>0$, this implies
\begin{align}\label{eq:criticaloutertail}
 \int_{v_n\theta_n\le|\xi|\le\pi v_n/2}|f_n(\xi)|\,d\xi=O(\mathcal{H}_n^{-4}).
\end{align}
Both critical tail bounds are $O(\mathcal R_n(\alpha,\beta))$.

The approximating Fourier transform has a negligible tail as well. Indeed, \eqref{eq:anbounds}--\eqref{eq:bnbounds} bound $a_n^*$ and $b_n^*$ uniformly, and for every fixed nonnegative integer $k$,
\[
 \int_{|\xi|>T_n}|\xi|^k e^{-\xi^2/2}\,d\xi\le C_k e^{-T_n^2/4}.
\]
Increasing $M$ if necessary gives
\begin{align}\label{eq:gaussiantail}
 \int_{|\xi|>T_n}e^{-\xi^2/2}|1+a_n^*(\mathrm{i}\xi)^3+b_n^*(\mathrm{i}\xi)^4|\,d\xi
 =O_p\bigl(\mathcal R_n(\alpha,\beta)\bigr).
\end{align}

Finally, put $\xi=v_n\theta$ and $x=(s-\mu_n)/v_n$ in \eqref{eq:inversion}. For $s\equiv n\pmod2$,
\begin{align*}
 \frac{v_n}{2}\mathbb{P}_q(S_n=s)=\frac1{2\pi}
 \int_{-\pi v_n/2}^{\pi v_n/2}f_n(\xi)e^{-\mathrm{i} x\xi}\,d\xi.
\end{align*}
Since $|e^{-\mathrm{i} x\xi}|=1$, the difference between this integral and the approximating integral over $\R$ is bounded by
\begin{align*}
 &\left|\int_{-\pi v_n/2}^{\pi v_n/2}f_n(\xi)e^{-\mathrm{i} x\xi}\,d\xi
 -\int_\R e^{-\xi^2/2}\{1+a_n^*(\mathrm{i}\xi)^3+b_n^*(\mathrm{i}\xi)^4\}e^{-\mathrm{i} x\xi}\,d\xi\right|\\
 &\quad\le\int_{|\xi|\le T_n}\left|f_n(\xi)-e^{-\xi^2/2}\{1+a_n^*(\mathrm{i}\xi)^3+b_n^*(\mathrm{i}\xi)^4\}\right|\,d\xi\\
 &\qquad+\int_{T_n<|\xi|\le\pi v_n/2}|f_n(\xi)|\,d\xi
 +\int_{|\xi|>T_n}e^{-\xi^2/2}|1+a_n^*(\mathrm{i}\xi)^3+b_n^*(\mathrm{i}\xi)^4|\,d\xi\\
 &\quad=O_p\bigl(\mathcal R_n(\alpha,\beta)\bigr).
\end{align*}
Here we used \eqref{eq:centralintegral} and
\eqref{eq:gaussiantail}, together with
\eqref{eq:subcriticaltail} for $0\le p<3/4$
and \eqref{eq:criticalmiddletail}--\eqref{eq:criticaloutertail}
for $p=3/4$. Hence, uniformly over the admissible lattice,
\begin{align}\label{eq:edgeworth-fourier}
 \frac{v_n}{2}\mathbb{P}_q(S_n=s)=\frac1{2\pi}\int_\R e^{-\xi^2/2}
 \{1+a_n^*(\mathrm{i}\xi)^3+b_n^*(\mathrm{i}\xi)^4\}e^{-\mathrm{i} x\xi}\,d\xi
 +O_p\bigl(\mathcal R_n(\alpha,\beta)\bigr).
\end{align}
Differentiating the Gaussian Fourier transform $j$ times under the integral, for every nonnegative integer $j$, gives
\begin{align*}
 \frac1{2\pi}\int_\R(\mathrm{i}\xi)^j e^{-\xi^2/2}e^{-\mathrm{i} x\xi}\,d\xi
 =(-1)^j\varphi^{(j)}(x)=\varphi(x)H_j(x).
\end{align*}
Applying this identity for $j=0,3,4$ in \eqref{eq:edgeworth-fourier} yields
\[
 \frac{v_n}{2}\mathbb{P}_q(S_n=s)=\varphi(x)\left\{1+
 \frac{\kappa_3^{(q)}(n)}{6v_n^3}H_3(x)+\frac{\kappa_4^{(q)}(n)}{24v_n^4}H_4(x)\right\}
 +O_p\bigl(\mathcal R_n(\alpha,\beta)\bigr),
\]
which is \eqref{eq:edgeworth}. The modulus-one Fourier multiplier makes the error uniform in $s\equiv n\pmod2$.

Moreover, \eqref{eq:anbounds} and \eqref{eq:bnbounds} imply
\[
 \sup_{q\in[0,1]}\left|\frac{\kappa_3^{(q)}(n)}{v_n^3}\right|\longrightarrow0,
 \qquad\sup_{q\in[0,1]}\left|\frac{\kappa_4^{(q)}(n)}{v_n^4}\right|\longrightarrow0.
\]
Also $\sup_{q\in[0,1]}\mathcal R_n(\alpha,\beta)\to0$, since $|\beta|\le1$. As $\varphi H_3$ and $\varphi H_4$ are bounded on $\R$, \eqref{eq:edgeworth} gives
\begin{align*}
 \sup_{q\in[0,1]}\sup_{\substack{s\in\Z\\s\equiv n\!\!\!\pmod2}}
 \left|\frac{v_n}{2}\mathbb{P}_q(S_n=s)-\varphi\left(\frac{s-\mu_n}{v_n}\right)\right|\longrightarrow0.
\end{align*}
This is the local CLT, uniform over the entire admissible lattice and all initial probabilities $q\in[0,1]$.
\end{proof}

We record the constants used below. For $0\le p<5/8$, put
\begin{align}\label{eq:Cp}
 C(p)=\frac{1+2\alpha^2}{4(1-4\alpha)\sqrt{2\pi}},\qquad \alpha=2p-1.
\end{align}
For $0<\alpha<1/2$, set
\[
 a_3(\alpha,\beta)=\frac{(1-2\alpha)^{3/2}}6\gamma_3(\alpha,\beta),
\]
where $\gamma_3$ is given in Proposition~\ref{prop:third}.

To obtain the optimal estimates, we use the following elementary lemma on maxima over the standardized lattice.
\begin{lemma}\label{lem:profiles}
Let $g:\R\to\R$ be continuous with $g(x)\to0$ as $|x|\to\infty$, and let $\delta_n>0$ tend to zero. Then, for every real sequence $(a_n)$,
\begin{align}\label{eq:latticeprofile}
 \sup_{x\in a_n+\delta_n\Z}|g(x)|\longrightarrow\sup_{x\in\R}|g(x)|,
\end{align}
uniformly in $a_n$. In particular,
\begin{align}\label{eq:profilemax}
 \sup_{x\in\R}|\varphi(x)H_4(x)|=\frac3{\sqrt{2\pi}}.
\end{align}
\end{lemma}
\begin{proof}
The lattice supremum never exceeds the supremum over $\R$. The assertion is immediate if $g\equiv0$. Otherwise, decay at infinity implies that $|g|$ attains its maximum at some $x_0$. Every lattice $a_n+\delta_n\Z$ contains a point $x_n$ with $|x_n-x_0|\le\delta_n/2$. Since $g$ is uniformly continuous, for each $\varepsilon>0$ there is an $\eta>0$, independent of $a_n$, such that $|g(x)-g(y)|<\varepsilon$ whenever $|x-y|<\eta$. For $\delta_n/2<\eta$,
\[
 \sup_{x\in a_n+\delta_n\Z}|g(x)|\ge |g(x_n)|
 \ge |g(x_0)|-\varepsilon=\sup_{x\in\R}|g(x)|-\varepsilon.
\]
This proves~\eqref{eq:latticeprofile}, including uniformity under translation.

For~\eqref{eq:profilemax}, use $(\varphi H_j)'=-\varphi H_{j+1}$. Since $H_5(x)=x(x^4-10x^2+15)$, the stationary points of $\varphi H_4$ satisfy $x=0$ or $x^2=5\pm\sqrt{10}$. At zero its absolute value is $3/\sqrt{2\pi}$. At a nonzero stationary point, writing $y=x^2$ gives $y^2-10y+15=0$, hence $H_4(x)=y^2-6y+3=4(y-3)$. The corresponding value, multiplied by $\sqrt{2\pi}$, is $4|y-3|e^{-y/2}<3$ for both $y=5\pm\sqrt{10}$. Since $\varphi H_4$ tends to zero at infinity, its maximum absolute value is as stated.
\end{proof}

We now prove Theorem~\ref{thm:localrate}.
\begin{proof}[Proof of Theorem~\ref{thm:localrate}]
For $q=1/2$, we have $\beta=0$, $\mu_n=0$, $v_n=\sigma_n$, and $\kappa_3^{(1/2)}(n)=0$. Theorem~\ref{thm:edgeworth} gives, uniformly in $s\equiv n\pmod2$,
\[
 \frac{\sigma_n}{2}\mathbb{P}_{1/2}(S_n=s)
 =\varphi(s/\sigma_n)\left\{1+\frac{\kappa_4^{(1/2)}(n)}{24\sigma_n^4}H_4(s/\sigma_n)\right\}
 +O_p(\mathcal R_n(\alpha,0)).
\]
Consequently,
\begin{align}\label{eq:Esymprofile}
 E_n(p,1/2)=\frac{|\kappa_4^{(1/2)}(n)|}{24\sigma_n^4}
 \sup_{\substack{s\in\Z\\ s\equiv n\pmod 2}}|\varphi(s/\sigma_n)H_4(s/\sigma_n)|
 +O_p(\mathcal R_n(\alpha,0)).
\end{align}
Proposition~\ref{prop:fourth} and~\eqref{vnasymp} imply
\begin{align}\label{eq:standardk4}
 \frac{\kappa_4^{(1/2)}(n)}{\sigma_n^4}\sim
 \begin{cases}
 -\dfrac{2(1+2\alpha^2)}{1-4\alpha}n^{-1},&-1\le\alpha<1/4,\\[2mm]
 -\dfrac94\dfrac{\log n}{n},&\alpha=1/4,\\[2mm]
 -12D_\alpha(1-2\alpha)^2 n^{4\alpha-2},&1/4<\alpha<1/2,\\[2mm]
 -\left(9+\dfrac{\pi^2}{2}\right)\mathcal{H}_n^{-2},&\alpha=1/2.
 \end{cases}
\end{align}
In every case,
\begin{align}\label{eq:symremainder}
 \mathcal R_n(\alpha,0)=o\left(\frac{|\kappa_4^{(1/2)}(n)|}{\sigma_n^4}\right).
\end{align}
Indeed, for $-1\le\alpha<1/6$ the remainder is $n^{-2}=o(n^{-1})$; at $\alpha=1/6$ it is $n^{-2}\log n=o(n^{-1})$; and for $1/6<\alpha<1/4$ it is $n^{6\alpha-3}=o(n^{-1})$, since $6\alpha-3<-1$. At $\alpha=1/4$ it is $n^{-3/2}=o(n^{-1}\log n)$. For $1/4<\alpha<1/2$, its ratio to $n^{4\alpha-2}$ is $n^{2\alpha-1}\to0$. Finally, at criticality its ratio to $\mathcal{H}_n^{-2}$ is $\mathcal{H}_n^{-1}\to0$.

The lattice $\{s/\sigma_n:s\equiv n\pmod2\}$ has mesh $2/\sigma_n\to0$. Applying Lemma~\ref{lem:profiles} to $g=\varphi H_4$ in~\eqref{eq:Esymprofile} and using~\eqref{eq:symremainder} yields
\begin{align*}
 E_n(p,1/2)\sim\frac1{8\sqrt{2\pi}}\frac{|\kappa_4^{(1/2)}(n)|}{\sigma_n^4}.
\end{align*}
Substitution of~\eqref{eq:standardk4} gives all four rates, since $4\alpha-2=8p-6$. 
More precisely, the positive finite limits are
\begin{align*}
 nE_n(p,1/2)&\longrightarrow C(p)&& (0\le p<5/8),\\
 \frac{n}{\log n}E_n(5/8,1/2)&\longrightarrow\frac9{32\sqrt{2\pi}},\\
 n^{2-4\alpha}E_n(p,1/2)&\longrightarrow\frac{3D_\alpha(1-2\alpha)^2}{2\sqrt{2\pi}}
 &&(5/8<p<3/4),\\
 (\log n)^2E_n(3/4,1/2)&\longrightarrow\frac{9+\pi^2/2}{8\sqrt{2\pi}}.
\end{align*}

Next, suppose that $q\neq1/2$. We first consider the case $-1\le\alpha\le0$. Lemma~\ref{neg:lem:log}, \eqref{eq:independentcums}, and \eqref{eq:endpoint}, together with $v_n^2\asymp n$, give $\kappa_3^{(q)}(n)/(6v_n^3)=O_p(|\beta|n^{-3/2})=o_p(n^{-1})$, uniformly in $q$. Proposition~\ref{prop:fourth} and $v_n^2\sim n/(1-2\alpha)$ give
\begin{align}\label{eq:k4lowstandard}
 \frac{\kappa_4^{(q)}(n)}{24v_n^4}
 =-\frac{1+2\alpha^2}{12(1-4\alpha)}n^{-1}+o_p(n^{-1}),
 \qquad -1\le\alpha<1/4,
\end{align}
also uniformly in $q$. Since $\mathcal R_n(\alpha,\beta)=O(n^{-2})$, Theorem~\ref{thm:edgeworth} yields
\[
 E_n(p,q)=\frac{1+2\alpha^2}{12(1-4\alpha)}n^{-1}
 \sup_{\substack{s\in\Z\\ s\equiv n\pmod 2}}\left|\varphi\left(\frac{s-\mu_n}{v_n}\right)H_4\left(\frac{s-\mu_n}{v_n}\right)\right|+o_p(n^{-1}).
\]
The mesh $2/v_n$ tends to zero uniformly in $q$, and Lemma~\ref{lem:profiles} is uniform under translation. The supremum therefore tends uniformly to $3/\sqrt{2\pi}$, proving $nE_n(p,q)\to C(p)$ uniformly in $q$ in this range.

Now let $0<\alpha<1/2$ and fix $\beta\ne0$. Proposition~\ref{prop:third} and~\eqref{eq:meanvar} imply
\begin{align}\label{eq:a3asym}
 \frac{\kappa_3^{(q)}(n)}{6v_n^3}
 =a_3(\alpha,\beta)n^{3\alpha-3/2}+o(n^{3\alpha-3/2}).
\end{align}
For $0<\alpha<1/6$, this is $o(n^{-1})$, whereas~\eqref{eq:k4lowstandard} still holds. Moreover,
\[
 G_3(n)n^{-3}=n^{-2}=o(n^{-1}),\qquad
 n^{\alpha-5/2}G_2(n)=n^{\alpha-3/2}=o(n^{-1}).
\]
Thus the $H_4$-term in Theorem~\ref{thm:edgeworth} is the leading term, and Lemma~\ref{lem:profiles} again gives $nE_n(p,q)\to C(p)$.

At $\alpha=1/6$, the third-cumulant coefficient is $a_3(1/6,\beta)n^{-1}+o(n^{-1})$, while the fourth-cumulant coefficient is $-19n^{-1}/72+o(n^{-1})$, since
\[
 -\frac{1+2(1/6)^2}{12\{1-4(1/6)\}}=-\frac{19}{72}.
\]
The remainder terms are $n^{-2}\log n$ and $n^{-4/3}$, both $o(n^{-1})$. Theorem~\ref{thm:edgeworth} and Lemma~\ref{lem:profiles} therefore give
\begin{align*}
 nE_n(7/12,q)\longrightarrow
 \sup_{x\in\R}\left|\varphi(x)\left\{a_3(1/6,\beta)H_3(x)-\frac{19}{72}H_4(x)\right\}\right|>0.
\end{align*}
The limit is positive because the even part of the profile is the nonzero function $-19\varphi H_4/72$.

For $1/6<\alpha<1/2$, the $H_3$-term in Theorem~\ref{thm:edgeworth}, whose coefficient is given by~\eqref{eq:a3asym}, is the leading term. To check that the $H_4$-term is negligible, divide its respective orders $n^{-1}$, $n^{-1}\log n$, and $n^{4\alpha-2}$ by $n^{3\alpha-3/2}$. The ratios are
\[
 \begin{cases}
 n^{1/2-3\alpha},&1/6<\alpha<1/4,\\
 n^{-1/4}\log n,&\alpha=1/4,\\
 n^{\alpha-1/2},&1/4<\alpha<1/2,
 \end{cases}
\]
and all tend to zero. Likewise, $G_3(n)n^{-3}=n^{6\alpha-3}$ has ratio $n^{3\alpha-3/2}\to0$ to the leading order. The ratios of $n^{\alpha-5/2}G_2(n)$ to that order are, respectively,
\[
 n^{-2\alpha},\qquad n^{-1/2}\log n,\qquad n^{2\alpha-1},
\]
in the same three ranges. Thus $\mathcal R_n(\alpha,\beta)=o(n^{3\alpha-3/2})$.
The leading profile is $a_3(\alpha,\beta)\varphi H_3$ and the lattice mesh tends to zero, so
\begin{align*}
 n^{3/2-3\alpha}E_n(p,q)\longrightarrow |a_3(\alpha,\beta)|\sup_{x\in\R}|\varphi(x)H_3(x)|>0.
\end{align*}
Positivity follows from Proposition~\ref{prop:third}; the exponent is $3\alpha-3/2=6p-9/2$.

Finally, at $\alpha=1/2$, Proposition~\ref{prop:third} and $v_n^2\sim n\mathcal{H}_n$ give
\[
 \frac{\kappa_3^{(q)}(n)}{6v_n^3}
 =\frac{\gamma_{3,c}(\beta)}6\mathcal{H}_n^{-3/2}+o(\mathcal{H}_n^{-3/2}).
\]
Proposition~\ref{prop:fourth} gives $\kappa_4^{(q)}(n)/(24v_n^4)=O(\mathcal{H}_n^{-2})$, and the remainder in Theorem~\ref{thm:edgeworth} is $O(\mathcal{H}_n^{-5/2})$. Thus both the $H_4$-term and the remainder are $o(\mathcal{H}_n^{-3/2})$. Lemma~\ref{lem:profiles} applied to $\varphi H_3$ yields
\begin{align*}
 \mathcal{H}_n^{3/2}E_n(3/4,q)\longrightarrow \frac{|\gamma_{3,c}(\beta)|}{6}\sup_{x\in\R}|\varphi(x)H_3(x)|>0,
\end{align*}
where nonvanishing again follows from Proposition~\ref{prop:third}.
\end{proof}
The following remark explains the mechanisms behind the transition points in Theorem~\ref{thm:localrate}.
\begin{remark}\label{rem:local-transitions}
The different transition points arise from the
third- and fourth-cumulant corrections.
When $q=1/2$, the third cumulant vanishes,
and \eqref{eq:standardk4} shows that the
fourth-cumulant correction changes order at
$\alpha=1/4$ ($p=5/8$), with a logarithmic
factor at the boundary.

For $0<\alpha<1/2$ and fixed $q\ne1/2$, the third-cumulant
correction is of order $n^{3\alpha-3/2}$
by \eqref{eq:a3asym}. Its order matches
the fourth-cumulant correction of order
$n^{-1}$ when
\[
3\alpha-\frac32=-1
\quad\Longleftrightarrow\quad
\alpha=\frac16
\quad\Longleftrightarrow\quad
p=\frac7{12}.
\]
Beyond this point, the third-cumulant
correction dominates.
\end{remark}

We now prove Corollary~\ref{cor:usual-local}.
\begin{proof}[Proof of Corollary~\ref{cor:usual-local}]
For $0\le p<3/4$, put $\alpha=2p-1$, $c_p=1-2\alpha=3-4p$, and $w_n=\sqrt{n/c_p}$. Write $x=s/w_n$, $\delta_n=\mu_n/w_n$, and $r_n=v_n/w_n$, so that $(s-\mu_n)/v_n=(x-\delta_n)/r_n$.
Uniformly in $x\in\R$, as $\delta\to0$ and $r\to1$,
\begin{align}\label{eq:gaussian-perturbation}
 \frac1r\varphi\left(\frac{x-\delta}{r}\right)
 =\varphi(x)+\delta\varphi(x)H_1(x)+(r-1)\varphi(x)H_2(x)
 +O(\delta^2+(r-1)^2).
\end{align}
Indeed, the derivatives with respect to $\delta$ and $r$ at $(0,1)$ are $\varphi H_1$ and $\varphi H_2$, respectively. All second partial derivatives are uniformly bounded in $x$ when $(\delta,r)$ lies in a fixed neighborhood of $(0,1)$, so Taylor's theorem applies uniformly.
By definition,
\begin{align}\label{eq:usual-local-decomposition}
 \frac{\sqrt n}{2}\mathbb{P}_q(S_n=s)-\sqrt{c_p}\varphi(x)
 &=\frac{\sqrt{c_p}}{r_n}\left\{\frac{v_n}{2}\mathbb{P}_q(S_n=s)-\varphi\left(\frac{x-\delta_n}{r_n}\right)\right\}\\
 &\quad+\sqrt{c_p}\left\{\frac1{r_n}\varphi\left(\frac{x-\delta_n}{r_n}\right)-\varphi(x)\right\}.
\notag\end{align}
Also,~\eqref{eq:meanvar} gives
\begin{align}\label{eq:rn-square}
 r_n^2-1=-2\alpha\frac{h_2(n)}n-c_p\beta^2\frac{h_n^2}n.
\end{align}

First suppose $q=1/2$, so $\beta=\delta_n=0$.
For $-1<\alpha<0$,~\eqref{eq:rn-square} gives $r_n-1=o(n^{-1})$: the finite product $h_2(n)=0$ for $n\ge2$ when $\alpha=-1/2$, and otherwise $h_2(n)=O(n^{2\alpha})$.
At $\alpha=0$, $r_n=1$; for $\alpha=-1$, the moment formula $v_n^2=n/3$ for $n\ge3$, obtained in the proof of Proposition~\ref{prop:fourth}, gives $r_n=1$. By Theorem~\ref{thm:edgeworth} and Proposition~\ref{prop:fourth}, uniformly in $s\equiv n\pmod2$,
\begin{align*}
 \frac{v_n}{2}\mathbb{P}_{1/2}(S_n=s)-\varphi(s/v_n)
 =-\frac{1+2\alpha^2}{12(1-4\alpha)}n^{-1}
 \varphi(s/v_n)H_4(s/v_n)+o(n^{-1}).
\end{align*}
Consequently,~\eqref{eq:gaussian-perturbation} and~\eqref{eq:usual-local-decomposition} imply
\begin{align*}
 \frac{\sqrt n}{2}\mathbb{P}_{1/2}(S_n=s)-\sqrt{c_p}\varphi(x)
 =-\sqrt{c_p}\frac{1+2\alpha^2}{12(1-4\alpha)}n^{-1}\varphi(x)H_4(x)+o(n^{-1}).
\end{align*}
The replacement of $s/v_n=x/r_n$ by $x$ in the profile is uniform: every polynomial times $\varphi$ is uniformly continuous and vanishes at infinity, and $r_n\to1$. Since the $x$-lattice has mesh $2/w_n\to0$, Lemma~\ref{lem:profiles} gives
\begin{align*}
 n\widetilde E_n(p,1/2)\longrightarrow\sqrt{c_p}\,C(p)>0,
 \qquad 0\le p\le1/2.
\end{align*}

For $0<\alpha<1/2$, we have $h_2(n)\sim n^{2\alpha}/\Gamma(1+2\alpha)$, and hence
\begin{align}\label{eq:usual-local-r-sym-positive}
 r_n-1\sim-\frac{\alpha}{\Gamma(1+2\alpha)}n^{2\alpha-1}.
\end{align}
By Theorem~\ref{thm:edgeworth} and Proposition~\ref{prop:fourth}, the first term on the right-hand side of~\eqref{eq:usual-local-decomposition} has orders $n^{-1}$ for $0<\alpha<1/4$, $n^{-1}\log n$ at $\alpha=1/4$, and $n^{4\alpha-2}$ for $1/4<\alpha<1/2$, and is therefore $o(n^{2\alpha-1})$. It follows from~\eqref{eq:gaussian-perturbation},~\eqref{eq:usual-local-decomposition}, and~\eqref{eq:usual-local-r-sym-positive} that
\[
 \frac{\sqrt n}{2}\mathbb{P}_{1/2}(S_n=s)-\sqrt{c_p}\varphi(x)
 =-\frac{\sqrt{c_p}\alpha}{\Gamma(1+2\alpha)}n^{2\alpha-1}\varphi(x)H_2(x)+o(n^{2\alpha-1}),
\]
uniformly on the lattice. Lemma~\ref{lem:profiles} and the identity $\sup_x|\varphi(x)H_2(x)|=\varphi(0)=1/\sqrt{2\pi}$ give
\begin{align*}
 n^{1-2\alpha}\widetilde E_n(p,1/2)
 \longrightarrow\frac{\sqrt{c_p}\alpha}{\Gamma(1+2\alpha)\sqrt{2\pi}}>0,
 \qquad 0<\alpha<1/2.
\end{align*}
Substitution of $\alpha=2p-1$ gives the symmetric subcritical conclusion.

Now fix $q\ne1/2$. For $-1\le\alpha<-1/2$, both $\delta_n$ and $r_n-1$ are $o(n^{-1})$. Then, the second term in~\eqref{eq:usual-local-decomposition} is $o(n^{-1})$ uniformly on the lattice. By Theorem~\ref{thm:edgeworth}, Lemma~\ref{neg:lem:log}, \eqref{eq:endpoint},
and Proposition~\ref{prop:fourth}, uniformly in $s\equiv n\pmod2$,
\[
 \frac{v_n}{2}\mathbb{P}_q(S_n=s)-\varphi\left(\frac{s-\mu_n}{v_n}\right)
 =-\frac{1+2\alpha^2}{12(1-4\alpha)}n^{-1}
 \varphi\left(\frac{s-\mu_n}{v_n}\right)
 H_4\left(\frac{s-\mu_n}{v_n}\right)+o(n^{-1}).
\] Applying Lemma~\ref{lem:profiles} in the same way yields
\begin{align*}
 n\widetilde E_n(p,q)\longrightarrow\sqrt{c_p}\,C(p)>0,
 \qquad 0\le p<1/4.
\end{align*}

At $\alpha=-1/2$, we have $h_n\sim1/\sqrt{\pi n}$ and $w_n=\sqrt{n/2}$, so
\begin{align}\label{eq:usual-local-delta-boundary}
 \delta_n\sim\beta\sqrt{2/\pi}\,n^{-1}.
\end{align}
Since $h_2(n)=0$ for $n\ge2$, we have $r_n-1=O(n^{-2})$. Also $(1+2\alpha^2)/(12(1-4\alpha))=1/24$. Thus, by~\eqref{eq:gaussian-perturbation}, Theorem~\ref{thm:edgeworth},
Lemma~\ref{neg:lem:log}, and Proposition~\ref{prop:fourth},
\[
 \frac{\sqrt n}{2}\mathbb{P}_q(S_n=s)-\sqrt2\varphi(x)
 =\frac1n\varphi(x)\left\{\frac{2\beta}{\sqrt\pi}H_1(x)-\frac1{12\sqrt2}H_4(x)\right\}+o(n^{-1})
\]
uniformly. Applying Lemma~\ref{lem:profiles} in the same way yields
\begin{align*}
 n\widetilde E_n(1/4,q)\longrightarrow
 \sup_{x\in\R}\left|\varphi(x)\left\{\frac{2\beta}{\sqrt\pi}H_1(x)-\frac1{12\sqrt2}H_4(x)\right\}\right|>0.
\end{align*}
The even part ensures positivity.

For $-1/2<\alpha<1/2$, the relation $h_n\sim n^\alpha/\Gamma(1+\alpha)$ gives
\begin{align}\label{eq:usual-local-delta-asym}
 \delta_n\sim\frac{\beta\sqrt{c_p}}{\Gamma(1+\alpha)}n^{\alpha-1/2}.
\end{align}
Equation~\eqref{eq:rn-square} yields $r_n-1=O(n^{2\alpha-1})=o(n^{\alpha-1/2})$. It remains to verify that the first term on the right-hand side of~\eqref{eq:usual-local-decomposition} is $o(n^{\alpha-1/2})$. For $\alpha\le0$, Theorem~\ref{thm:edgeworth}, Lemma~\ref{neg:lem:log},
and \eqref{eq:independentcums} show that the first term in
\eqref{eq:usual-local-decomposition} is
$O(n^{-1})=o(n^{\alpha-1/2})$. For $\alpha>0$, Proposition~\ref{prop:third} and $v_n^2\sim n/c_p$
show that the $H_3$-term in Theorem~\ref{thm:edgeworth} is
$O(n^{3\alpha-3/2})=o(n^{\alpha-1/2})$.
Proposition~\ref{prop:fourth} and Theorem~\ref{thm:edgeworth} show that
the $H_4$-term and the remainder are also $o(n^{\alpha-1/2})$.
Thus~\eqref{eq:gaussian-perturbation},~\eqref{eq:usual-local-decomposition}, and~\eqref{eq:usual-local-delta-asym} give
\[
 \frac{\sqrt n}{2}\mathbb{P}_q(S_n=s)-\sqrt{c_p}\varphi(x)
 =\frac{\beta c_p}{\Gamma(1+\alpha)}n^{\alpha-1/2}\varphi(x)H_1(x)+o(n^{\alpha-1/2}),
\]
uniformly. Lemma~\ref{lem:profiles} and the identity $\sup_x|\varphi(x)H_1(x)|=\varphi(1)=1/\sqrt{2\pi e}$ give
\begin{align*}
 n^{1/2-\alpha}\widetilde E_n(p,q)
 \longrightarrow\frac{|\beta|c_p}{\Gamma(1+\alpha)\sqrt{2\pi e}}>0,
 \qquad -1/2<\alpha<1/2.
\end{align*}
Substituting $\alpha=2p-1$ gives the asserted exponent.

Finally, at $p=3/4$, put $w_n=\sqrt{n\log n}$, $x=s/w_n$, $\delta_n=\mu_n/w_n$, and $r_n=v_n/w_n$. Since $h_n\sim2\sqrt{n/\pi}$ and $\mathcal{H}_n=\log n+\gamma_{\rm E}+o(1)$, where $\gamma_{\rm E}$ is Euler's constant,
\begin{align}
 \delta_n&=\frac{2\beta}{\sqrt{\pi\log n}}+o((\log n)^{-1/2}),\label{eq:usual-local-critical-delta}\\
 r_n^2&=1+\frac{\gamma_{\rm E}-4\beta^2/\pi}{\log n}+o((\log n)^{-1}).\label{eq:usual-local-critical-r}
\end{align}
For $q=1/2$, we have $\delta_n=0$ and $r_n-1=\gamma_{\rm E}/(2\log n)+o((\log n)^{-1})$. By Theorem~\ref{thm:edgeworth} and Proposition~\ref{prop:fourth}, the first term in the critical analogue of~\eqref{eq:usual-local-decomposition} is $O(\mathcal H_n^{-2})$. Thus, the critical analogue of~\eqref{eq:usual-local-decomposition}
and~\eqref{eq:gaussian-perturbation} give, uniformly,
\[
 \frac{\sqrt{n\log n}}2\mathbb{P}_{1/2}(S_n=s)-\varphi(x)
 =\frac{\gamma_{\rm E}}{2\log n}\varphi(x)H_2(x)+o((\log n)^{-1}).
\]
Therefore, Lemma~\ref{lem:profiles} gives
\begin{align*}
 (\log n)\widetilde E_n(3/4,1/2)\longrightarrow\frac{\gamma_{\rm E}}{2\sqrt{2\pi}}>0.
\end{align*}
For fixed $q\ne1/2$, the mean correction in~\eqref{eq:usual-local-critical-delta} dominates the scale correction in~\eqref{eq:usual-local-critical-r} and the first term in the critical analogue of~\eqref{eq:usual-local-decomposition}, which is $O(\mathcal H_n^{-3/2})$ by Theorem~\ref{thm:edgeworth} and Propositions~\ref{prop:third} and~\ref{prop:fourth}. Hence, uniformly,
\[
 \frac{\sqrt{n\log n}}2\mathbb{P}_q(S_n=s)-\varphi(x)
 =\frac{2\beta}{\sqrt{\pi\log n}}\varphi(x)H_1(x)+o((\log n)^{-1/2}).
\]
Using Lemma~\ref{lem:profiles} and the maximum of $|\varphi H_1|$ once more gives
\begin{align*}
 \sqrt{\log n}\,\widetilde E_n(3/4,q)\longrightarrow\frac{\sqrt2\,|\beta|}{\pi\sqrt e}>0.
\end{align*}
The proof is complete.
\end{proof}
The following remark explains the mechanisms behind the transition points in Corollary~\ref{cor:usual-local}.
\begin{remark}\label{rem:usual-local-transitions}
As seen from \eqref{eq:gaussian-perturbation} and
\eqref{eq:usual-local-decomposition}, the transition points
in Corollary~\ref{cor:usual-local} arise from the competition
between the error under exact standardization ($E_n(p,q)$),
the mean correction ($\delta_n$), and the
standard-deviation correction ($r_n-1$).

When $q=1/2$, the mean correction vanishes.
By \eqref{eq:usual-local-r-sym-positive}, the
standard-deviation correction has order $n^{2\alpha-1}$
and dominates the error under exact standardization
for $0<\alpha<1/2$, leading to the transition at $p=1/2$.

For fixed $q\ne1/2$, Theorem~\ref{thm:localrate}
shows that the error under exact standardization
is of order $n^{-1}$ near $\alpha=-1/2$,
while \eqref{eq:usual-local-delta-asym} shows that
the mean correction has order $n^{\alpha-1/2}$
for $-1/2<\alpha<1/2$.
These orders meet at
\[
\alpha-\frac12=-1
\quad\Longleftrightarrow\quad
p=\frac14.
\]
At $\alpha=-1/2$, \eqref{eq:usual-local-delta-boundary}
confirms that the mean correction is of order $n^{-1}$.
Beyond this point, the mean correction dominates
both the error under exact standardization and
the standard-deviation correction.
\end{remark}

\section{Edgeworth expansions and optimal rates of CLT}\label{sec:kolmogorov}
We derive the optimal rates of convergence in the central limit theorem
under exact centering and variance normalization from the local
Edgeworth analysis, and compare their transitions with those
of local approximation.

We now prove Theorem~\ref{lem:kolmogorov-edgeworth}.

\begin{proof}[Proof of Theorem~\ref{lem:kolmogorov-edgeworth}]
As in the proof of Theorem~\ref{thm:edgeworth}, let $f_n(\xi)=\mathbb{E}_qe^{\mathrm{i}\xi Z_n}$.
It follows from \eqref{ineq-Gauss} that
\begin{align}
 &\int_{|\xi|\le T_n}\frac1{|\xi|}
 \left|f_n(\xi)-e^{-\xi^2/2}\left\{1+
 \frac{\kappa_3^{(q)}(n)}{6v_n^3}(\mathrm{i}\xi)^3+
 \frac{\kappa_4^{(q)}(n)}{24v_n^4}(\mathrm{i}\xi)^4\right\}\right|\,d\xi\notag\\
 &\qquad=O_p\bigl(\mathcal R_n(\alpha,\beta)\bigr).
 \label{eq:kol-small-integral}
\end{align}

By \eqref{eq:gaussiantail}, \eqref{eq:subcriticaltail}, \eqref{eq:criticalmiddletail}, and \eqref{eq:criticaloutertail}, together with $T_n\to\infty$, we obtain
\begin{align}
 &\int_{T_n<|\xi|\le\pi v_n/2}\frac1{|\xi|}
 \left|f_n(\xi)-e^{-\xi^2/2}\left\{1+
 \frac{\kappa_3^{(q)}(n)}{6v_n^3}(\mathrm{i}\xi)^3+
 \frac{\kappa_4^{(q)}(n)}{24v_n^4}(\mathrm{i}\xi)^4\right\}\right|\,d\xi\notag\\
 &\qquad=O_p\bigl(\mathcal R_n(\alpha,\beta)\bigr).
 \label{eq:kol-outer-integral}
\end{align}
Define the integrable signed density
\[
 g_n(x)=\varphi(x)\left\{1+\frac{\kappa_3^{(q)}(n)}{6v_n^3}H_3(x)
 +\frac{\kappa_4^{(q)}(n)}{24v_n^4}H_4(x)\right\}.
\]
Since
\[
\int_{\mathbb R}\varphi(x)H_j(x)\,dx=0,\qquad j\ge1,
\]
the function $g_n$ has total mass $1$. By repeated integration by parts,
\[
\int_{\mathbb R}e^{\mathrm{i}\xi x}\varphi(x)H_j(x)\,dx
=(\mathrm{i}\xi)^j e^{-\xi^2/2}.
\]
Hence the Fourier transform of $g_n$ is
\[
 e^{-\xi^2/2}\left\{1+\frac{\kappa_3^{(q)}(n)}{6v_n^3}(\mathrm{i}\xi)^3
 +\frac{\kappa_4^{(q)}(n)}{24v_n^4}(\mathrm{i}\xi)^4\right\}.
\]
Let
\[
G_n(x)\coloneqq \int_{-\infty}^x g_n(y)\,dy.
\]
Since $g_n$ has total mass $1$,
\[
G_n(-\infty)=0,\qquad G_n(+\infty)=1,
\]
and $G_n'(x)=g_n(x)$. Moreover, for every finite partition $x_0<\cdots<x_m$,
we have
\[
\sum_{k=1}^m
|G_n(x_k)-G_n(x_{k-1})|
\le
\int_{\mathbb R}|g_n(y)|\,dy.
\]
By \eqref{eq:anbounds} and \eqref{eq:bnbounds},
\[
\sup_{q\in[0,1]}\int_{\mathbb R}|g_n(y)|\,dy
\]
is uniformly bounded for all sufficiently large $n$, since
$\varphi|H_3|$ and $\varphi|H_4|$ are integrable. Hence $G_n$ has uniformly bounded total variation. Furthermore, since $\varphi H_3$ and $\varphi H_4$ are bounded, \eqref{eq:anbounds} and \eqref{eq:bnbounds} also imply that
\[
\sup_x|G_n'(x)|=\sup_x|g_n(x)|
\]
is uniformly bounded for all sufficiently large $n$ and $q\in[0,1]$.

Hence the bounded-variation form of Esseen's smoothing inequality %
\cite[p.~109]{Petrov1975}, with $b=1$, gives, for every $T>0$,
\[
\sup_x\left|\mathbb P_q(Z_n\le x)-G_n(x)\right|
\le
\int_{-T}^T
\left|
\frac{f_n(\xi)-\widehat g_n(\xi)}{\xi}
\right|\,d\xi
+
\frac{C_p}{T},
\]
where
\[
\widehat g_n(\xi)
=
e^{-\xi^2/2}\left\{
1+
\frac{\kappa_3^{(q)}(n)}{6v_n^3}(\mathrm{i}\xi)^3+
\frac{\kappa_4^{(q)}(n)}{24v_n^4}(\mathrm{i}\xi)^4
\right\}.
\]
Taking $T=\pi v_n/2$ and using \eqref{eq:kol-small-integral} and \eqref{eq:kol-outer-integral}, we obtain
\[
\sup_x\left|\mathbb P_q(Z_n\le x)-G_n(x)\right|
=
O_p\left(v_n^{-1}+\mathcal R_n(\alpha,\beta)\right).
\]
Finally, $(\varphi H_j)'=-\varphi H_{j+1}$, so
\[
 G_n(x)=\Phi(x)-\frac{\kappa_3^{(q)}(n)}{6v_n^3}\varphi(x)H_2(x)
 -\frac{\kappa_4^{(q)}(n)}{24v_n^4}\varphi(x)H_3(x).
\]
This proves \eqref{eq:kolmogorov-edgeworth}.
\end{proof}
We now prove Theorem~\ref{thm:kolmogorov}.
\begin{proof}[Proof of Theorem~\ref{thm:kolmogorov}]
We first establish the lattice lower bound. Chebyshev's inequality gives
\[
 \mathbb{P}_q(|S_n-\mu_n|\le2v_n)\ge\frac34.
\]
For sufficiently large $n$, the interval $[\mu_n-2v_n,\mu_n+2v_n]$ contains at most $4v_n$ points of the span-$2$ lattice. Therefore
\begin{align}\label{eq:maxatom}
 \max_{\substack{s\in\Z\\s\equiv n\pmod2}}\mathbb{P}_q(S_n=s)\ge\frac3{16v_n}.
\end{align}
A distribution function with an atom of mass $m$ is at Kolmogorov distance at least $m/2$ from every continuous distribution function. Hence
\begin{align*}
 D_n(p,q)\ge\frac3{32v_n}.
\end{align*}
Since $v_n\asymp\sqrt n$ when $p<3/4$, this implies
\begin{align}\label{eq:kol-lattice-n}
 D_n(p,q)\ge c_{p}n^{-1/2}.
\end{align}

Suppose first that $q=1/2$. The third cumulant vanishes, and
Proposition~\ref{prop:fourth} gives
\begin{align}\label{eq:kol-k4-orders}
 \frac{|\kappa_4^{(1/2)}(n)|}{v_n^4}\asymp
 \begin{cases}
 n^{-1},&0\le p<5/8,\\
 (\log n)n^{-1},&p=5/8,\\
 n^{8p-6},&5/8<p<3/4,\\
 \mathcal{H}_n^{-2},&p=3/4.
 \end{cases}
\end{align}
By \eqref{eq:Rscale},
\begin{align*}
 \mathcal R_n(\alpha,0)=
 \begin{cases}
 o(n^{-1/2}),&0\le p\le11/16,\\
 o(n^{8p-6}),&11/16<p<3/4,\\
 O(\mathcal{H}_n^{-3}),&p=3/4.
 \end{cases}
\end{align*}
For $p<3/4$, Theorem~\ref{lem:kolmogorov-edgeworth} therefore gives
\[
 D_n(p,1/2)=O_p\left(n^{-1/2}+\frac{|\kappa_4^{(1/2)}(n)|}{v_n^4}\right).
\]
The two orders $n^{8p-6}$ and $n^{-1/2}$ coincide exactly when $8p-6=-1/2$, or $p=11/16$. Thus this upper bound and \eqref{eq:kol-lattice-n} show that $D_n(p,1/2)\asymp n^{-1/2}$ for $0\le p\le11/16$.

For $11/16<p<3/4$, apply Theorem~\ref{lem:kolmogorov-edgeworth} at $x=1$. As $H_3(1)=-2$,
\[
 \mathbb{P}_{1/2}(Z_n\le1)-\Phi(1)
 =\frac{\kappa_4^{(1/2)}(n)}{12v_n^4}\varphi(1)
 +O_p\left(v_n^{-1}+\mathcal R_n(\alpha,0)\right).
\]
In this range,
\[
 v_n^{-1}+\mathcal R_n(\alpha,0)
 =o\left(\frac{|\kappa_4^{(1/2)}(n)|}{v_n^4}\right).
\]
Moreover, Proposition~\ref{prop:fourth} yields
\[
 \frac{\kappa_4^{(1/2)}(n)}{v_n^4}
 \sim-12D_\alpha(1-2\alpha)^2n^{4\alpha-2}.
\]
The coefficient is nonzero, so $D_n(p,1/2)\ge c_pn^{4\alpha-2}=c_pn^{8p-6}$. Combined with the upper bound, this proves the claimed order in this range.

At $p=3/4$, Proposition~\ref{prop:fourth} gives
\[
 \frac{\kappa_4^{(1/2)}(n)}{v_n^4}
 \sim
 -\left(9+\frac{\pi^2}{2}\right)\mathcal{H}_n^{-2}.
\]
Moreover, by \eqref{eq:meanvar} and \eqref{eq:Rscale},
\[
 v_n^{-1}
 =O((n\mathcal{H}_n)^{-1/2})
 =o(\mathcal{H}_n^{-2}),
 \qquad
 \mathcal R_n(1/2,0)
 =O(\mathcal{H}_n^{-3}).
\]
Applying Theorem~\ref{lem:kolmogorov-edgeworth} again at $x=1$
therefore yields
\[
 D_n(p,1/2)\asymp\mathcal{H}_n^{-2}.
\]
Altogether,
\begin{align*}
 D_n(p,1/2)\asymp
 \begin{cases}
 n^{-1/2},&0\le p\le11/16,\\
 n^{8p-6},&11/16<p<3/4,\\
 \mathcal{H}_n^{-2},&p=3/4.
 \end{cases}
\end{align*}

Now fix $q\ne1/2$. For $0<\alpha<1/2$, Proposition~\ref{prop:third} gives
\begin{align}\label{eq:kol-k3-asymp}
 \frac{\kappa_3^{(q)}(n)}{6v_n^3}\sim a_3(\alpha,\beta)n^{3\alpha-3/2},
 \qquad a_3(\alpha,\beta)\ne0.
\end{align}
For $p\le1/2$, we instead have $\kappa_3^{(q)}(n)/v_n^3=O_p(n^{-3/2})$ by Lemma \ref{neg:lem:log}, \eqref{eq:independentcums}, and \eqref{eq:endpoint}. Proposition \ref{prop:fourth} gives
\begin{align*}
 \frac{\kappa_4^{(q)}(n)}{v_n^4}=\begin{cases}
 O_p(n^{-1}),&-1\le\alpha<1/4,\\
 O_p((\log n)n^{-1}),&\alpha=1/4,\\
 O_p(n^{4\alpha-2}),&1/4<\alpha<1/2,\\
 O(\mathcal{H}_n^{-2}),&\alpha=1/2.
 \end{cases}
\end{align*}
Moreover, \eqref{eq:Rscale} gives
\begin{align*}
 \mathcal R_n(\alpha,\beta)=
 \begin{cases}
 o(n^{-1/2}),&0\le p\le2/3,\\
 o(n^{6p-9/2}),&2/3<p<3/4.
 \end{cases}
\end{align*}
For $0\le p\le2/3$, both standardized cumulant corrections are $O_p(n^{-1/2})$.
Thus Theorem~\ref{lem:kolmogorov-edgeworth} gives $D_n(p,q)=O_{p,q}(n^{-1/2})$ for $0\le p\le2/3$, and \eqref{eq:kol-lattice-n} supplies the matching lower bound.

For $2/3<p<3/4$, evaluate the same theorem at $x=0$.
In this range, the fourth-cumulant correction is $O_p(n^{4\alpha-2})=o(n^{3\alpha-3/2})$, and $v_n^{-1}+\mathcal R_n(\alpha,\beta)=o(n^{3\alpha-3/2})$.
Since $H_2(0)=-1$ and $H_3(0)=0$,
\[
 \mathbb{P}_q(Z_n\le0)-\frac12=\frac{\kappa_3^{(q)}(n)}{6v_n^3}\varphi(0)
 +O_{p,q}\left(v_n^{-1}+\mathcal R_n(\alpha,\beta)\right).
\]
Equation~\eqref{eq:kol-k3-asymp} and its nonzero coefficient yield
$D_n(p,q)\ge c_{p,q}n^{3\alpha-3/2}=c_{p,q}n^{6p-9/2}$. The upper bound from Theorem~\ref{lem:kolmogorov-edgeworth} has the same order.

Finally, at $p=3/4$, Proposition~\ref{prop:third} gives
\[
 \frac{\kappa_3^{(q)}(n)}{6v_n^3}
 \sim
 \frac{\gamma_{3,c}(\beta)}6\mathcal{H}_n^{-3/2},
\]
where $\gamma_{3,c}(\beta)\ne0$ since $q\ne1/2$. Moreover,
\[
 v_n^{-1}=o(\mathcal{H}_n^{-3/2}),\qquad
 \frac{\kappa_4^{(q)}(n)}{v_n^4}=O(\mathcal{H}_n^{-2}),\qquad
 \mathcal R_n(1/2,\beta)=O(\mathcal{H}_n^{-5/2}).
\]
Thus the upper estimate and evaluation at $x=0$ give
\[
 D_n(p,q)\asymp\mathcal{H}_n^{-3/2}.
\]
Altogether,
\begin{align*}
 D_n(p,q)\asymp
 \begin{cases}
 n^{-1/2},&0\le p\le2/3,\\
 n^{6p-9/2},&2/3<p<3/4,\\
 \mathcal{H}_n^{-3/2},&p=3/4.
 \end{cases}
\end{align*}
Since $\mathcal{H}_n\sim\log n$, these are precisely the rates in
Theorem~\ref{thm:kolmogorov}.
\end{proof}
The following remark explains the mechanisms behind
the transition points in Theorem~\ref{thm:kolmogorov}.

\begin{remark}\label{rem:kolmogorov-transitions}
For $p<3/4$, the transition points arise from the competition
between the lattice effect, reflected in the lower bound
\eqref{eq:kol-lattice-n}, and the third- and fourth-cumulant
corrections.

When $q=1/2$, the third cumulant vanishes,
and \eqref{eq:kol-k4-orders} shows that the
fourth-cumulant correction has order
$n^{8p-6}$ for $5/8<p<3/4$.
Its order matches the lattice error when
\[
8p-6=-\frac12
\quad\Longleftrightarrow\quad
p=\frac{11}{16}.
\]
Beyond this point, the fourth-cumulant correction dominates.

For fixed $q\ne1/2$, \eqref{eq:kol-k3-asymp}
shows that the third-cumulant correction
has order $n^{6p-9/2}$ for $1/2<p<3/4$.
Its order matches the lattice error when
\[
6p-\frac92=-\frac12
\quad\Longleftrightarrow\quad
p=\frac23.
\]
Beyond this point, the third-cumulant correction dominates.
\end{remark}

We now prove Corollary~\ref{cor:usual-kolmogorov}.
\begin{proof}[Proof of Corollary~\ref{cor:usual-kolmogorov}]
First let $0\le p<3/4$, and put
\[
 \alpha=2p-1,\qquad c_p=1-2\alpha=3-4p,\qquad w_n=\sqrt{n/c_p}.
\]
Write

\[
 \mathscr F_n(x)=\mathbb{P}_q\left(\frac{S_n-\mu_n}{v_n}\le x\right),\qquad
 D_n(p,q)=\sup_x|\mathscr F_n(x)-\Phi(x)|,
\]
and set $\delta_n=\mu_n/w_n$ and $r_n=v_n/w_n$. A change of variable in the definition of $\widetilde D_n$ gives
\[
 \widetilde D_n(p,q)=\sup_x\left|\mathscr F_n\left(\frac{x-\delta_n}{r_n}\right)-\Phi(x)\right|.
\]
It follows that
\begin{align}\label{eq:kolm-comparison}
 \left|\widetilde D_n(p,q)-\sup_x\left|\Phi\left(\frac{x-\delta_n}{r_n}\right)-\Phi(x)\right|\right|
 \le D_n(p,q).
\end{align}
As $\delta\to0$ and $r\to1$, Taylor's formula at $(\delta,r)=(0,1)$ gives
\begin{align}\label{eq:cdf-gaussian-perturbation}
 \Phi\left(\frac{x-\delta}{r}\right)-\Phi(x)
 =-\delta\varphi(x)-(r-1)x\varphi(x)
 +O\bigl((|\delta|+|r-1|)^2\bigr),
\end{align}
uniformly in $x\in\R$. The remainder is uniform in $x$, since the second derivatives are bounded uniformly in $x$ when $r$ stays in a fixed neighborhood of $1$. Moreover, \eqref{eq:meanvar} implies
\begin{align}\label{eq:usual-kolm-r}
 r_n^2-1=-2\alpha\frac{h_2(n)}n-c_p\beta^2\frac{h_n^2}n.
\end{align}

Suppose first that $q=1/2$, so $\beta=\delta_n=0$. For $1/2<p<3/4$, the asymptotic formula
\begin{align}
 h_2(n)\sim\frac{n^{2\alpha}}{\Gamma(1+2\alpha)}
\end{align}
and \eqref{eq:usual-kolm-r} yield
\begin{align}\label{eq:usual-kolm-r-sym}
 r_n-1\sim
 -\frac{\alpha}{\Gamma(1+2\alpha)}n^{2\alpha-1}.
\end{align} For $1/2<p\le5/8$, equivalently $0<\alpha\le1/4$,
$r_n-1=O(n^{2\alpha-1})=O(n^{-1/2})$. 
For $0\le p\le1/2$, \eqref{eq:usual-kolm-r} gives $r_n-1=O_p(n^{-1})$. 
Since, by Theorem \ref{thm:kolmogorov}, $D_n(p,1/2)=O(n^{-1/2})$ for $0\le p\le5/8$, \eqref{eq:kolm-comparison}--\eqref{eq:cdf-gaussian-perturbation} give the same upper bound for $\widetilde D_n$.

For every discrete random variable $Y$ and continuous distribution function $G$,
\[
 \sup_x|\mathbb{P}(Y\le x)-G(x)|\ge\frac12\sup_y\mathbb{P}(Y=y).
\]
Corollary \ref{cor:usual-local} gives $\sup_{s\equiv n\!\pmod2}\mathbb{P}_{1/2}(S_n=s)\asymp n^{-1/2}$ below criticality. Thus
\begin{align*}
 \widetilde D_n(p,1/2)\asymp n^{-1/2},\qquad 0\le p\le5/8.
\end{align*}

If $p>5/8$, or $\alpha>1/4$, the exact-normalization bounds give $D_n(p,1/2)=o(n^{2\alpha-1})$ by Theorem \ref{thm:kolmogorov}: for $\alpha\le3/8$ this follows from $O(n^{-1/2})$, and for $\alpha>3/8$ from $O(n^{4\alpha-2})$. Consequently, \eqref{eq:kolm-comparison}, \eqref{eq:cdf-gaussian-perturbation}, and \eqref{eq:usual-kolm-r-sym} imply
\[
 \widetilde D_n(p,1/2)\sim\frac{\alpha}{\Gamma(1+2\alpha)}n^{2\alpha-1}\sup_x|x\varphi(x)|.
\]
Substitution of $\alpha=2p-1$ gives the symmetric subcritical conclusion.

Now fix $q\ne1/2$. For $-1<\alpha<1/2$, $h_n\sim n^\alpha/\Gamma(1+\alpha)$ gives
\begin{align}\label{eq:usual-kolm-delta}
 \delta_n\sim\frac{\beta\sqrt{c_p}}{\Gamma(1+\alpha)}n^{\alpha-1/2}.
\end{align}
At the endpoint $\alpha=-1$, we instead have $h_n=0$ and hence $\delta_n=0$ for $n\ge2$. Equation~\eqref{eq:usual-kolm-r} also gives
\begin{align*}
 r_n-1=O(n^{2\alpha-1}).
\end{align*}
Here the finite-product definition covers exceptional values at which
$h_2(n)$ eventually vanishes. In particular, for $0\le p\le1/2$,
\[
 \delta_n=O(n^{-1/2}),\qquad
 r_n-1=O(n^{-1}),
\]
while Theorem~\ref{thm:kolmogorov} gives
\[
 D_n(p,q)=O(n^{-1/2}).
\]
The comparison yields
\[
 \widetilde D_n(p,q)=O(n^{-1/2}).
\]
On the other hand, Corollary~\ref{cor:usual-local} and the atom bound give
\[
 \widetilde D_n(p,q)\ge c_p'n^{-1/2}.
\]
Therefore
\[
 \widetilde D_n(p,q)\asymp n^{-1/2},
 \qquad 0\le p\le1/2.
\]

For $1/2<p<3/4$, we have
\[
 r_n-1=o(\delta_n).
\]
Moreover, Theorem~\ref{thm:kolmogorov} gives
\[
 D_n(p,q)=o(n^{\alpha-1/2}).
\]
Hence \eqref{eq:kolm-comparison},
\eqref{eq:cdf-gaussian-perturbation}, and
\eqref{eq:usual-kolm-delta} yield
\[
 \widetilde D_n(p,q)\asymp n^{\alpha-1/2},
 \qquad 1/2<p<3/4.
\]
Substituting $\alpha=2p-1$ and $c_p=3-4p$ gives the asymmetric subcritical conclusion.

At $p=3/4$, retain $\mathscr F_n$ and set
\[
 \delta_n=\frac{\mu_n}{\sqrt{n\log n}},\qquad
 r_n=\frac{v_n}{\sqrt{n\log n}}.
\]
$\widetilde D_n(3/4,q)=\sup_x|\mathscr F_n((x-\delta_n)/r_n)-\Phi(x)|$, so the same comparison applies. Since $h_n\sim2\sqrt{n/\pi}$ and $\mathcal{H}_n=\log n+\gamma_{\rm E}+o(1)$,
\begin{align}
 \delta_n&=\frac{2\beta}{\sqrt{\pi\log n}}+o((\log n)^{-1/2}),
 \label{eq:usual-kolm-critical-delta}\\
 r_n^2&=1+\frac{\gamma_{\rm E}-4\beta^2/\pi}{\log n}+o((\log n)^{-1}).
\notag\end{align}
For $q=1/2$, $\delta_n=0$ and $r_n-1=\gamma_{\rm E}/(2\log n)+o((\log n)^{-1})$, whereas
$D_n(3/4,1/2)=O((\log n)^{-2})=o((\log n)^{-1})$ by Theorem \ref{thm:kolmogorov}. Therefore
\begin{align*}
 \widetilde D_n(3/4,1/2)\sim
 \frac{\gamma_{\rm E}}{2\log n}\sup_x|x\varphi(x)|
 =\frac{\gamma_{\rm E}}{2\sqrt{2\pi e}}(\log n)^{-1}.
\end{align*}
For fixed $q\ne1/2$, \eqref{eq:usual-kolm-critical-delta} gives $|\delta_n|\asymp(\log n)^{-1/2}$, while $r_n-1=O((\log n)^{-1})$ and $D_n(3/4,q)=O((\log n)^{-3/2})=o((\log n)^{-1/2})$ by Theorem \ref{thm:kolmogorov}. Hence
\begin{align*}
 \widetilde D_n(3/4,q)\sim\frac{2|\beta|}{\sqrt{\pi\log n}}\sup_x\varphi(x)
 =\frac{\sqrt2\,|\beta|}{\pi}(\log n)^{-1/2}.
\end{align*}
These prove the critical conclusions and complete the proof.
\end{proof}
The following remark explains the mechanisms behind
the transition points in Corollary~\ref{cor:usual-kolmogorov}.

\begin{remark}\label{rem:usual-kolmogorov-transitions}
As seen from \eqref{eq:kolm-comparison} and
\eqref{eq:cdf-gaussian-perturbation}, the transition points
in Corollary~\ref{cor:usual-kolmogorov} arise from the competition
between the error under exact standardization ($D_n(p,q)$),
the mean correction ($\delta_n$), and the
standard-deviation correction ($r_n-1$).

When $q=1/2$, the mean correction vanishes.
By \eqref{eq:usual-kolm-r-sym}, the
standard-deviation correction has order
$n^{4p-3}$ for $1/2<p<3/4$.
On the other hand, Theorem~\ref{thm:kolmogorov}
shows that the error under exact standardization
is of order $n^{-1/2}$ near $p=5/8$.
These orders meet when
\[
4p-3=-\frac12
\quad\Longleftrightarrow\quad
p=\frac58.
\]
Beyond this point, the standard-deviation correction
dominates the error under exact standardization.

For fixed $q\ne1/2$, \eqref{eq:usual-kolm-delta}
shows that the mean correction has order
$n^{2p-3/2}$ for $1/2\le p<3/4$.
By Theorem~\ref{thm:kolmogorov}, the error under
exact standardization is of order $n^{-1/2}$
near $p=1/2$.
These orders meet when
\[
2p-\frac32=-\frac12
\quad\Longleftrightarrow\quad
p=\frac12.
\]
Beyond this point, the mean correction dominates
both the error under exact standardization and
the standard-deviation correction.
\end{remark}

\appendix
\section{Exact cumulants at $p=3/4$}\label{app:criticalmoments}
We give the harmonic-sum calculations used in Propositions~\ref{prop:third} and~\ref{prop:fourth}.
At $\alpha=1/2$, where $h_n>0$, define
\[
 u_n\coloneqq \frac{\mathbb{E}_1[S_n^3]}{h_n}.
\]
Since every odd moment is $\beta$ times its value at $q=1$,
\[
 \mathbb{E}_q[S_n^3]=\beta h_nu_n
\]
for every $q\in[0,1]$, including $\beta=0$.
Applying recurrence~\eqref{eq:m3-rec} with $q=1$ and using
$h_{n+1}/h_n=(2n+1)/(2n)$, we obtain
\[
 u_{n+1}
 =
 \frac{2n+3}{2n+1}u_n
 +
 \frac{6n+1}{2n+1},
 \qquad
 u_1=1.
\]
Solving this recurrence gives
\[
 \frac{u_n}{2n+1}
 =
 3\mathcal{H}_{2n}
 -\frac32\mathcal{H}_n
 -4
 +\frac4{2n+1}.
\]
Since $\mathbb{E}_q[S_n^2]=n\mathcal{H}_n$, it follows that
\[
 \kappa_3^{(q)}(n)
 =
 \mu_n
 \left\{
 3(2n+1)(\mathcal{H}_{2n}-\mathcal{H}_n)
 +\frac32\mathcal{H}_n
 -8n
 \right\}
 +2\mu_n^3,
\]
which proves \eqref{eq:k3exactcritical} without dividing by the
possibly vanishing parameter $\beta$.

For the fourth cumulant, recurrence~\eqref{eq:m4-rec}, together with
$\mathbb{E}_{1/2}[S_n^2]=n\mathcal{H}_n$, gives
\[
 \kappa_4^{(1/2)}(n+1)
 =
 \left(1+\frac2n\right)\kappa_4^{(1/2)}(n)
 -3\mathcal{H}_n^2
 -4\mathcal{H}_n
 -2,
 \qquad
 \kappa_4^{(1/2)}(1)=-2.
\]
Solving this recurrence yields
\begin{align}\label{eq:crit-kappa4}
 \kappa_4^{(1/2)}(n)
 =
 -n\left\{
 3(n+1)\mathcal{H}_n^{(2)}
 -3\mathcal{H}_n^2
 -10\mathcal{H}_n
 +9n
 \right\},
\end{align}
which is \eqref{eq:k4critical}.

\section{Previous local limit estimates}\label{app:local}
We briefly discuss a common issue in two previously claimed local limit estimates for the elephant random walk. 
In both cases, a distribution-function estimate is converted into a relative estimate for point probabilities by differencing. The resulting additive error must then be divided by the Gaussian point-mass approximation, and this introduces the reciprocal of the corresponding lattice scale. 
The cited proofs omit this factor.

\subsection{The local estimate of Fan--Hu--Ma}

Following the notation of \cite{FanHuMa2021}, set
\[
 a_n=\prod_{j=1}^{n-1}(1+\alpha/j)^{-1},
 \qquad
 V_n=\sum_{j=1}^n a_j^2.
\]
In the passage from (46) to (47) of \cite{FanHuMa2021},
the factor $\sqrt{V_n}/a_n$ is missing.
This calculation treats the case $0<p<1/2$, the only case
for which the proof is written out there.
Since $\sqrt{V_n}/a_n\asymp\sqrt n$, this factor cannot be
absorbed into a constant independent of $n$.
Consequently, the passage from (46) to (47) does not justify
the claimed relative error bound.

\subsection{The local estimates of Huang--Fan--Xiang}
\label{app:local-hfx}

The same issue occurs in the proof of Theorem~2.7(i) of
\cite{HuangFanXiang2026}.
In the passage from (5.3) to (5.4), the factor $\tau_n/a_n$
is missing, where
\[
 \tau_n=\sqrt{V_n+n a_n^2\sigma^2}.
\]
Here we write $V_n$ for $v_n$ in \cite{HuangFanXiang2026}.
For $Z_i\equiv1$, which reduces the model to the ordinary ERW,
we have $\sigma^2=0$, and hence
\[
 \tau_n=\sqrt{V_n},
 \qquad
 \frac{\tau_n}{a_n}\asymp\sqrt n.
\]
Thus, this factor cannot be absorbed into a constant
independent of $n$, and the passage from (5.3) to (5.4)
does not justify the claimed relative error bound.
\section{A difference-equation lemma}\label{appendix:techlem}
\begin{lemma}\label{lem:difference}
Suppose that $x_{n+1}=(1+c/n)x_n+d_n$ for all sufficiently large $n$, where $c<1$ and $d_n\to d$. Then $x_n/n\to d/(1-c)$.
\end{lemma}
\begin{proof}
Choose $N$ so that $1+c/n>0$ for $n\ge N$, and put
$g_n=\prod_{j=N}^{n-1}(1+c/j)$.
Since $g_n\asymp n^c$ and $c<1$, we have $n/g_n\to\infty$.
Moreover,
\[
 \frac{x_{n+1}}{g_{n+1}}-\frac{x_n}{g_n}
 =\frac{d_n}{g_{n+1}},
 \qquad
 \frac{n+1}{g_{n+1}}-\frac{n}{g_n}
 =\frac{1-c}{g_{n+1}}.
\]
Hence, the Stolz--Ces\`aro theorem yields
\[
 \frac{x_n}{n}
 =\frac{x_n/g_n}{n/g_n}
 \longrightarrow\frac{d}{1-c}.
\]
\end{proof}

\section*{Acknowledgments}
The authors would like to thank Professor Kouji Yano for carefully reading the manuscript and for valuable comments. 
G.T. was supported by JST SPRING, Grant Number JPMJSP2138.

\section*{Declaration on the use of generative AI}
During the preparation of this manuscript, the authors used the generative AI tools ChatGPT (OpenAI) and Claude (Anthropic) to assist with language editing and improving the clarity of the presentation, and to develop and check mathematical arguments. All mathematical arguments, statements, and conclusions were independently verified by the authors, who take full responsibility for the content of the manuscript.


\end{document}